\documentclass[12pt,english]{article}
\usepackage[T1]{fontenc}
\usepackage[utf8]{inputenc}
\usepackage{geometry}
\usepackage{babel}
\usepackage{units}
\usepackage{amsmath}
\usepackage{amsthm}
\usepackage{amssymb}
\usepackage{tikz}
\usepackage[unicode=true,pdfusetitle,
 bookmarks=true,bookmarksnumbered=false,bookmarksopen=false,
 breaklinks=false,pdfborder={0 0 1},backref=false,colorlinks=false]
 {hyperref}

\makeatletter

\newcommand*\samethanks[1][\value{footnote}]{\footnotemark[#1]}
\newcommand{\cP}{ \mathcal{P} }

\renewcommand{\Pr}{ \mathbb{P} }

\theoremstyle{plain}
\newtheorem{thm}{\protect\theoremname}
\theoremstyle{remark}
\newtheorem{claim}[thm]{\protect\claimname}
\theoremstyle{plain}
\newtheorem{lem}[thm]{\protect\lemmaname}
\theoremstyle{remark}
\newtheorem*{claim*}{\protect\claimname}
\theoremstyle{definition}
\newtheorem*{problem*}{\protect\problemname}
\newtheorem*{thm*}{\protect\theoremname}
\theoremstyle{definition}
\newtheorem{definition}{\protect\definitionname}
\theoremstyle{plain}
\newtheorem{cor}[thm]{\protect\corollaryname}
\theoremstyle{plain}

\theoremstyle{plain}
\newtheorem{proposition}[thm]{\protect\propositionname}
\providecommand{\claimname}{Claim}
\providecommand{\conjecturename}{Conjecture}
\providecommand{\corollaryname}{Corollary}
\providecommand{\lemmaname}{Lemma}
\providecommand{\problemname}{Problem}
\providecommand{\theoremname}{Theorem}
\providecommand{\propositionname}{Proposition}
\providecommand{\definitionname}{Definition}

\title{Logical isomorphism and logical limit laws for small permutation classes and regular languages}
\author{Klas Markstr\"om\thanks{Institutionen f\"or matematik och matematisk statistik, Ume\aa\text{ } Universitet, Sweden.\\	Emails: $\{$\texttt{klas.markstrom}, \texttt{eero.raty}$\}$\texttt{@umu.se}. E.R. is supported in part by a postdoctoral grant 213-0204 from the Olle Engkvist Foundation, by a postdoctoral grant from the Osk. Huttunen
Foundation, Institute for Basic Science (IBS-R029-C4), and by Swedish Research Council grant VR 2025-03758. Part of the writing was conducted while second author was a senior researcher in ECOPRO research group at the Institute for Basic Science.}  \and Eero R\"aty\samethanks[1] }

\date{}

\begin{document}

\maketitle

\begin{abstract}
    In this paper we show  that for each small permutation class $\mathcal{F}$ there exists a regular language $L_\mathcal{F}$ such that the monadic second-order TOTO logic on $\mathcal{F}$  is isomorphic  to classical monadic second-order logic on $L_\mathcal{F}$.

    Thanks to the fact that regular languages are closed under monadic second-order restrictions this shows that many subsets of $\mathcal{F}$ have a well-behaved structure.  In particular, this allows us to show that random permutations from any small permutation class satisfies convergence laws for both first-order logic and monadic second-order logic. In the latter case, the law is modular.

    Our proofs introduce a way of using the automaton description of regular languages to prove properties of generating functions which may be of independent interest. Potential applications include first-order limit laws for random words in some regular languages. 
\end{abstract}

\section{Introduction}
This paper has several objectives. We first demonstrate how the logical characterization of regular languages can be utilized in enumerative and probabilistic combinatorics. Next, we prove that with regard to monadic second-order logic, small permutation classes, among several other classes of permutations, are isomorphic to regular languages. Finally, this is used to prove  logical limit laws for random permutations in all such classes. 

One of the classic techniques in enumerative combinatorics is the use of bijective proofs. Here a bijection $\psi$  from some set under study $A$ to some well understood set $B$ is found and as a consequence we know that the set $A$ has the same size as $B$.  Two well known examples are given by pattern avoiding permutations, where both $\mathrm{Av}(123)$, the set of permutations avoiding the pattern $123$,  and $\mathrm{Av}(132)$, the set of those avoiding $132$, have bijections to the set of Dyck paths \cite{Krat}. As a consequence, we learn that the number of permutations avoiding these two patterns are the same, equal to the number of Dyck paths, and we get a bijection between $\mathrm{Av}(123)$ and $\mathrm{Av}(132)$. However, this bijection is far from unique, so much so that entire papers \cite{Claesson,Bloom} have been devoted to collecting such bijections and investigating which properties of a permutation they preserve.  It is known, see e.g.\ \cite{albertrationality}, that the two sets $\mathrm{Av}(123)$ and $\mathrm{Av}(132)$ have very different structure with respect to many properties, and \cite{MinerPak} argues that this is what lies behind the long list of distinct bijections.  

Here our aim  is to study something stronger than simply a bijection between two sets $A$ and $B$.  We are going to assume that both sets have a logic $\mathbb{L}$ defined on them, e.g.\ a version of first-order logic,  and then demand that each $\mathbb{L}$-sentence over $A$ has a $\mathbb{L}$-sentence over $B$ whose truth is preserved by our bijection, and similarly each $\mathbb{L}$-sentence over $B$ has a $\mathbb{L}$-sentence over $A$ whose truth is preserved by our bijection. When such a bijection exists we will say that $A$ and $B$ are $\mathbb{L}$-isomorphic. Given such an isomorphism,  we learn not only that the two sets have the same size, but also that the number of objects in $A$ which satisfy a property which can be formulated in $\mathbb{L}$  is the same as the number of objects in $B$ satisfying another property specifiable in $\mathbb{L}$. 

The concrete setting which we will focus on is that $A$ is a permutation class and $B$ is the set of words from a regular language. Our logic $\mathbb{L}$ is going to be either first-order logic or monadic second-order logic, as defined for $A$ and $B$. This setting is relevant for several reasons.  

First, there has recently been an interest in the existence of limit laws for first-order properties of random permutations from different permutation classes.  In~\cite{braunfeld2021logical},  it was shown that the class of layered permutations satisfies such a limit law, and in~\cite{albert2024logical} and~\cite{ozdemir2023first} limit laws for random permutations from $\mathrm{Av}(132)$ and $\mathrm{Av}(123)$  respectively were proven. As we shall show, all \emph{small} permutation classes satisfy a limit law, thereby settling a conjecture from~\cite{braunfeld2021logical} for these classes.  

Second, monadic and first-order logic have a close connection to regular languages which has consequences for any set $A$.  In particular, the B\"{u}chi-Elgot-Trakhtenbrot theorem, see e.g.\ \cite{Wolfgang}, shows that a language is regular if and only if it can be defined using monadic second-order logic. As a consequence, the set of words from a regular language which satisfy a first-order or monadic second-order property is also a regular language. Using generating functions for words in the language makes it possible to show that random objects from the set $A$ satisfy a limit law. Here the type of limit law can both be simple, showing that probabilities converge as the relevant size $n$ goes to infinity,  or modular, showing that the same is true when $n$ is restricted to a fixed residue class, modulo some number $p$. As a related work, we also remark that Braunfeld~\cite{braunfeld2025decidability} used monadic second-order logic to prove that the basis and generating function of certain geometric permutation classes are computable directly from certain simple parameters defining these classes. We revisit this result properly in Section~\ref{subsetion: albert construction} after introducing necessary background. 

Before moving on to formal definitions, we will give a brief overview of logic on combinatorial objects and permutation classes.

\subsection{Logics on combinatorial objects}
A key difference between the study of logical limit laws for random permutations and for random graphs is that, unlike in the graph case, there are at least two genuinely different and very natural logics for permutations, each of which has its own advantages. The first logic appeared in~\cite{foy1990probabilities} and is now known as the Theory of Two Orders (abbreviated as $\mathrm{TOTO}$). This logic concerns a set equipped with two independent linear orders. Informally speaking, one should view permutations in $\mathrm{TOTO}$ as a set of points in the plane with no two points placed on the same horizontal or vertical line, and the two orders can be chosen to be the usual orders comparing the $x$- and $y$-coordinates of the respective points. $\mathrm{TOTO}$ is suitable for e.g  problems about pattern-avoidance in permutations, and is the logic we shall consider in this paper; see Section~\ref{subsection: permutations} for a more detailed introduction to $\mathrm{TOTO}$. The second logic, which we will not define here, is the  Theory of One Bijection ($\mathrm{TOOB}$), which is more suitable when working with problems that concern the cycle-structure of permutations. For an excellent account of both of these, see e.g.\ the recent paper of Albert, Bouvel and F\'eray~\cite{albert2020two}.  We will also make use of monadic second-order logic. This can briefly be viewed as extending first-order logic  with the ability to quantify not only over elements  but also over sets of elements.

\subsubsection{Logical limit laws}
For uniform random permutations, Foy and Woods~\cite{foy1990probabilities} showed that in general $\mathrm{TOTO}$ formulae do not satisfy a limit law, meaning that as $n$ grows the probability that a given $\mathrm{TOTO}$ sentence is true  does not  always converge. In 2022 Braunfeld and Kukla~\cite{braunfeld2021logical} showed, among other results, that the class of so-called `layered' permutations  satisfies a first-order limit law in $\mathrm{TOTO}$. They also made the strong conjecture that every proper permutation class satisfies a first-order limit law in $\mathrm{TOTO}$. 

More recently, Albert, Bouvel, F\'eray and  Noy~\cite{albert2024logical} showed that a uniformly chosen random permutation from the class $\mathrm{Av}(231)$ of $231$-avoiding permutations satisfies a first-order logical limit law in $\mathrm{TOTO}$. Furthermore, they showed that the set $\mathrm{Lim}(\mathrm{Av}(231))$ of limits of probabilities that a fixed $\mathrm{TOTO}$ sentence is satisfied is dense in $[0,1]$.  Shortly thereafter, \"Ozdemir~\cite{ozdemir2023first} obtained a first-order logical limit law in $\mathrm{TOTO}$ for $\mathrm{Av}(321)$; using symmetries, these results combined imply that all six of the permutation classes $\mathrm{Av}(\sigma)$, $\sigma \in \mathfrak{S}_{3}$, satisfy a first-order convergence law.   

Throughout this paper, we only consider the uniform measures in $\cP_n$, where $\cP$ is a given permutation class. For convenience, we denote random permutations by bold symbols, and when the permutation sub-class $\cP$ is clear from the context, we write $\boldsymbol{\sigma}_n$ for the permutation chosen uniformly from $\cP_n$, and do not explicitly mention that the distribution chosen is uniform. 

\subsection{Overview of permutation classes}

Permutation classes have been extensively studied since the pioneering work of Knuth~\cite[Section 2.2.1]{knuth1968art} in the late 1960s. A permutation class $\cP$ is a set of finite permutations $\cP$ which is closed under taking subpatterns. One of the most fundamental questions in the area is to ask what can one say about the growth-rate of $\vert \cP_n\vert =\vert \cP\cap\mathfrak{S}_{n}\vert $ as $n \rightarrow \infty$? Here, we are restricting our attention to \textit{proper} permutation classes, i.e.\ classes which do not contain all finite permutations. Indeed, while the collection $\mathfrak{S}_{n}$ of all permutations of $[n]$ has size $n!$, the celebrated Stanley--Wilf conjecture stated that any proper permutation class has at most exponential growth in $n$. After a series of work~\cite{alon2000number, klazar2000furedi}, the Stanley--Wilf conjecture was proved by Marcus and Tardos~\cite{marcus2004excluded}. 

Given Marcus and Tardos's result, we have an exponential upper bound for the size of a proper permutation class.  When $ \limsup_{n\rightarrow \infty} \vert \cP_n\vert^{1/n} = \liminf_{n\rightarrow \infty} \vert \cP_n\vert^{1/n}$ we refer to their common value as the \emph{Stanley--Wilf limit} or \emph{growth rate} $\mathrm{gr}(\cP)$ of the permutation class $\cP$.

The Stanley--Wilf limit is conjectured to exist for all permutation classes $\cP$ (see e.g.~\cite[Conjecture 1.1]{vatter2015permutation}, though this appears to be an older folklore conjecture). A natural question to ask is then: which values can be attained as the Stanley--Wilf limit of some permutation class? When $\cP$ is the set $\mathrm{Av}(\pi)$ of all permutations avoiding the pattern $\pi$, the Stanley--Wilf limit exists, by a remark of Arratia~\cite{arratia1999stanley} in combination with \cite{marcus2004excluded}, but finding the value is hard even for simple patterns.

Given the difficulty of determining growth rates even for some seemingly simple permutation classes, it is natural to approach the problem from a slightly different angle and ask instead: given a (small) constant $c > 0$, what can be said about the structural properties of permutation classes with growth-rate less than $c$, or to determine whether $c$ is the Stanley--Wilf limit of some class? 
Questions of this type have seen significantly more progress than the ones concerned with determining the growth rates. One of the first results in this direction was obtained by Kaiser and Klazar~\cite{kaiser2002growth}, who showed in 2002 that if for some arbitrary $n$ we have $\vert \cP_n\vert < F_n$, where $F_n$ denotes the $n$-th Fibonacci number, then $\vert \cP_n\vert $ grows at most polynomially in $n$. Thus for all proper permutation classes $\cP$, the growth rate $\mathrm{gr}(\cP)$ is either $0$, $1$ or at least $(1+\sqrt{5})/2$. They also obtained a precise classification of the possible growth rates/Stanley--Wilf limits below $2$ (which correspond to certain generalized Fibonacci numbers). 

After several improvements~\cite{huczynska2006grid, polygrowth}, a remarkable breakthrough was made by Vatter~\cite{vatter2011small}, who showed that a certain type of phase transition occurs at the growth-rate $\kappa \approx 2.205$, where $\kappa$ is the least positive root of $1+2X^2-X^3$. We say that a permutation class $\mathcal{P}$ is a \emph{small permutation class} if $\mathrm{\overline{gr}}(\mathcal{P}) := \limsup_{n\rightarrow \infty} |\mathcal{P}_n|^{1/n}< \kappa$. In his influential work, Vatter~\cite{vatter2011small} provided a rather specific structural classification for small permutation classes (which we will state in Section~\ref{sec: Inflations of geometric grid classes}), proved that there are only countably many small permutation classes, and also proved that there are already uncountably many permutation classes with growth-rate precisely $\kappa$ and  pairwise distinct generating functions.

Other natural phase transitions occurring at higher growth-rates were subsequently identified by Pantone and Vatter~\cite{pantonevatter}, and Vatter~\cite{vatter2010permutation}: Pantone and Vatter~\cite{pantonevatter} determined an explicit constant $\xi$ such that there are only countably many Stanley--Wilf limits below $\xi$, but such that for every $\varepsilon>0$ there are uncountably many Stanley--Wilf limits in every half-open interval $[\xi, \xi+\varepsilon)$. Vatter~\cite{vatter2010permutation}, on the other hand, showed that there exists a constant $\lambda \approx 2.48$ such that every real number greater than $\lambda$ is the Stanley--Wilf limit of some proper permutation class. There is a much broader literature on permutation classes and pattern-avoiding permutations - we refer the reader to the monographs and surveys of Bona~\cite{bona2008combinatorics}, Kitaev~\cite{kitaev2010survey}, Linton~\cite{linton2010permutation}, Steingrimsson~\cite{steingrimsson2013some} and Vatter~\cite{vatter2015permutation} for much more on this rich area of combinatorics.

Vatter~\cite{vatter2011small} also conjectured that every small permutation class has a rational generating function, and suggested that 'this is exactly the sort of problem one expects to be able to handle with regular languages'. The conjecture was later proved by Albert, Ru{\v{s}}kuc, and Vatter~\cite{albert2015inflations}; however, their proof does not establish a bijection with a regular language.

\subsection{Contributions of this paper}\label{subsection: results}
In this paper, we prove a first-order convergence law for small permutation classes, thus settling the conjecture of Braunfeld and Kukla~\cite{braunfeld2021logical} for such classes. We do so by proving the following theorem, which proves that for all small permutation classes $\mathcal{F}$ there exists a regular language $L_\mathcal{F}$ and a length-preserving bijection $f_\mathcal{F}:L_\mathcal{F}\rightarrow \mathcal{F}$ so that for any  first-order sentence $\Psi$ on $\mathcal{F}$ there exists a corresponding first-order sentence $\Phi$ on $L_\mathcal{F}$ so that $f_\mathcal{F}$ induces a bijection between their truth-sets. In addition, an analogous statement holds for monadic second-order sentences. In a sense, this can be viewed as a 'half' of the proof that $\mathcal{F}$ and $L_\mathcal{F}$ are first-order or monadic second-order isomorphic. 
\begin{thm}\label{theorem: eeroklas}
    Let $\mathcal{F}$ be a small permutation class. 
    \begin{enumerate}
        \item Then there exists a regular language $L_{\mathcal{F}}$ and a length-preserving bijection
        \[f_{\mathcal{F}}: L_{\mathcal{F}} \rightarrow \mathcal{F}\]
        satisfying the following property: for every first-order logic sentence $\Psi$ on $\mathcal{F}$ there exists a first-order logic  sentence $\Phi$ on $L_{\mathcal{F}}$ so that 
        \[\{w \in L_\mathcal{F}: w \models \Phi\} = \{w \in L_\mathcal{F}: f_\mathcal{F}(w) \models \Psi\}.\]
        The same conclusion also holds for sentences in monadic second-order logic.
        \item{(sketched)} The bijection $f_\mathcal{F}$ is in fact a monadic second-order isomorphism: that is, for every sentence $\Phi$ in monadic second-order logic on $L_\mathcal{F}$ there exists a corresponding sentence in monadic second-order logic on $\mathcal{F}$ for which the conclusion above holds. In addition, it is also possible to construct a regular language $L_\mathcal{F}'$ and a length-preserving bijection $f'_\mathcal{F}: L_\mathcal{F}' \rightarrow \mathcal{F}$ which is a first-order isomorphism. 
        \item For the last three assertions, suppose that $\mathcal{F}$ is not finite. There exist a positive real number $\alpha$, a real number $\beta \geq 1$, and a non-negative integer $c$ so that 
        \[\vert \mathcal{F}_n\vert = (\alpha+o(1))\cdot \beta ^n\cdot n^c.\]
        \item Given a monadic second-order sentence $\Psi$ on $\mathcal{F}$, write \[\mathcal{F}^\Psi_n := \{\sigma \in \mathcal{F}_n: \sigma \models \Psi\}.\]
        Then the sets $\mathcal{F}^\Psi_n$ have a rational generating function.
        \item For every first-order $\mathrm{TOTO}$ sentence $\Phi$ the limit  $\lim_{n\rightarrow \infty}\Pr(\pi \models \Phi)$ exists, where $\pi$  is taken uniformly at random from $\mathcal{F}_n$. 
    \end{enumerate}
\end{thm}
Before moving on, we briefly explain what we mean by the second point. In this paper, we construct a regular language $L_\mathcal{F}$ and a bijection $f_\mathcal{F}:L_\mathcal{F}\rightarrow \mathcal{F}$, and we prove that it satisfies all the properties of Theorem~\ref{theorem: eeroklas} except the second one. In Section~\ref{sec: conclusion}, we also check that for sentences in monadic second-order logic the correspondence works both ways as stated in the second part of the theorem. However, it turns out that for our construction of $f_\mathcal{F}$, the reverse correspondence need not preserve first-order definability due to technical reasons. In Section~\ref{sec: conclusion}, we sketch a slightly different construction of a regular language and a bijection for which the correspondence of sentences in first-order logic works both ways. We remark that given that the first-order convergence law can already be proved using our function, and due to the length of the paper, we chose not to write down the last two constructions in full detail. 

In \cite{ourotherpaper} it is also shown that in the range of exponential growth the set of limit probabilities is dense in the interval $[0,1]$, while in the polynomial growth region the set of limits is finite for each permutation class. For the simplest classes in the polynomial region, there also exist infinite models, leading to 0/1-laws.

Given that we have a first-order convergence law, it is natural to ask what happens for sentences in monadic second-order logic. By using the fact that small permutation classes and their monadic second-order restrictions are in bijection with a regular language, and so can be described by a finite automaton, we may prove a \textit{modular convergence law}. 
\begin{thm} \label{thm: Monadic modular convergence}
    Let $\mathcal F$ be an infinite small permutation class,  and let $\boldsymbol\sigma_n$ denote a permutation chosen uniformly at random from $\mathcal F_n$. For every $\mathrm{MSO}$-$\mathrm{TOTO}$ sentence $\Psi$ there exists an integer $\ell\geq1$ such that for every $0\leq r<\ell$, the limit 
    \[\lim_{n\rightarrow \infty} \mathbb{P}\left( \mathbf{\sigma_{n\ell+r}}\models \Psi\right)\] 
    exists. 
\end{thm}
In general the modular behaviour shown in the theorem cannot be excluded. For every pair of positive integers $\ell$ and $1 \leq a \leq \ell$,  we can construct a monadic second-order sentence $\Psi_{\ell,a}$ so that $\sigma \models \Psi_{\ell,a}$ if and only if $|\sigma| \equiv a \text{ (mod }\ell\text{)}$.  We discuss this briefly in Subsection\ref{msl}.

The proofs in this paper make use of formal languages, just like some of the earlier limit law papers. However, our methods  differ in several ways.  First, our bijection is not as well-behaved in terms of the connection between patterns and subwords as the bijections used for geometric grid classes. In an upcoming paper, we aim to show that this is necessary due to differences in the partial order structures of the permutation classes. Second,  where previous papers have relied on multivariate analysis we instead use the description of regular languages in terms of automata. The automaton theory for regular languages makes it possible for us to  show that a first-order restriction of a class either excludes periodic behavior, or lead to a subclass with density zero. We remark that our technique applies to certain collections of regular languages beyond those in bijection with a small permutation class. Indeed, given a regular language $L$ which has an accepting automaton $\mathcal{A}$ whose dominant components are aperiodic, then our techniques prove $L$ satisfies a first-order convergence law. 

\section{Preliminaries}\label{section: preliminaries}
\subsection{Permutations and $\mathrm{TOTO}$}\label{subsection: permutations}
We denote by $\mathfrak{S}_{n}$ the collection of permutations of $[n]:=\{1,2, \ldots, n\}$, and we write $\vert \sigma\vert =n$ if $\sigma \in \mathfrak{S}_{n}$. We often represent a permutation $\sigma \in \mathfrak{S}_{n}$  as the ordered sequence $\sigma(1)\sigma(2)\ldots \sigma(n)$. Given permutations $\tau\in \mathfrak{S}_{m}$ and $\sigma \in \mathfrak{S}_{n}$ for $n\geq m$, we say that $\tau$ is a \emph{subpattern} of $\sigma$ if there exists an order-preserving injection $f: \ [m]\rightarrow[n]$ such that for all $i,j\in [m]$,  $\tau(i)\leq \tau(j)$ if and only if $\sigma(f(i))\leq \sigma(f(j))$.  A \emph{permutation class} is a collection of permutations $\mathcal{F} \subseteq \bigcup_{n\in \mathbb{N}} \mathfrak{S}_{n}$ which is closed under taking subpatterns. Given a fixed permutation $\tau$, we denote by $\mathrm{Av}(\tau)$ the permutation class of all \emph{$\tau$-avoiding permutations}, that is, the collection of permutations $\sigma$ not containing $\tau$ as a subpattern.

For a general background in first-order logic we refer the reader to \cite{Rautenberg}. As already mentioned we will be working with $\mathrm{TOTO}$,  the Theory Of Two Orders.  In this theory, we think of permutations as finite sets of points in the plane with the additional property that no two points lie on the same horizontal or vertical line. Indeed, given such a set of points $\{(x_i, y_i): \ i \in [n]\}$ with $x_1<x_2<\ldots <x_n$, there is a unique associated permutation $\sigma \in\mathfrak{S}_{n}$ with $y_{\sigma^{-1}(1)}<y_{\sigma^{-1}(2)}<\ldots <y_{\sigma^{-1}(n)}$. Conversely, for any permutation $\sigma \in \mathfrak{S}_n$ the set $\{(i,\sigma(i)): 1 \leq i \leq |\sigma|\}$ naturally induces the permutation $\sigma$, and we sometimes refer to this as the \emph{canonical drawing} of $\sigma$. Thus, $\mathrm{TOTO}$ is concerned with sets equipped with two linear orders, one for each coordinate,  and isomorphism classes of finite $\mathrm{TOTO}$ models are in one-to-one correspondence with permutations, for a proof see e.g~\cite[Proposition 4]{albert2020two}. 

Formally, $\mathrm{TOTO}$ consists of a countably infinite collection of variables $x_{1}, x_{2}, \dots$ and a set of relation symbols\footnote{Elsewhere, the set of relations $R$ is often chosen to contain relations ${<_1, <_2}$. In our setting, certain technical arguments become slightly simpler if we instead work with the non-strict orders $\leq_1$ and $\leq_2$, where $x \leq_1 y$ if and only if $x <_1 y$ or $x = y$ (and similarly for $\leq_2$). These two formulations of $\mathrm{TOTO}$ are clearly equivalent.} $R = \{\leq_1,\leq_2\}$ interpreted as two linear orders. An \textit{atomic formula} in $\mathrm{TOTO}$ is a formula of the form $x_{i}=x_{j}$, $x_{i}\leq_1 x_{j}$ or $x_i \leq_2 x_j$. The set $\mathrm{FO}$ of \textit{first-order formulae} in $\mathrm{TOTO}$ is the least collection of expressions satisfying the following:
 \begin{enumerate}
 	\item[(i)] all atomic formulae lie in $\mathrm{FO}$;
    \item[(ii)] $\bot$ and $\top$ are in $\mathrm{FO}$
    \item[(iii)] if $p$ is in $\mathrm{FO}$, then $\neg p$ is in $\mathrm{FO}$
 	\item[(iv)] if $p$ and $q$ are in $\mathrm{FO}$, then so is $\alpha(p,q)$, where $\alpha$ is any of the standard logical connectives of arity $2$ (for example, this means that $p \lor q$, $p \land q$, $p \rightarrow q$ are in $\mathrm{FO}$)
 	\item[(v)] if $p\in\mathrm{FO}$ and $x$ is a free variable\footnote{Informally, $x$ is a \textit{free variable} in a sentence $p\left(x\right)$ if $p$ does not contain any quantifiers bounding the variable $x$.}  occurring in $p$,
 	then $\forall x: \ p\left(x\right)$ and $\exists x:\ p\left(x\right)$ also belong to $\mathrm{FO}$. 
 \end{enumerate}
A formula $p \in \mathrm{FO}$ is called a \emph{(first-order) sentence} if it contains no free variables. We denote by $\mathrm{FOS} \subseteq \mathrm{FO}$ the set of all first-order sentences.

We also define a monadic second-order extension of $\mathrm{TOTO}$, which we denote by $\mathrm{MSO}$-$\mathrm{TOTO}$. In addition to the countably infinite collection of first-order variables $x_{1}, x_{2}, \dots$, the language of $\mathrm{MSO}$-$\mathrm{TOTO}$ contains a countably infinite collection of second-order variables $y_{1}, y_{2}, \dots$. Informally, these second-order variables can be viewed as sets consisting of first-order elements. Besides the atomic formulae of $\mathrm{TOTO}$, $\mathrm{MSO}$-$\mathrm{TOTO}$ includes atomic formulae of the form $x \in y$ where $x$ is a first-order variable and $y$ is a second-order variable. 

In our proofs it is more convenient to work directly with formulae that are written in disjunctive normal form. A \textit{literal} $\ell$ is a formula that is either an atomic formula or the negation of an atomic formula. We say that a sentence $\Psi$ is written in \textit{disjunctive normal form} if $\Psi$ is given by
 \[
 Q_{1}x_{1}\dots Q_{k}x_{k}\left(\left(p_{1,1}\land p_{1,2}\land \dots\land p_{1,a_{1}}\right)\lor\left(p_{2,1}\land p_{2,2}\land\dots\land p_{2,a_{2}}\right) \lor \dots\lor\left(p_{b,1}\land p_{b,2}\land \dots\land p_{b,a_{b}}\right)\right),
 \]
 where each $Q_{i}\in\left\{ \forall,\exists\right\} $ is a quantifier,
 $x_{1},\dots,x_{k}$ are variables, and each $p_{i,j}=p_{i,j}\left(\underline{x}\right)$
 is a literal using a subset of the variables $x_{1},\dots,x_{k}$. In the first-order setting, the variables must be first-order variables, while in the monadic setting they are allowed to be either first-order or second-order variables. 

 Note that a sentence $\Psi$ as above is said to have \emph{quantifier depth} (see e.g.~\cite{albert2020two}) at most $k$. It is well-known that every sentence can be converted to one in disjunctive normal form, and thus is logically equivalent to a sentence in disjunctive normal form. 
 \begin{proposition}
 	\label{prop: every FOS can be written in DFN}  Given a monadic second-order sentence $\Psi$, there exists a monadic second-order sentence $\Psi_{DNF}$ that is written in disjunctive normal form, and for which $\Psi$ and $\Psi_{DNF}$ are logically equivalent. 
 \end{proposition}
\noindent For a more comprehensive background on the first-order logic $\mathrm{TOTO}$, see~\cite{albert2020two, albert2024logical}. 

Finally, we introduce a useful notation. Given a finite set $J$, enumerate its elements as $J = \{j_1,\dots,j_m\}$. We write $Q_{j\in J} x_j: \varphi$ as a shorthand for $Qx_{j_1} \dots Qx_{j_m}: \varphi$, where $Q \in \{\forall,\exists \}$ is a quantifier. 
\subsection{Regular languages} 

Regular languages are a class of formal languages that can be recognized by simple computational models, such as finite state automata, which we will revisit in the next subsection. Since our primary interest in regular languages arises from their applications to logical limit laws, we provide only a brief introduction focusing on those aspects that are directly relevant to our work. For a more comprehensive background on regular languages, see for example \cite{hopcroft-motwani-ullman}. 

For our purposes, very few standard properties of regular languages are required. Let $\Sigma$ be a finite alphabet. The \emph{Kleene star} $\Sigma^*$ is defined as the set of finite words over $\Sigma$. It is a standard fact that $\Sigma^*$ is a regular language for any finite set (see e.g.\cite{hopcroft-motwani-ullman}). To state the second result of our interest, we first introduce the necessary logical background on regular languages. 

Let $\Sigma$ be a finite alphabet and let $K \subseteq \Sigma^*$ be a regular language. The language of the (monadic) first-order logic on $K$ consists of
\begin{enumerate}
    \item[(i)] an infinite set of first-order variables $x_1, x_2, \dots$, interpreted as positions in a word,
    \item[(ii)] for each symbol $a \in \Sigma$, a predicate $a(\cdot)$. One should view $a(x)$ as being true if the letter at position $x$ is $a$,
    \item[(iii)] the ordinary relations $<$ and $=$ between natural numbers, and
    \item[(iv)] the usual propositional connectives and first-order quantifiers.
\end{enumerate}
\noindent Analogously to $\mathrm{TOTO}$, first-order formulae are built from atomic formulae of the form $i < j$ or $i = j$ for variables $i$ and $j$, and from formulae of the form $a(i)$ for $a \in \Sigma$. For notational convenience, from now on we  write $w_i = a$ instead of $a(i)$. We also write $w_i = w_j$ for the first-order formula 
\[\bigvee_{a \in \Sigma} (w_i = a \land w_j = a),\]
which is satisfied if and only if the letters at positions $i$ and $j$ are equal. 

The monadic second-order logic on words is obtained from the first-order logic in a manner closely analogous to the construction described in section~\ref{subsection: permutations}. It consists of an infinite set of first-order variables $x_1, x_2, \dots$ and an infinite set of second-order variables $y_1, y_2, \dots$, together with the atomic formulae of the first-order language described above. In addition, monadic second-order logic includes atomic formulae of the form $x \in y$, where $x$ is a first-order variable and $y$ is a second-order variable.

A key structural property of regular languages is the following closure property with respect to monadic second-order logic.  
\begin{thm}[B\"{u}chi-Elgot-Trakhtenbrot]\label{thm: closure for regular languages}
    A language $L$ is a regular language iff there exists a monadic second-order sentence  $\Psi$ such that $L$ is the set of words from $\Sigma^*$ satisfying $\Psi$. That is, $L:= \left\{w \in \Sigma^*: w \models \Psi\right\}$. 
\end{thm}
Note that since the conjunction of two sentences is a new sentence, this also implies that any monadic second-order restriction of a regular language is another regular language. 

\noindent The \emph{generating function} of a language $L$ is the formal power series 
\[f_L(x) = \sum_{n = 0}^{\infty} |L_n|x^n.\] 
We say that $L$ has a \emph{rational generating function} if there exists a positive real number $r$ such that $f_L$ has radius of convergence at least $r$ and it coincides inside the radius of convergence with a rational function. Recall that a function $g(x)$ is a rational function if it can be expressed as a quotient of two polynomials, that is, $g(x) = \frac{P(x)}{Q(x)}$ for some polynomials $P$ and $Q$. The final property of regular languages required for our purposes is the following classical result.
\begin{thm} \label{thm: reg languages have RGF}
Every regular language has a rational generating function.
\end{thm}
\noindent We conclude this section with a simple observation. Since many of the alphabets considered later are constructed inductively and naturally decompose as disjoint unions of structured subsets, the following lemma will be useful for simplifying first-order formulae in such settings.
\begin{lem}\label{lem:projection}
    Let $\Sigma_{1}$ and $\Sigma_2$ be two finite sets, and let $\pi: \Sigma_{1} \rightarrow \Sigma_2$ be any function. Then, for any subset $A \subseteq \Sigma_2$, the condition $\pi(w_i) \in A$ can be expressed as a  first-order formula. 
\end{lem}
\begin{proof}
     Since $\Sigma_1$ is a finite set, the pre-image $B = \pi^{-1}(A)$ is also a finite set. The claim follows by observing that $\pi(w_i) \in A$ is clearly equivalent to
    \[\bigvee_{b \in B} (w_i = b),\]
    which is a first-order formula.  
\end{proof}
\subsection{Automata theory}\label{sec: automata theory}

A \emph{finite deterministic automaton} $\mathcal{A}$ consists of a tuple $(Q,\Sigma,\delta,q_0,F)$ where $Q$ is a finite state set, $\Sigma$ is the finite alphabet,  $\delta:Q\times\Sigma\to Q$ is a partial transition function, $q_0$ is an initial state, $F\subseteq Q$ is the set of accepting states. We extend $\delta$ naturally to words by setting $\delta(q,\varepsilon) = q$ where $\varepsilon$ is the empty word, and $\delta(q,vw) = \delta(\delta(q,v),w)$ for a word $w$ and letter $v$. We say that a word $w$ is accepted by $\mathcal{A}$ if $\delta(q_0,w)\in F$.

We say that $\mathcal{A}$ is \emph{complete} if for every pair $(q,v) \in Q \times \Sigma$ the value $\delta(q,v)$ is defined. As we primarily only care about the properties of the language accepted by the automaton, we may always treat $\mathcal{A}$ as complete by adding a rejection state $\star$ to which all the undefined $\delta(q,v)$'s are mapped to, and which satisfies $\delta(\star,v) = \star$ for every $v$. 

A key property of finite deterministic automata $\mathcal{A}$ is that the language $L$ accepted by it is always a regular language. For a more detailed background on automata theory, again see for example~\cite{hopcroft-motwani-ullman}. For the purposes of our paper, we only require a few standard properties. The first one of these is the following property satisfied by first-order restrictions of regular languages, which is a consequence of \cite[Theorem~VI.1.1]{Straubing}. 
\begin{lem}\label{lemma: iterate}
Let $\Phi$ be a first-order sentence on words over $\Sigma$. There
are a complete finite deterministic automaton $\mathcal B$
recognizing $\{w:w\models\Phi\}$ and an integer $N\geq1$ such that
\begin{equation}\label{eq: iterate}
 \delta_{\mathcal B}(r,v^N)=\delta_{\mathcal B}(r,v^{N+1})
 \qquad(r\in Q_{\mathcal B},\ v\in\Sigma^*).
\end{equation}
\end{lem}

We also require certain structural properties of finite deterministic automata that affect the generating function of the accepted regular language. Roughly speaking, simply the fact that small permutation classes are in bijection with a regular language will not be enough to prove the first-order convergence law, as the sequence of  coefficients of the corresponding rational generating functions may have modular behaviour. However, we will prove that the fact that permutation classes are closed under taking sub-patterns together with the structure of the function $f_{\mathcal{F}}$ constructed imply that no such modular behaviour can occur. In order to do so, we need to introduce some necessary background.

For our purposes, it is convenient to view the automaton as a labelled multi-digraph. The vertices of the digraph $D_{\mathcal{A}}$ are the states $q\in Q$, and for each $q,q'\in Q$, we place a directed edge $q \rightarrow q'$ labelled with $v\in \Sigma$ if and only if $\delta(q,v) = q'$. A \emph{strongly connected component} of $\mathcal{A}$ is a maximal subset of states $S$ for which the induced digraph $D[S]$ is strongly connected, i.e.\ for every $u,v \in S$ there exists a directed path from $u$ to $v$ in $D[S]$. Note that while this definition allows a single state without a self-loop to be a strongly connected component, this is fine for our purposes as we only focus our attention on those which have an effect on the growth-speed of the generating function. In addition, we say that a state $q$ or a strongly connected component $C$ is \emph{productive} if it occurs on some walk from the initial state to an accepting state. 

Given an automaton $\mathcal{A}$, let $\mathcal{A'}$ denote the automaton obtained by removing all the states that are not productive. Furthermore, we write $\delta'$ for the restriction of $\delta$; note that if the walk starting from $q_0$ corresponding to a word $w$ would enter a removed state, we simply reject $w$ already at this stage. Since deleted states are not productive, we note that $w$ would have been rejected in $\mathcal{A}$ eventually, and thus $\mathcal{A}$ and $\mathcal{A'}$ both accept the same language. 

Given a finite deterministic automaton $\mathcal{A}$, its transition matrix $A$ is defined as 
\[
A_{p,q} = |\{a \in \Sigma: \delta(p,a) = q\}|.
\]
Given a strongly connected component $C$, we write $A_C$ for the restriction of $A$ on the states in $C$. Let $A'$ denote the transition matrix of the automaton $\mathcal{A'}$ from which all the non-productive states are removed. We write $\rho(\mathcal{A})$ for the spectral radius of $A'$, i.e.\ for the largest absolute value of an eigenvalue of $A'$. The reason behind defining $\rho(\mathcal{A})$ as the spectral radius of $A'$ instead of $A$ is the fact that the spectral radius of $A'$ naturally corresponds to the upper growth rate of the language accepted by $\mathcal{A}$ (while the spectral radius of $A$ can be controlled by the rejecting states). We say that a productive component $C$ is \emph{dominant} if $\rho(A_C) = \rho(\mathcal{A})$.

The \emph{period} of a strongly connected component $C$ containing a directed cycle is the greatest common divisor of the lengths of directed cycles contained in $C$. We say that a strongly connected component $C$ containing a cycle is \emph{periodic} if the period of $C$ is strictly greater than $1$, and otherwise it is aperiodic. We remark that this notion is purely graph-theoretic, and is different from the  aperiodicity of the transition monoid, which is a notion that is not required for our paper. 

While it may not be clear for the reader at this stage, the existence of a rational generating function on its own will not be enough to prove the third and fifth statements in Theorem~\ref{theorem: eeroklas}. One way to prove these statements is to verify that there exists an automaton accepting the language $L_{\mathcal{F}}$ whose dominant components are aperiodic, and this is indeed how our proof proceeds. To properly execute this strategy, we  require certain standard linear algebraic properties of transition matrices, together with some preliminary results which allow us to compare the spectral radius of a first-order restriction of a dominant aperiodic component to its original spectral radius. These results are presented in Section~\ref{sec:spectral-comparison}, and we introduce the relevant linear algebraic background also in the same subsection. 

\subsection{Geometrically griddable permutation classes and geometric grid classes}	\label{subsection: geom grid classes} 
In this subsection, we define geometric grid classes and geometrically griddable classes of permutations, which play a central role in the theory of permutation classes with polynomial growth. 

Given a matrix $M$ with entries in $\{-1, 0, 1\}$, we shall consider two specific permutation classes that can be associated to it, namely the class of permutations induced by $M$-compatible gridded figures, denoted by $\mathrm{Grid}\left(M\right)$, and the geometric grid class of $M$, which we denote by $\mathrm{Geom}\left(M\right)$.  As we are primarily interested in  geometric grid classes, we present here a slightly simplified definition of the gridded figures treated by Albert, Ru{\v{s}}kuc and Vatter in~\cite{albert2015inflations}.

Given an $\ell \times m$ matrix $M$, we index the rows from bottom to top, and the columns from left to right, and write $M_{i,j}$ for the entry in the $i^{th}$ column from the left and $j^{th}$ row from the bottom. This choice of indexing has a geometric motivation: given a unit square grid on the rectangular subset $\left[0,\ell\right]\times\left[0,m\right]$ of the plane $\mathbb{R}^2$, the unit square whose top-right corner is $\left(i,j\right)$ can be associated in a natural way with the entry $M_{i,j}$.

To make this association formal, we let the standard $\ell\times m$-\emph{gridding}  $G_{\ell\times m}$ be the partition of the rectangle $R=\left[0,\ell \right]\times\left[0,m\right]$ into $\ell m$ interior-disjoint unit squares whose corners are the lattice points of $\mathbb{Z}^{2}$. We write $C_{i,j}$ for the unit square (or \emph{cell}) whose top-right corner is $\left(i,j\right)$. 
\begin{definition}
    An $\ell\times m$ \textit{gridded figure }$\left(F,G_{\ell\times m}\right)$ consists of the standard $\ell\times m$ gridding $G_{\ell\times m}$ together with a finite set of points $F\subseteq\left[0, \ell\right]\times\left[0,m\right]$ such that no two points of $F$ lie on the same horizontal or vertical line, and no point of $F$ lies on the grid-lines of $G_{\ell\times m}$.

\end{definition}
\noindent Clearly each gridded figure $\left(F,G_{\ell\times m}\right)$ induces a permutation $\sigma$: let $x_1<x_2<\ldots < x_{\vert F\vert}$ and $y_1<y_2<\ldots <y_{\vert F\vert}$ denote the $x$- and $y$- coordinates of points in $F$. Then there is a unique permutation $\sigma \in \mathfrak{S}_{\vert F\vert}$ such that $F=\{(x_i, y_{\sigma(i)}): \ i\in [\vert F\vert ] \}$. 

\begin{definition}
    We say that an $\ell\times m$ gridded figure $\left(F,G_{\ell\times m}\right)$ is \textit{compatible} with a matrix $M$ if $M$ is an $\ell \times m$ matrix with entries in $\{-1,0,1\}$ such that for each $\left(i,j\right)\in\left[\ell \right]\times\left[m\right]$ we have 
    \begin{align*}
    F\cap C_{i,j} \textrm{ is }\left\{\begin{array}{ll}
	   \text{empty } & \text{if }M_{i,j}=0,\\
      \text{a (possibly empty) increasing segment} & \text{if }M_{i,j}=1,\\
	   \text{a (possibly empty) decreasing segment} & \text{if }M_{i,j}=-1.
    \end{array}\right.
    \end{align*}
    Here we say that a set of points $\{(x_i, y_i): i \in [t]\}$ in $\mathbb{R}^2$ with $x_1<x_2<\ldots <x_t$ forms an \emph{increasing segment} if $y_1<y_2<\ldots <y_t$, and a \emph{decreasing segment} if $y_1>y_2>\ldots > y_t$.
\end{definition}

\noindent We say that two gridded figures $(F,G)$ and $(F',G)$ with the same gridding $G$ are \emph{equivalent} if they induce the same permutation $\sigma$, and for each $i \in [|\sigma|]$, the points $p_i \in F$ and $p'_i \in F'$ corresponding to the $i$-th point in $\sigma$ belong to the same cells in $G$. We write $\textrm{Grid}^{\#}\left(M\right)$ for the set of equivalence classes of all gridded figures $\left(F,G_{\ell\times m}\right)$ that are compatible with $M$. Further, we let $\textrm{Grid}\left(M\right)$ denote the set of
all permutations $\sigma$ induced by the gridded figures $\left(F,G_{\ell\times m}\right)$ which belong to any equivalence class contained in $\textrm{Grid}^{\#}\left(M\right)$. It is easily checked that $\mathrm{Grid}(M)$ is a permutation class. We also write $\pi^\#$ for the natural map $\pi^\#: \mathrm{Geom}^\#(M) \rightarrow \mathrm{Geom}(M)$ which removes the gridlines. 

The \emph{standard figure} of $M$ is the subset of $\left[0,\ell\right]\times\left[0,m\right]$ consisting of the union of (a) the increasing diagonals $\left(i-1,j-1\right)+\{(t,t): \ t\in (0,1)\}$ in the cells $C_{i,j}$ with $M_{i,j}=1$, and (b) the decreasing diagonals $\left(i-1,j-1\right)+\{(t,1-t): \ t\in (0,1)\}$ in the cells $C_{i,j}$ with $M_{i,j}=-1$. Note that the standard figure thus contains no point inside the cells $C_{i,j}$ for which $M_{i,j}=0$.
\begin{definition}
    Let $M$ be an $\ell\times m$ matrix and let $G$ denote the standard $\ell\times m$ gridding. We write $\textrm{Geom}^{\#}\left(M\right)$ for the set of equivalence classes in $\textrm{Grid}^{\#}\left(M\right)$ which contain a gridded figure $\left(F,G\right)$ as a representative for which $F$ is a finite subset of the standard figure of $M$. The \emph{geometric grid class} of $M$, denoted by $\textrm{Geom}\left(M\right)$, is the collection of all permutations $\sigma$ induced by the equivalence classes in $\textrm{Geom}^{\#}\left(M\right)$. 
\end{definition}
 \noindent It is clear that $\textrm{Geom}\left(M\right)$ is a permutation class. We say that a permutation class $\mathcal{P}$ is \emph{geometrically griddable} if we have $\mathcal{P}\subseteq\mathrm{Geom}\left(M\right)$
for some matrix $M$ whose entries are in the set $\left\{ -1,0,1\right\} $.

We trivially have $\mathrm{Geom}^{\#}\left(M\right)\subseteq\mathrm{Grid}^{\#}\left(M\right)$
and $\mathrm{Geom}\left(M\right)\subseteq\mathrm{Grid}\left(M\right)$.
However, these inclusions are not always equalities. Indeed,
as pointed out in~\cite{albert2013geometric}, for $M=\left(\begin{array}{rr}
	-1 & 1\\
	1 & -1
\end{array}\right)$, the permutation $2413$ lies in
$\textrm{Grid}(M)$ but not in $\textrm{Geom}(M)$.
However the authors of~\cite{albert2013geometric} showed one does get equality in certain special cases, which turn out to be exactly the ones needed for our purposes. Given an $\ell\times m$ matrix $M$, define the  \emph{row-column graph} $H_{M}$ to be the graph with vertices $V\left(H_{M}\right):=\left[\ell \right]\sqcup\left[m\right]$ and edges consisting of those pairs $ij$ for which $M_{i,j}\neq0$.  
\begin{proposition}[Albert, Atkinson, Bouvel, Ru{\v{s}}kuc and Vatter~\cite{albert2013geometric}]\label{prop: equality grid and geometric grid when HM is a forest}
	\label{thm:Forest theorem}If $H_M$ is a forest, then $\textrm{Grid}^{\#}\left(M\right)=\mathrm{Geom}^{\#}\left(M\right)$ 	(and thus $\textrm{Grid}\left(M\right)=\mathrm{Geom}\left(M\right)$). 
\end{proposition}
\noindent From now on, we say that a matrix $M$ is a \emph{forest} if the associated row-column graph $H_M$ is a forest. 

In many parts of our arguments, the corresponding results (especially concerning bijectivity of certain functions) are substantially easier for gridded figures. Recall that each gridded figure induces a permutation, as we have already remarked. However, for a given matrix $M$ there may be many fundamentally different non-equivalent $M$-compatible gridded figures giving rise to the same permutation, which is a problem if we hope to use such gridded figures to prove any counting results (such as convergence laws). To resolve this difficulty, Albert, Atkinson, Bouvel, Ru{\v{s}}kuc and Vatter introduced the \textit{left-greedy order}, a linear order providing a way for comparing two gridded figures which gives rise to the same permutation. 
\begin{definition}\label{def: left greedy}
    Given a permutation $\sigma$ and a gridding $G$, we define the \emph{left-greedy order} $<_{LG}$ on equivalence classes of gridded figures with gridding $G$ that induce $\sigma$ by letting $(F, G)<_{LG}(F', G')$ if one of the following occurs:
    \begin{enumerate}
	   \item there exists an index $i_0$ such that for all $i<i_0$ and for all $j$, the number of points of $F$ and $F'$ in the $(i,j)$-th cells of their respective griddings are equal, but there are strictly more points of $F$ than of $F'$ in the cells located in the $i_0$-th column;
	   \item there exist indices $i_0$ and $j_0$ so that for every $(i,j)$ with either $i < i_0$, or $i = i_0$ and $j < j_0$,  the number of points of $F$ and $F'$ in the $(i,j)^{th}$ cell of their respective griddings are equal, the number of points of $F$ and $F'$ in the cells located in the $i_0$-th column are equal,     
       but there are strictly more points of $F$ than of $F'$ in the $(i_0,j_0)$-th cell.
    \end{enumerate}
\end{definition}

\noindent Intuitively, $(F, G)<_{LG} (F', G)$ if $(F,G)$ is `more compressed towards the left, or equally left-compressed and more compressed towards the bottom' of their respective griddings. It is easily seen that for any fixed gridding $G$, $<_{LG}$ interplays well with our notion of equivalent griddings (e.g.\ if $(F, G)$ and $(F',G)$ are equivalent and $(F',G)<_{LG}(F'', G)$ then $(F,G)<_{LG}(F'',G)$), and thus for each fixed permutation $\sigma$, $<_{LG}$ induces a total order on its $M$-compatible griddings. We also write $(F, G)\leq_{LG} (F', G)$  if and only if $(F, G) <_{LG} (F', G)$, or $(F, G)$ and $(F', G)$ belong to the same equivalence class. See Figure~\ref{fig:left-greedy} for an illustration. 

\begin{figure}
    \centering
\begin{tikzpicture}[x=.85cm,y=.85cm,font=\small]

  \foreach \s in {0,4,8,12} {
    \begin{scope}[shift={(\s,0)}]
      \draw[step=1,gray!60,thin] (0,0) grid (3,4);
      \draw[blue!50,thick] (0,2) -- (1,1);
      \draw[blue!50,thick] (0,3) -- (1,2);
      \draw[blue!50,thick] (1,3) -- (2,4);
      \draw[blue!50,thick] (2,0) -- (3,1);
    \end{scope}
  }
  \begin{scope}
    \foreach \x/\y in
      {.2/1.8,.4/1.6,.7/2.3,1.3/3.3,1.7/3.7,2.5/.5}
      \fill (\x,\y) circle (2.2pt);
  \end{scope}
  \begin{scope}[shift={(4,0)}]
    \foreach \x/\y in
      {.3/1.7,.7/1.3,1.25/3.25,1.5/3.5,1.75/3.75,2.5/.5}
      \fill (\x,\y) circle (2.2pt);
  \end{scope}
  \begin{scope}[shift={(8,0)}]
    \foreach \x/\y in
      {.3/2.7,.7/1.3,1.25/3.25,1.5/3.5,1.75/3.75,2.5/.5}
      \fill (\x,\y) circle (2.2pt);
  \end{scope}
  \begin{scope}[shift={(12,0)}]
    \foreach \x/\y in
      {.3/2.7,.7/2.3,1.25/3.25,1.5/3.5,1.75/3.75,2.5/.5}
      \fill (\x,\y) circle (2.2pt);
  \end{scope}

  \foreach \x in {3.5,7.5,11.5}
    \node at (\x,2) {$<_{LG}$};

  \node at (7.5,-.55) {$\sigma=324561$};
\end{tikzpicture}
    \caption{An example of gridded figures all inducing the permutation $\sigma=324561$ and their left-greedy ordering. Note that these are all $M$-compatible for $M$ with $M_{3,1}=M_{2,4}=1$, $M_{1,2}=M_{1,3}=-1$, and $M_{i,j} =0$ otherwise.}
    \label{fig:left-greedy}
\end{figure}
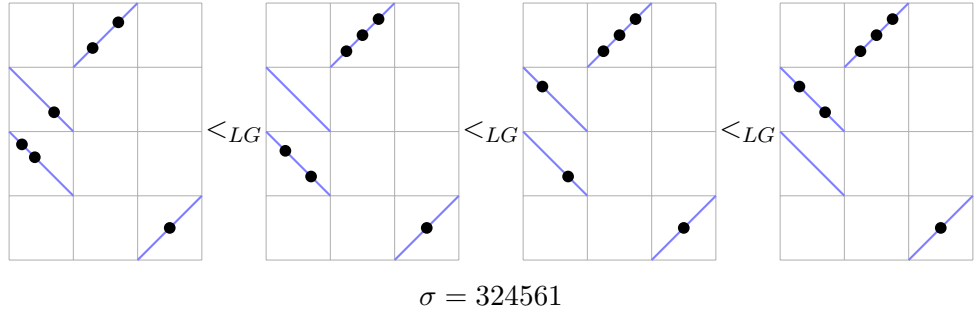

To simplify our notation, instead of writing everything in terms of gridded figures, we simply choose such representative of the equivalence class, and call it a gridded permutation. To highlight the distinction between ordinary permutations and gridded permutations, we use $\#$ in the subscript (such as $\sigma_\#$) when working with gridded permutations. One should view $\sigma_\#$ as the representative of the equivalence class containing the associated gridded figure. It turns out that such choice will not produce any ambiguity, as all of our arguments are independent of the chosen representative of the equivalence class. In particular, when writing $\sigma_\# = \theta_\#$ for gridded permutations $\sigma_\#$ and $\theta_\#$, one should interpret the equality as a statement that $\sigma_\#$ and $\theta_\#$ belong to the same equivalence class. 

Given a permutation $\sigma$, we write $G(\sigma)$ for the collection of all (equivalence classes of) gridded permutations in $\mathrm{Geom}^\#(M)$ which induce $\sigma$. We refer to the least element in $G(\sigma)$ under the order $<_{LG}$ as the \emph{left-greedy representation} of $\sigma$.

\subsection{Inflations of geometric grid classes}\label{sec: Inflations of geometric grid classes}
Next, we will define inflations of geometric grid classes, a generalization of the geometric grid classes based on  matrices $M$ having entries in $\{-1, 0, 1, \star\}$, satisfying certain additional properties. Recall that $M$ is said to be a forest if the associated row-column graph is a forest. In addition, we say that $M$ is \emph{$\star$-isolated} if  whenever $M_{i,j}=\star$, then $M_{i,j'}=0$ for all $j'\neq j$ and $M_{i',j}=0$ for all $i'\neq i$. That is, the $\star$-entries are unique non-zero entries in their respective rows and columns. Finally, we say that $M$ is a \emph{$\star$-isolated forest matrix} if $M_{i,j} \in \{-1,0,1,\star\}$ for all $(i,j)$, and $M$ is both $\star$-isolated and forest. We shall associate a certain class of griddable permutations to each $\star$-isolated forest matrix $M$. To define these properly, we need to make a number of technical definitions.

An \emph{interval} inside a permutation $\sigma$ is an integer interval $I=\left[a,b\right]\cap\mathbb{Z}_{>0}$ such that its image $\sigma\left(I\right):=\left\{ \sigma\left(i\right):\ i\in I\right\} $ under $\sigma$ is also an integer interval. If the restriction of $\sigma$ to $I$ is a monotone permutation, then we refer to $I$ as a \emph{monotone interval} of $\sigma$.
\begin{definition}
    Given a permutation $\tau \in \mathfrak{S}_k$ and permutations $\sigma_{1},\dots,\sigma_{k}$, we define the \emph{inflation} $\tau\left[\sigma_{1},\dots,\sigma_{k}\right]$ of $\tau$ by the $k$-tuple of permutations $(\sigma_{1},\ldots,\sigma_{k})$ to be the permutation obtained by replacing each entry $\tau\left(i\right)$ by an interval that is order-isomorphic to the permutation $\sigma_{i}$. 
\end{definition}
 \noindent As an example to illustrate this definition, the inflation of $\tau=4312$ by $(132, 21, 1,231)$ is the permutation
\[
4312[132,21,1,231] = 798\,65\,1\,342.
\]
where the spaces in the expression on the right hand-side have been added to highlight the transitions between each of the four intervals in the inflation of $\tau$. See Figure~\ref{fig:inflation} for an illustration. One may think of inflations in a natural way as permutations that can be decomposed into smaller  permutations.
\begin{figure}
    \centering
\begin{tikzpicture}[scale=1.2,font=\small]
  \draw[step=1,black,thin] (0,0) grid (4,4);
  \foreach \x/\y in {.25/3.25,.50/3.75,.75/3.50}
    \fill (\x,\y) circle (2pt);
  \foreach \x/\y in {1.33/2.67,1.67/2.33}
    \fill (\x,\y) circle (2pt);
  \fill (2.50,.50) circle (2pt);
  \foreach \x/\y in {3.25/1.50,3.50/1.75,3.75/1.25}
    \fill (\x,\y) circle (2pt);

\end{tikzpicture}
    \caption{Example of an inflation $4312[132,21,1,231]=798\,65\,1\,342$}
    \label{fig:inflation}
\end{figure}
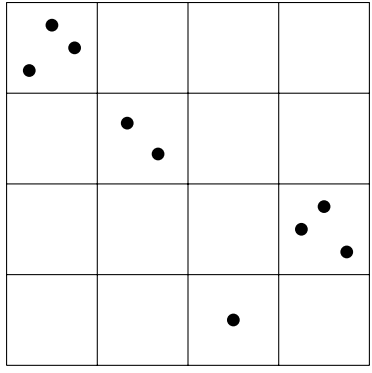
\begin{definition}
	Given two permutation classes $\mathcal{P}$ and $\mathcal{Q}$, we
	define the \emph{inflation} $\mathcal{P}\left[\mathcal{Q}\right]$ of $\mathcal{P}$ by $\mathcal{Q}$
	to be:
	\[
	\mathcal{P}\left[\mathcal{Q}\right]:=\left\{ \tau\left[\sigma_{1},\dots,\sigma_{k}\right]:\ \tau\in\mathcal{P}_{k}, \ \sigma_{1},\ldots,\sigma_{k}\in\mathcal{Q}\right\} .
	\]
	Given an inflation of the form $\sigma=\tau\left[\sigma_{1},\dots,\sigma_{k}\right]\in\mathcal{P}\left[\mathcal{Q}\right]$ where $\tau \in \mathcal{P}$ and $\sigma_1, \ldots, \sigma_k \in \mathcal{Q}$, we say that $\tau\left[\sigma_{1},\dots,\sigma_{k}\right]$ is a \emph{$\mathcal{Q}$-decomposition} of $\sigma$, and we call $\tau$ a \emph{$\mathcal{P}$-profile} of the expression $\sigma = \tau\left[\sigma_{1},\dots,\sigma_{k}\right]$.
\end{definition}
\noindent It is easily checked that for any pair of permutation classes $\mathcal{P}, \mathcal{Q}$, the inflation $\mathcal{P}\left[\mathcal{Q}\right]$ also forms a permutation class. Also, note that a permutation $\sigma\in\mathcal{P}\left[\mathcal{Q}\right]$ may have multiple \emph{$\mathcal{Q}$}-decompositions; there is no uniqueness implied in our definition. However, for convenience reasons we assign the notion of $\mathcal{P}$-profiles into inflations, instead of permutations themselves, as in our applications, we often care about a particular choice of inflation, and the profile of that particular inflation.  

Despite having non-uniqueness of \emph{$\mathcal{Q}$}-decompositions, one can get around this by using a decomposition analogue of the left-greedy ordering.

\begin{definition}\label{def: LG order}
    Let $\mathcal{P}$ and $\mathcal{Q}$ be two permutation classes. Let $\tau\left[\sigma_{1},\dots,\sigma_{k}\right]$ and $\tau'\left[\sigma'_{1},\dots,\sigma'_{k'}\right]$ be two $\mathcal{Q}$-decompositions of the same permutation in $\mathcal{P}[\mathcal{Q}]$, we say that $\tau\left[\sigma_{1},\dots,\sigma_{k}\right]<_L\tau'\left[\sigma'_{1},\dots,\sigma'_{k'}\right]$ if one of the following holds:
    \begin{enumerate}
        \item $\vert \tau \vert =k <k' =\vert \tau'\vert$; or
        \item $k=k'$ and there exists $i_0\in [k]$ such that for all $i\in [i_0-1]$ we have $\vert \sigma_{i}\vert =\vert \sigma'_{i} \vert$ but $\vert \sigma_{i_0}\vert >\vert \sigma'_{i_0}\vert$.
    \end{enumerate}
\end{definition}

\noindent As an example for $\mathcal{Q}=\mathrm{Av}(12)$ and $\mathcal{P}$ being the set of all permutations, the $\mathcal{Q}$-decompositions of the permutation $1432$ satisfy \[ 12[1,321]<_L 132[1,21,1] <_L 132[1,1,21] <_L 1432[1,1,1,1] \]

For every $\sigma \in \mathcal{P}[\mathcal{Q}]$, it is easily seen that there exists a unique $<_L$-minimal $\mathcal{Q}$-decomposition of $\sigma$. We often refer to this $<_L$-minimal decomposition as the \emph{left-greedy decomposition}. This will turn out to be a very useful fact in our arguments.

\begin{definition}
Given a permutation class $\mathcal{P}$, we define the sequence of permutation classes $\left(\mathcal{P}^{\left[d\right]}\right)_{d\in \mathbb{Z}_{\geq1}}$ 
inductively by setting $\mathcal{P}^{\left[1\right]}:=\mathcal{P}$
and $\mathcal{P}^{\left[d+1\right]}:=\mathcal{P}\left[\mathcal{P}^{\left[d\right]}\right]$. 
\end{definition}
\noindent In our arguments, we shall be naturally led to consider a particular family of inflations, namely inflations of the identity permutation.
\begin{definition}
	Given permutations $\sigma_1, \ldots, \sigma_k$, their \emph{direct sum} is the permutation $\sigma_1 \oplus \sigma_2\oplus \cdots \oplus \sigma_k:= 123\ldots k[\sigma_1, \sigma_2, \ldots , \sigma_k]$. 
\end{definition} 
\noindent In addition to direct sums, we will also need to handle inflations of the monotone decreasing permutation, which we shall denote by $\sigma_1 \ominus \sigma_2\ominus \cdots \ominus \sigma_k:= k(k-1)(k-2)\ldots 21[\sigma_1, \sigma_2, \ldots , \sigma_k]$.

\begin{definition}
    The \emph{increasing oscillating sequence} is the integer sequence $(s_n)_{n\in \mathbb{Z}_{\geq 0}}$ defined by $s_{n}=n+4$ if $n$ is even and $s_n=n$ otherwise; thus its first few terms are $4, 1, 6, 3, 8, 5, \ldots$. We let $F_k:=\{(i, s_i): \ 0 \leq i < k\}$, and denote by $\omega_k$ the permutation associated to $F_k$. We define \emph{decreasing oscillating sequence} analogously as $G_k = \{(i,-s_i): 0 \leq i < k\}$, and denote $\varpi_k$ the permutation associated to $G_k$. 
    
    We let $\mathcal{O}_k$ denote the finite permutation class consisting of all subpermutations of $\omega_k$ and $\varpi_k$, and let $\widetilde{\mathcal{O}}_{k}$ denote the permutation class that is the union of $\mathcal{O}_k$, and the set containing all monotone increasing and decreasing permutations of finite length. That is, we set 
    \[\widetilde{\mathcal{O}}_{k} = \mathcal{O}_{k} \cup \mathrm{Av}(21) \cup \mathrm{Av}(12).\]
\end{definition}
\begin{definition}
    Given a  $\ell\times m$ $\star$-isolated forest matrix $M$, the standard $\ell \times m$ gridding $G$ and integers $d,k\in \mathbb{Z}_{+}$, we define $\mathrm{Geom}^{\#}(M)_{d,k}$ to be the collection of all equivalence classes which contain a gridded figure $(F,G)$ satisfying the following:
    \begin{enumerate}
        \item if $M_{i,j}=0$ then $F\cap C_{i,j}=\emptyset$;
        \item if $M_{i,j}=+1$, we have $F\cap C_{i,j} \subseteq \left\{(i-1, j-1)+(t,t): \ t\in (0,1)\right\}$, while if $M_{i,j}=-1$ we have $F\cap C_{i,j} \subseteq \left\{(i-1, j-1)+(t,1-t): \ t\in (0,1)\right\}$;
       \item if $M_{i,j}=\star$, then $F\cap C_{i,j}$ induces a permutation from ${\widetilde{\mathcal{O}}_{k}}^{[d]}$.
    \end{enumerate}
    Finally we let $\mathrm{Geom}(M)_{d,k}$ denote the collection of all permutations induced by an equivalence class in $\mathrm{Geom}^{\#}(M)_{d,k}$.
\end{definition}
\noindent As Definition~\ref{def: left greedy} of the left-greedy order $<_{LG}$ was given for gridded figures in general, without any reference to a matrix, it applies in this setting also, and given $\sigma \in \mathrm{Geom}(M)_{d,k}$ we can again write $\sigma^{\#}$ for the $<_{LG}$-minimal  equivalence class in $\mathrm{Geom}^{\#}(M)_{d,k}$ that induces $\sigma$. With these notions in place, we can now state the fundamental theorem of Vatter~\cite[Theorem 5.4]{vatter2011small} on the structure of small permutation classes; we state it here in its reformulation by Albert, Ru{\v{s}}kuc and Vatter~\cite[Theorem 8.4]{albert2015inflations}.
\begin{proposition}[Vatter]\label{prop: small implies contained inside Geom (M) dk}
    Let $\mathcal{F}$ be a small permutation class. Then there exists a $\star$-isolated forest matrix $M$ and integers $d,k\in \mathbb{Z}_{+}$ such that $\mathcal{F}\subseteq \mathrm{Geom}(M)_{d,k}$.
\end{proposition}
\noindent In fact, the result of Vatter states that one may require $M$ to satisfy a slightly stronger condition than only being $\star$-isolated forest matrix. However, for our purposes this form of a statement is sufficient, as we prove our main results for $\star$-isolated forest matrices. 

Finally, we say that a $\{-1,0,1,\star\}$-valued matrix $M$ is \emph{multiplicative} if there exists $\delta_1,\dots,\delta_\ell$ and $\varepsilon_1,\dots,\varepsilon_m \in \{\pm1\}$ so that $M_{i,j} = \delta_i \varepsilon_j$ whenever $M_{i,j} \in \{\pm1\}$. When $M$ is a $\star$-isolated forest matrix, a simple inductive argument on the number of non-zero entries of $M$ implies that $M$ is also multiplicative. Indeed, removing an edge in $H_M$ whose endpoint is a leaf corresponds to changing a non-zero entry that is the unique non-zero entry in its respective row or column to $0$ in $M$. Since that entry is the unique non-zero entry in its row or column, its removal does not affect the possibility of choosing signs $\delta_i$ and $\varepsilon_j$ consistently.

While there are matrices with entries in $\{-1,0,1\}$ that are not multiplicative, we point out that in \cite[Proposition 4.2]{albert2013geometric} it was proved that for any such $M$, there exists a multiplicative matrix $M'$ so that $\mathrm{Geom}^\#(M) \subseteq \mathrm{Geom}^\#(M')$. Thus, from the perspective of geometric grid classes, the restriction to multiplicative matrices would not be in general a substantial restriction. 
     
\subsection{Previous work on regular languages and small permutation classes} \label{subsetion: albert construction}

In this subsection, we present a construction of Albert, Atkinsson, Bouvel, Ru\v{s}kuc and Vatter \cite{albert2013geometric} which establishes a bijection between $\mathrm{Geom}(M)$ and $\mathrm{Geom}^\#(M)$ with a regular language when $M$ is a multiplicative forest matrix with entries in $\{-1,0,1\}$. 

Given such $\ell \times m$ matrix $M$, let $\delta_1,\dots,\delta_\ell,\varepsilon,\dots,\varepsilon_m \in \{\pm 1\}$ be chosen so that for every $(i,j)$ with $M_{i,j} \neq 0$ we have $M_{i,j} = \delta_i \cdot \varepsilon_j$. Let $G$ denote the standard unit-square gridding of $[0,\ell] \times [0,m]$, and let $\Sigma_M = \{a_{i,j}: M_{i,j} \neq 0\}$ be the alphabet indexed by the non-zero entries of $M$.

Define the function  $g_M: \Sigma_M^* \rightarrow \mathrm{Geom}^\#(M)$ as follows. Let $w = w_1\dots w_n$ be a word of length $n$, and fix arbitrary real numbers $0 < d_1 < \dots < d_n < 1$. Let $F_w$ denote the gridded figure which consists of points $p_1,\dots,p_n$, where for each $t \in [n]$, 
\begin{enumerate}
    \item If $w_t = a_{i,j}$ with $\delta_i = 1$, then $p_t$ is placed at the intersection of the diagonal of slope $M_{i,j}$ in the cell $C_{i,j}$ and the vertical line $x = i - 1 + d_t$.
    \item If $w_t = a_{i,j}$ with $\delta_i = -1$, then $p_t$ is placed at the intersection of the diagonal of slope $M_{i,j}$ in the cell $C_{i,j}$ and the vertical line $x = i - d_t$.
\end{enumerate}
\noindent Thus each point $p_t$ lies on the diagonal of slope $M_{i,j}$ in the cell $C_{i,j}$ determined by the letter $w_t$. The fractional part of the $x$-coordinate of $p_t$ is either $d_t$ or $1 - d_t$, depending on the sign of $\delta_i$. A straightforward computation shows that the fractional part of the $y$-coordinate is likewise $d_t$ or $1 - d_t$, depending on the sign of $\varepsilon_j$.

It is not difficult to verify that the equivalence class of the gridded figure $(F_w, G)$ does not depend on the particular choice of the parameters $d_1, \dots, d_n$, but only on the word $w$. We therefore define $g_M(w)$ to be the equivalence class containing $(F_w, G)$, and we define $f_M: \Sigma_M^* \rightarrow \mathrm{Geom}(M)$ by setting $f_M = \pi^\# \circ g_M$. 

A straightforward inductive argument shows that both $g_M$ and $f_M$ are surjective. In general, however, neither map is injective. In \cite{albert2013geometric}, the authors identified a natural equivalence relation on $\Sigma_M^*$ whose equivalence classes coincide with the fibres of $g_M$. Using standard structural results on trace monoids, they deduced (see \cite[Corollary 7.3]{albert2013geometric}) the existence of a regular language $L_1 \subseteq \Sigma_M^*$ such that the restriction of $g_M$ to $L_1$ is bijective onto $\mathrm{Geom}^\#(M)$ - we return to the technicalities in section~\ref{section: small classes} when we construct the corresponding finite deterministic automaton. 

To prove the existence of the second regular language $L_2$, the authors of \cite{albert2013geometric} introduced an elegant marking scheme for both letters and points in permutations. Using these marked letters and points, they proved that there exists a regular sublanguage $L_2 \subseteq L_1$ with the following property: for every permutation $\sigma \in \mathrm{Geom}(M)$, the set $L_2$ contains a unique word $w \in L_1$ such that $f_M(w) = \sigma$  and $g_M(w)$ is the left-greedy representative of $f_M(w)$.

A key ingredient in their argument is the natural correspondence between the letters of $w$ and the points in the associated gridded figure $g_M(w)$ (and hence in $f_M(w)$). For later purposes, it is convenient to formalize this correspondence. Define functions $\alpha_{w,M}:[|w|] \rightarrow f_M(w)$ and $\beta_{w,M}:[|w|] \rightarrow g_M(w)$ as follows. For each $t \in [|w|]$, we set $\beta_{w,M}(t) = p_t$, where $p_t$ is the point constructed in the definition of $g_M$. We then define $\alpha_{w,M}(t)$ to be the image of $\beta_{w,M}(t)$ under the map $\pi^\#$.
These maps will play an important role in establishing logical correspondences between small permutation classes and regular languages. In section~\ref{sect: construction}, we construct analogous maps to handle matrices containing $\star$-entries.

We say that a function $f: L \rightarrow \mathrm{Geom}(M)$ or $g: L \rightarrow \mathrm{Geom}^\#(M)$ is \emph{order preserving} if whenever $u \leq v$ in the sub-word order\footnote{that is, $u$ can be obtained from $v$ by deleting letters}, then $f(u)$ is a sub-pattern of $f(v)$, or $g(u)$ is a gridded sub-pattern of $g(v)$ respectively. We are now ready to state the relevant results from~\cite{albert2013geometric}. 
\begin{thm}\label{thm: Albert et al regular language}
    Let $M$ be a multiplicative forest matrix with entries in $\{-1,0,1\}$. There exist regular languages $L_1$ and $L_2$ and order preserving bijections $g_M:L_1 \rightarrow \mathrm{Geom}^\#(M)$ and $f_M:L_2\rightarrow \mathrm{Geom}(M)$. 
\end{thm}
\noindent We remark that the fact that both of these functions are order preserving follows somewhat easily by considering the images of the points in the sub-word via the maps $\alpha_{u,M}$ and $\alpha_{v,M}$ introduced earlier. 

The connection between geometric grid classes and regular languages was further studied by Braunfeld~\cite{braunfeld2025decidability}. For every finite matrix $M$ with entries in $\{-1,0,1\}$, he proved that the basis and the rational generating function of $\mathrm{Geom}(M)$ are computable from $M$. More generally, given an $\mathrm{MSO}$-$\mathrm{TOTO}$ sentence $\Psi$, the set
\[
    \mathcal{C}_{M,\Psi}
    := \{\sigma\in\mathrm{Geom}(M):\sigma\models\Psi\}
\]
has a rational generating function that is computable from $M$ and $\Psi$. His results do not apply for matrices with $\star$-entries, which are behind the major technical challenges  in our results, and do not provide convergence laws. However, he does make use of monadic second-order logic as a way of describing geometric grid-classes. Using this, and the connection to regular languages, allows him to establish that the generating function and basis are computable for these classes. 

We conclude this subsection by briefly comparing our approach with that of \cite{albert2013geometric} at a conceptual level. As in \cite{albert2013geometric}, we begin by constructing an alphabet $\Sigma_M$ together with a function $g_M: \Sigma_M^* \rightarrow \mathrm{Geom}_{d,k}^\#(M)\cup \{\star \}$. However, our construction differs in two essential aspects.

First, we allow certain words to be mapped to the auxiliary symbol $\star$. This reflects a structural difference in the underlying alphabets. In our setting, the alphabet $\Sigma_M$ is not endowed with the same global combinatorial structure as in \cite{albert2013geometric}; rather, it decomposes as a disjoint union of several structured subsets. Words that mix letters originating from different components of this decomposition do not admit a natural interpretation as permutations in $\mathrm{Geom}(M)_{d,k}$. To overcome this issue, we choose to map such words to the special element $\star$. By contrast, in \cite{albert2013geometric} this issue does not arise, due to better structural properties of the classes $\mathrm{Geom}(M)$ when $M$ does not contain $\star$-entries. 

In our setting, however, the maps $g_M$ and $f_M$ constructed in section~\ref{section: small classes} are not order-preserving. To illustrate this phenomenon, let $i_n$ and $d_n$ denote the increasing and decreasing permutations of length $n$, respectively. One can construct a word $w$ of length $n^2$ in our language such that $f_M(w) = i_n[d_n, \dots, d_n]$ while there exists a subword $w'$ obtained by deleting two letters from $w$ for which $f_M(w') = i_{n-1}[d_{2n-2},d_n,\dots,d_n]$. In particular, $f_M(w')$ is not, in general, a sub-pattern of $f_M(w)$. The fact that the functions $f_M$ and $g_M$ are not order-preserving in our case suggests that an approach based on marked letters and marked points analogous to the one used in \cite{albert2013geometric} does not easily extend to our framework.

\section{Encoding permutations in $\tilde{\mathcal{O}}_{\textit{k}}^{\left[\textit{d}\right]}$}

Our goal in this section is to prove the following intermediate result, which allows us to deal with the $\star$-entries of a matrix $M$. 

\begin{thm}\label{thm:Main O_K}
    Let $d$ and $k$ be positive integers. There exists a regular language $L_{d,k}$ and a bijection $f_{d,k}:L_{d,k} \rightarrow \tilde{\mathcal{O}}_{\textit{k}}^{\left[\textit{d}\right]}$ so that for every first-order or monadic second-order sentence $\Psi$ on $\tilde{\mathcal{O}}_{\textit{k}}^{\left[\textit{d}\right]}$ there exists a first-order or a monadic second-order sentence $\Phi$  on $L_{d,k}$ for which 
    \[\{w \in L_{d,k}: w \models \Phi\} = \{w \in L_{d,k}: f_{d,k}(w) \models \Psi\}.\]
\end{thm}
\noindent In order to have a better grasp on the iterative structure of the permutation class $\tilde{\mathcal{O}}_{\textit{k}}^{\left[\textit{d}\right]}$, we define certain intermediate structures  $\mathcal{G}_{d,k}$ as follows. We start by setting $\mathcal{G}_{1,k} =\tilde{\mathcal{O}}_{\textit{k}}$. In order to distinguish between the elements of $\widetilde{\mathcal{O}}_k^{[d]}$ and $\mathcal{G}_{d,k}$, we use the overline $\overline{\theta}$ to highlight the fact the element $\overline{\theta}$ belong to $\mathcal{G}_{d,k}$ instead of being a permutation. 

Having defined $\mathcal{G}_{d-1,k}$ for some $d \geq 2$, define $\mathcal{G}_{d,k}$ to be the set consisting of all tuples $\left(\tau,\overline{\theta}_1,\dots,\overline{\theta}_{|\tau|}\right)$ with $\tau \in \tilde{\mathcal{O}}_{\textit{k}}$, and $\overline{\theta}_1,\dots,\overline{\theta}_{|\tau|} \in \mathcal{G}_{d-1,k}$. Given an element $\sigma = (\tau,\overline{\theta}_1,\dots,\overline{\theta}_{|\tau|}) \in \mathcal{G}_{d,k}$, we refer to $\tau$ as the \emph{profile} of $\sigma$. We also define the natural projection maps $\pi_{d,k}: \mathcal{G}_{d,k} \rightarrow \tilde{\mathcal{O}}_{\textit{k}}^{\left[\textit{d}\right]}$ recursively by setting $\pi_{1,k} = \mathrm{id}$ and 
\[\pi_{d,k}(\tau,\overline{\theta}_1,\dots,\overline{\theta}_{|\tau|}) = \tau[\pi_{d-1,k}(\overline{\theta}_1),\dots,\pi_{d-1,k}(\overline{\theta}_{|\tau|})]\]
for $d \geq 2$. 

To prove Theorem~\ref{thm:Main O_K}, we proceed in three steps. First, we construct an alphabet $\Sigma_{d,k}$ together with a function $g_{d,k}: \Sigma_{d,k}^* \rightarrow \mathcal{G}_{d,k} \cup \{\star\}$ satisfying structural properties analogous to those required in Theorem~\ref{thm:Main O_K}. The auxiliary symbol $\star$ is used as an image of those words that are not relevant for our purposes. Crucially, we prove that there exists a first-order sentence which characterizes exactly the words that are mapped to $\star$.

Second, we prove that the function $f_{d,k} := \pi_{d,k} \circ g_{d,k}$ is both a first-order and a monadic second-order homomorphism. Finally, we establish the existence of a first-order sentence which selects a unique representative in each fibre $(\pi_{d,k} \circ g_{d,k})^{-1}(\sigma)$ for $\sigma \in \tilde{\mathcal{O}}_{\textit{k}}^{\left[\textit{d}\right]}$. This unique choice of representatives allows us to deduce Theorem~\ref{thm:Main O_K}.

Throughout this section, all constructions and proofs are carried out by induction on $d$, while the parameter $k$ is treated as fixed. For notational simplicity, we therefore suppress the dependence on $k$ whenever its value is clear from the context. 
\subsection{Construction}\label{sect: construction}

We write $\mathcal{Q}_k$ for the set of non-monotone permutations in $\mathcal{O}_k$. As the base case $d=1$, we take 
\[\Sigma_{1}=\left\{ a_{\mathrm{inc}},a_{\mathrm{dec}}\right\} \cup\bigcup_{\tau\in\mathcal{Q}_{k}}\left\{ a_{\tau,i}:1\leq i\leq\left|\tau\right|\right\},\]
and define the function $g_{1}$ by setting 
\begin{enumerate}
\item[(i)] $g_{1}\left(a_{\mathrm{inc}}^{n}\right)=12\dots n$ for all $n\geq 1$, $g_{1}\left(a_{\mathrm{dec}}^{n}\right)=n\left(n-1\right)\dots1$
for all $n\geq2$ and $g_{1}\left(a_{\mathrm{dec}}\right)=\star$;
\item[(ii)] $g_{1}\left(a_{\tau,1}\dots a_{\tau,\left|\tau\right|}\right)=\tau$
for every $\tau\in\mathcal{Q}_{k}$; and 
\item[(iii)] $g_{1}\left(w\right)=\star$ for any other $w\in\Sigma_{1}^{*}$.
\item [(iv)] Let $\varepsilon$ denote the empty word. Then $g_d(\varepsilon) = \star$. 
\end{enumerate}
As $\mathcal{G}_{1,k} = \tilde{\mathcal{O}}_{\textit{k}}$, we also have $f_1 = g_1$. We remark that the pre-image of every $\sigma \in \widetilde{\mathcal{O}}_k$ contains exactly one element. 

Having constructed the set of alphabets $\Sigma_{d-1}$ and the function $g_{d-1}$ for some $d \geq 2$, enumerate the elements of $\Sigma_{d-1}$ as $\left\{a_{1},\dots,a_{N}\right\}$. Define $\Sigma_{d,\mathcal{Q}}$ and $\Sigma_{d,m}$ by setting
\begin{align*}
\Sigma_{d,\mathcal{Q}} & =\left\{ a_{\tau,i,j}:\tau\in\mathcal{Q}_{k},i\in\left[\left|\tau\right|\right],j\in\left[N\right]\right\} ,\text{ and }\\
\Sigma_{d,m} & =\left\{ a_{\alpha,\beta,i}:\alpha\in\left\{\mathrm{inc},\mathrm{dec}\right\} ,\beta\subseteq\left\{ L,R\right\} ,i\in\left[N\right]\right\},
\end{align*}
\noindent and set $\Sigma_d = \Sigma_{d,\mathcal{Q}} \sqcup \Sigma_{d,m}$. At this stage, it turns out to be convenient to introduce several projection maps which capture the essential properties of the inductive construction. We introduce the maps $\rho_1: \Sigma_{d} \rightarrow \Sigma_{d-1}$, $\rho_2: \Sigma_{d} \rightarrow \mathcal{Q}_k \cup \{\mathrm{inc},\mathrm{dec}\}$, $\rho_3: \Sigma_{d,\mathcal{Q}} \rightarrow \mathbb{Z}_+$ and $\rho_4:\Sigma_{d,m} \rightarrow \mathcal{P}(\{L,R\})$ which are defined as follows. 
\begin{enumerate}
    \item[(i)] For every $a_{\tau,i,j} \in \Sigma_{d,\mathcal{Q}}$ we set $\rho_1(a_{\tau,i,j}) = a_j$, and for every $a_{\alpha,\beta,j} \in \Sigma_{d,m}$ we set $\rho_1(a_{\alpha,\beta,j}) = a_j$. In addition, we extend $\rho_1$ to a function $\rho_1: \Sigma_{d}^*\rightarrow \Sigma_{d-1}^*$ by setting $\rho_1(w_1\dots w_\ell) = \rho_1(w_1)\dots \rho_1(w_\ell)$. 
    \item[(ii)] For every $a_{\tau,i,j} \in \Sigma_{d,\mathcal{Q}}$, we set $\rho_2(a_{\tau,i,j}) = \tau$, and for every $a_{\alpha,\beta,j} \in \Sigma_{d,m}$ we set $\rho_2(a_{\alpha,\beta,j}) = \alpha$. 
    \item[(iii)] For every $a_{\tau,i,j} \in \Sigma_{d,\mathcal{Q}}$ we set $\rho_3(a_{\tau,i,j}) = i$, and note that we have $\rho_3(a_{\tau,i,j}) \in [|\tau|]$ for every $a_{\tau,i,j} \in \Sigma_{d,\mathcal{Q}}$. 
    \item[(iv)] For every $a_{\alpha,\beta,j} \in \Sigma_{d,m}$ we set $\rho_4(a_{\alpha, \beta, j}) = \beta$. 
\end{enumerate}
\noindent Note that we may also extend $\rho_2$ and $\rho_3$ to $\Sigma_1$ by setting $\rho_2(a_{\tau,i}) = \tau$, $\rho_2(a_\alpha) = \alpha$ for $\alpha \in \{\mathrm{inc},\mathrm{dec}\}$, $\rho_3(a_{\tau,i}) = i$. 

Define the \emph{profile} of a letter $b \in \Sigma_{d}$ to be the value of $\rho_2(b)$. Given a word $w\in\Sigma_{d,k}^{*}$, we say that $w$ is \emph{consistent} if each of its letter has the same profile. If $w$ is consistent, we refer to this common value as the \textit{profile} of $w$. 

If $w$ is not consistent, we set $g_d(w) = \star$. Now suppose that $w$ is non-empty and consistent, and first consider the case when its profile $\tau$ is in $\mathcal{Q}_k$. For each $j$ define $u_j$ to be the (possibly empty) sub-word of $w$ consisting of those letters $w_i$ of $w$ which satisfy $\rho_3(w_i) = j$. We say that $w$ \emph{decomposes properly} if $w$ coincides with the concatenation $w = u_1\dots u_\ell$, where $\ell = |\tau|$. 

Now suppose that $w$ is a consistent word whose profile is $\mathrm{inc}$ or $\mathrm{dec}$. In this case, we say that $w$ \emph{decomposes properly} if $w$ can be written as a concatenation $w = u_1\dots u_{\ell}$ for some $\ell \geq 1$, where the non-empty words $u_i := u_{i,1} \dots u_{i,k_i}$ satisfy the following conditions. 
\begin{align} \label{eq: consistent conditions for I,D - 1}
    \text{1. } \text{We have $L \in \rho_4(u_{i,s})$ if and only if $s = 1$. }
\end{align}
\noindent
\begin{align} \label{eq: consistent conditions for I,D - 2}
    \text{2. } \text{We have $R \in \rho_4(u_{i,s})$ if and only if $s = k_i$. }
\end{align}
\noindent In either case, we set $g_d(w) = \star$ if $w$ does not decompose properly. Otherwise, observe that $w$ can be written as $w = u_1\dots u_\ell$ for some $\ell$, with the words $u_i$ satisfying the appropriate conditions depending on whether the profile $\tau$ satisfies $\tau \in \mathcal{Q}_k$ or $\tau \in \{\mathrm{inc},\mathrm{dec}\}$. Define $\overline{\theta}_1, \dots,\overline{\theta}_\ell \in \mathcal{G}_{d-1,k} \cup \{\star\}$ by setting $\overline{\theta_i} = g_{d-1}\left(\rho_1(u_i)\right)$ for each $i \in [\ell]$. Recall that if  $\tau \in \mathcal{Q}_k$ and $u_i$ is the empty word, we set $\overline{\theta}_i = g_{d-1}(\varepsilon)= \star$. 

Define $g_d(w)$ according to the following rules. If there exists $i \in [\ell]$ for which $\overline{\theta}_i = \star$, set $g_d(w) = \star$. In addition, if $\tau = \mathrm{dec}$ and $\ell = 1$, we set $g_d(w) = \star$. Otherwise, define $g_{d}(w) = (\tau',\overline{\theta_1},\dots,\overline{\theta_{\ell}})$ where 
\begin{enumerate}
    \item [(i)] $\tau' = \tau$ whenever $\tau \in \mathcal{Q}_k$,
    \item [(ii)] $\tau' = 12\dots \ell$ whenever $\tau = \mathrm{inc}$, and 
    \item [(iii)] $\tau' = \ell(\ell-1)\dots1$ whenever $\tau = \mathrm{dec}$ (and in this case we must have $\ell \geq 2 $).
\end{enumerate}
Finally, set $f_d := \pi_d \circ g_d$ with the convention that $\pi_d(\star) = \star$. An easy inductive argument implies that both $f_d$ and $g_d$ are surjective: indeed, if the profile $\tau$ is the singleton $1$, then it is encoded only using a word whose profile is $\mathrm{inc}$.

Having constructed $\Sigma_d$ and the functions $f_d$ and $g_d$ for all $d$, we finish the subsection by introducing the notion of \emph{corresponding points}, which turns out to be very useful when proving that $f_d$ is a first-order homomorphism. Given a word $w \in \Sigma^{*}_{d}$ with $g_{d}(w) \neq \star$, we inductively construct bijections $\alpha_{w}: [|w|] \rightarrow f_{d}(w)$ and $\beta_w: [|w|] \rightarrow g_{d}(w)$, which naturally establishes a one-to-one correspondence between the letters of a word $w$ and the points in $g_{d}(w)$ and $f_{d}(w)$. As the base case, we construct the function $\beta_w$ for all $w \in \Sigma_1^*$ with $g_1(w) \neq \star$ as follows:
\begin{enumerate}
    \item[(i)] if $w = a_{\mathrm{inc}}^n$ or $w = a_{\mathrm{dec}}^n$, define $\beta_w(i)$ to be the $i^{th}$ least point with respect to $\leq_1$ in $g_{1}(w)$, and 
    \item[(ii)] if $w = a_{\tau,1}\dots a_{\tau,|\tau|}$, define $\beta_w(i)$ to be the point $(i,\tau(i))$ in the canonical drawing of $\tau := g_{1}(w)$. 
\end{enumerate}
Since the functions $f_1$ and $g_1$ are identical, we may view $g_1(w)$ as a permutation, which justifies the notion of the $i$-th least point in $g_1(w)$. We also set $\alpha_w = \beta_w$. It is easy to check that each of the functions $\alpha_w$ and $\beta_w$ defined as above is a bijection. 

Suppose that $d \geq 2$ and that the functions $\alpha_{w'}$ and $\beta_{w'}$ are defined for each $w' \in \Sigma_{d-1}^*$ with $g_{d-1}(w') \neq \star$. Given $w \in \Sigma_{d}^*$ with $g_d(w) \neq \star$, suppose that $w$ decomposes properly as $w = u_1\dots u_\ell$ for some $\ell$. For each $i \in [\ell]$, write $\overline{\theta}_i = g_{d-1}(\rho_1(u_i))$ and $\sigma_i = f_{d-1}(\rho_1(u_i))$. 

Given an index $s \in [|w|]$, let $j \in [\ell]$ be chosen so that $w_s$ belongs to the sub-word $u_j$, and write $s'$ for the index chosen so that $w_s$ corresponds to the letter $(u_j)_{s'}$ in $u_j$. Let $p \in \sigma_j$ and $q \in \overline{\theta}_j$ denote the points chosen so that $p = \alpha_{\rho_1(u_j)}(s')$ and $q = \beta_{\rho_1(u_j)}(s')$. Since $g_{d}(w) = (\tau',\overline{\theta}_1,\dots,\overline{\theta}_\ell)$ for some $\tau'$, there exists a point $q'$ in $g_{d}(w)$ which naturally corresponds to $q$, and is located in the part of $g_{d}(w)$ which corresponds to $\overline{\theta}_j$, and we define $\beta_w(s) = q'$. Similarly, there exists a point $p'$ corresponding to $p$ which is located in the part of $f_d(w)$ that corresponds to $\sigma_j$, and we define $\alpha_w(s) = p'$. Since each letter is contained in a unique sub-word $u_i$, and the functions $\alpha_w$ and $\beta_w$ map letters from $u_i$ to the respective objects $\sigma_i$ or $\overline{\theta}_i$ contained in $f_{d}(w)$ or $g_d(w)$ respectively, it is straightforward to check that the functions $\alpha_w$ and $\beta_w$ are bijections. 

\subsection{Capturing the $\star$-values by using first-order sentences}\label{section: stars determined by FOS's}
Our goal in this sub-section is to prove the following lemma, which states that it is possible to determine the set $g_d^{-1}(\star)$ by using a first-order sentence. 
\begin{lem}\label{lemma:gamma}
    Given positive integers $d$ and $k$, there exists $\Gamma_{d,k} \in \mathrm{FOS}$ for which we have $w \models \Gamma_{d,k}$ if and only if $g_{d,k}(w) \neq \star$. 
\end{lem}
\noindent Our proof is by induction on $d$, and treats the parameter $k$ as constant. Thus, we again suppress the presence of $k$ from our notation for clarity, and for instance simply write $\Gamma_d$ on behalf of $\Gamma_{d,k}$. Before proving Lemma~\ref{lemma:gamma}, we introduce some notation. Given a $\Psi \in \mathrm{FOS}$ or $\psi(k_1,\dots,k_\ell) \in \mathrm{FO}$, define $\Psi_{i}, \Psi_{i,j} \in \mathrm{FOS}$, and $\psi_i(k_1,\dots,k_\ell), \psi_{i,j}(k_1,\dots,k_\ell) \in \mathrm{FO}$ as follows. 
\begin{enumerate} 
    \item[(i)] $\Psi_i$ (resp.\ $\psi_{i}(k_1,\dots,k_\ell)$) is defined to be the first-order sentence (resp.\ first-order formula) obtained from $\Psi$ (resp.\ from $\psi$) by adding the additional condition $\rho_3(w_t) = i$ for every variable $t$ occurring in $\Psi$ (resp.\ in $\psi$). 
    \item[(ii)] $\Psi_{i,j}$ (resp.\ $\psi_{i,j}(k_1,\dots,k_\ell)$) is defined to be the first-order sentence (resp.\ first-order formula) obtained from $\Psi$ (resp.\ from $\psi)$  by adding the additional condition $(i \leq t) \land (t \leq j)$ for every variable $t$ occurring in $\Psi$ (resp.\ in $\psi$). 
\end{enumerate}
\begin{lem}\label{lemma: restriction to sub-words preserves FOS}
    Given $\Psi \in \mathrm{FOS}$ (resp.\ $\psi(x_1,\dots,x_\ell) \in \mathrm{FO}$), then we have $\Psi_i, \Psi_{i,j} \in \mathrm{FOS}$ (resp.\ $\psi_i(x_1,\dots,x_\ell), \psi_{i,j}(x_1,\dots,x_\ell) \in \mathrm{FO}$) for all $i$ and $j$. 
\end{lem}
\begin{proof}
    The proof is an easy application of Lemma~\ref{lem:projection}, we omit the details. 
\end{proof}
\begin{proof}[Proof of Lemma~\ref{lemma:gamma}]
    Let $k$ be a fixed positive integer. As the base case $d = 1$, consider the first-order sentence $\Gamma_{1}$ given by 
\begin{align*}
&\ \bigl( (\exists x:w_x=a_{\mathrm{inc}})\land (\forall x\colon w_x=a_{\mathrm{inc}})\bigr)\\
&\lor\bigl( (\exists x \exists y: (x < y) \land (w_x=a_{\mathrm{dec}})\land w_y = a_{\mathrm{dec}}) \land (\forall x\colon w_x=a_{\mathrm{dec}})\bigr)\\
&\lor\bigvee_{\tau\in\mathcal Q_k}\bigl(\Gamma_\tau\land\forall x\colon\rho_2(w_x)=\tau\bigr),
\end{align*}
where $\Gamma_{\tau}$ is the first-order sentence given by
\begin{align*}
& \bigwedge_{i=1}^{|\tau|} \left(\exists j_1 \forall j_2: w_{j_1} = a_{\tau,i} \land (w_{j_2} = a_{\tau,i} \implies j_1 = j_2)\right)
\\& \land \left(\forall j_1 \forall j_2: (j_1 < j_2 \implies \rho_3(w_{j_1}) < \rho_3(w_{j_2}))\right).
\end{align*}
In order to check that $\Gamma_{1}$ satisfies the required properties. we start by observing that when $w_1 = a_{\mathrm{inc}}$ the first line of $\Gamma_1$ ensures that we have $w \models \Gamma_1$ if and only if $w = a_{\mathrm{inc}}^n$ for any $n \in \mathbb{Z}_+$. Similarly, whenever $w_1 = a_{\mathrm{dec}}$, the second line ensures that we have $w \models \Gamma_1$ if and only if $w = a_{\mathrm{dec}}^n$ for any $n \in \mathbb{Z}_{\geq 2}$. In both instances, these two conditions are equivalent to the condition $g_{1}(w) \neq \star$. 

Finally, consider the case when $\rho_2(w_1) = \tau$ for some $\tau \in \mathcal{Q}_k$. In this case, we have $g_1(w) \neq \star$ if and only if $w = w^{(\tau)}$, where $w^{(\tau)} := a_{\tau,1} \dots a_{\tau,|\tau|}$. Now suppose that $w \models \Gamma_{1}$ and $\rho_2(w_1) = \tau$. Then, (i) the first part of the top line of $\Gamma_\tau$ ensures that each $a_{\tau,i}$ for $i \in [|\tau|]$ occurs at least once in $w$, (ii) the second part of the top line of $\Gamma_\tau$ ensures that each letter $a_{\tau,i}$ occurs exactly once in $w$, and (iii) the second line ensures that the letter $a_{\tau,i}$ occurs before $a_{\tau,j}$ in $w$ whenever $i < j$. It is evident that the only word satisfying all three of these conditions is $ w^{(\tau)}$, which completes the proof of the base case. 

Suppose that $d \geq 2$ and that $\Gamma_{d-1}$ satisfies the required conditions for $d-1$. Define auxiliary first order formulae $\psi_{\neq}(i,j,t)$ and $\psi_{=}(i,j)$ as 
\begin{align*}
    (i\neq j) \implies \left(\rho_{4}\left(w_{i}\right)=\left\{ L\right\} \land\rho_{4}\left(w_{j}\right)=\{R\}\land\left(i<t<j\implies\rho_{4}\left(w_{t}\right)=\emptyset\right) \land (\Gamma_{d-1})_{i,j}\right)
\end{align*}
and 
\begin{align*}
    (i=j) \implies \left(\rho_{4}\left(w_{i}\right)=\left\{ L,R\right\} \land(\Gamma_{d-1})_{i,j}\right)
\end{align*}
respectively. Note that both of them are first-order formulae by Lemma~\ref{lemma: restriction to sub-words preserves FOS}. Define $\Gamma_{d}$ to be the first-order sentence  
    \begin{align*}
& \forall i: (\rho_2(w_i) = \rho_2(w_1))
\\& \land \bigwedge_{\tau \in \mathcal{Q}_k} \left((\rho_2(w_1) = \tau  )\implies (\forall i \forall j:i < j \implies \rho_3(w_i) \leq \rho_3(w_j))\right)
\\& \land \bigwedge_{\tau\in\mathcal{Q}_{k}}\left(\left(\rho_{2}\left(w_{1}\right)=\tau\right)\implies\left(\bigwedge_{j=1}^{\left|\tau\right|}\left(\exists i: \rho_3(w_i) = j \land (\Gamma_{d-1})_{j}\right)\right)\right)
\\& \land  \left(\rho_2(w_1) \in \{\mathrm{inc},\mathrm{dec}\} \implies \left(\forall s \exists i \exists j\forall t: \left(i\leq s\leq j\right)\land \psi_{\neq}(i,j,t) \land \psi_{=}(i,j)\right)\right)
\\& \land \left(\rho_2(w_1) = \mathrm{dec} \implies \exists a\exists b:\left(a<b\land L\in\rho_4(w_a)\land L\in\rho_4(w_b)\right)\right).
\end{align*}
We start by checking that for every $w$ with $w \models \Gamma_d$ we have $g_d(w) \neq \star$. The first line of $\Gamma_d$ ensures that $w$ is consistent. Let $\tau \in \mathcal{Q}_k\cup\{\mathrm{inc},\mathrm{dec}\}$ be the profile $\tau$ of $w$, and first consider the case when $\tau \in \mathcal{Q}_k$. Thus, only the second and third line of $\Gamma_d$ are relevant in this case. 

Let $u_j$ denote the sub-word of $w$ consisting of those $w_i$ with $\rho_3(w_i) = j$. The condition $\exists i: \rho_3(w_i) = j$ ensures that $u_j$ is non-empty for every $ 1\leq j \leq |\tau|$, and the second line ensures that $w$ coincides with the concatenation $u_1\dots u_{|\tau|}$. Finally, the condition $w \models (\Gamma_{d-1})_j$ is equivalent to having $\rho_1(u_j) \models \Gamma_{d-1}$ for every $1 \leq j \leq |\tau|$, which by induction implies that $g_{d-1}(\rho_1(u_i)) \neq \star$ for every $1 \leq i\leq |\tau|$. In particular, we must have  $g_d(w) \neq \star$, as required. 

Now suppose that $\tau \in \{\mathrm{inc},\mathrm{dec}\}$. For each $s$, let $i(s)$ and $j(s)$ denote the indices chosen so that the fourth line is satisfied for this particular value of $s$, and write $u^{(s)}$ for the sub-word $w_{i(s)}\dots w_{j(s)}$. 

\begin{claim}\label{claim: I,D decomposes properly}
    For each $s$, the sub-word $u^{(s)}$ satisfies the conditions~\eqref{eq: consistent conditions for I,D - 1} and~\eqref{eq: consistent conditions for I,D - 2}. In particular, the word $w$ can be expressed as a concatenation $w = u_1\dots u_\ell$ for some $u_1,\dots,u_\ell$ chosen so that for every $j$ we have $u_j = u^{(s_j)}$ for some $s_j$. 
\end{claim}
\begin{proof}
    Given an index $s$, we start with the case when $i(s) = j(s) = s$. In this case, we must have $w \models \psi_{=}(i(s),j(s))$, which in turn implies that $\rho_4(w_{i(s)}) = \{L,R\}$. Thus, the word $u^{(s)}$ that  consists of a single letter $w_{i(s)}$ trivially satisfies both of the conditions~\eqref{eq: consistent conditions for I,D - 1} and~\eqref{eq: consistent conditions for I,D - 2}. 

    Now suppose that $i(s) \neq j(s)$, and thus we must have $i(s) < j(s)$.  Thus the fourth line ensures that $w \models \forall t: \psi_{\neq}(i(s),j(s),t)$, which by construction of $\psi_{\neq}$ implies
    \begin{enumerate}
        \item[(i)] $\rho_4(w_{i(s)}) = \{L\}$,
        \item[(ii)] $\rho_4(w_{j(s)}) = \{R\}$, and 
        \item[(iii)] $\rho_4(w_t) = \emptyset$ for every $i(s) < t < j(s)$. 
    \end{enumerate}
    These conditions together are clearly equivalent to the fact that $u^{(s)}$ satisfies both of the conditions~\eqref{eq: consistent conditions for I,D - 1} and~\eqref{eq: consistent conditions for I,D - 2}, which proves the first part of the claim. 

    To prove the second part, observe that the sub-words $u^{(s)}$ cover $w$ as each $w_s$ is trivially contained in $u^{(s)}$. In addition, recall that each $u^{(s)}$ consists of a set of contiguous letters. Hence, it suffices to prove that for each $s \neq s'$, the sub-words $u^{(s)}$ and $u^{(s')}$ are either identical or disjoint. 

    Suppose that $s \neq s'$ are chosen so that $u^{(s)}$ and $u^{(s')}$ are neither identical nor disjoint. Since both of them consists of a contiguous set of letters, we must have $i(s) \in [i(s'),j(s')]$ or $i(s')\in [i(s),j(s)]$. By symmetry, we may assume that the first case occurs. 

    Since $u^{(s')}$ satisfies property~\eqref{eq: consistent conditions for I,D - 1} and $L \in \rho_4(w_{i(s)})$, we must have $i(s) = i(s')$. As $u^{(s)}$ and $u^{(s')}$ are not identical, we must have $j(s) \neq j(s')$, and by symmetry, we may assume that $j(s) < j(s')$. Since $R \in \rho_4(w_{j(s)})$ and $j(s) \in [i(s),j(s')]$ this contradicts the fact that $u^{(s')}$ satisfies property~\eqref{eq: consistent conditions for I,D - 2}. This completes the proof of Claim~\ref{claim: I,D decomposes properly}. 
\end{proof}
    \noindent Claim~\ref{claim: I,D decomposes properly} implies that any $w$ with $w \models \Gamma_d$ and whose profile is in $\{\mathrm{inc},\mathrm{dec}\}$ decomposes as a concatenation $w = u_1\dots u_\ell$, where the sub-words $u_i$ satisfy properties~\eqref{eq: consistent conditions for I,D - 1} and~\eqref{eq: consistent conditions for I,D - 2}. In addition, the last part of both $\psi_{\neq}$ and $\psi_{=}$ implies that we must also have $\rho_1(u_i) \models \Gamma_{d-1}$ for each $i$, whence the inductive assumption implies that $g_{d-1}(\rho_1(u_i)) \neq \star$ for each $i$. When $\tau = \mathrm{dec}$, the last line of $\Gamma_d$ ensures that $\ell\geq 2$: as the letters marked with $L$ coincide with the first letters in the sub-words $u_i$ and there must be at least two such letters $L$, we indeed have $\ell \geq 2$. Thus we conclude that $g_d(w) \neq \star$. The proof of the converse direction follows similarly, we omit the details.
\end{proof}
\subsection{Establishing logical correspondences in $\widetilde{\mathcal{O}}_k^{[d]}$}\label{section: FO and MSO correspondence okd}

Our goal in this subsection is to prove that the function $f_d$ is both a first-order and a monadic second-order homomorphism. This follows immediately from the following slightly more general theorem which applies to formulae, and states that the correspondence between the variables in the two logics is given by the function $\alpha_w$. Recall that a formula is an atomic formula in $\mathrm{MSO}$-$\mathrm{TOTO}$ if it is of the form $x = y$, $x \leq_1 y$, $x \leq_2 y$ or $x \in z$  for some first-order variables $x$ and $y$, and a second-order variable $z$. 

\begin{thm}\label{thm: FO correspondence}
    Let $v^{(f)}_1,\dots,v^{(f)}_{c_1},v^{(b)}_1,\dots,v^{(b)}_{c_2}$ denote variables (either first- or second-order variables), with $v^{(f)}_i$'s corresponding to free and $v^{(b)}_i$'s corresponding to bound variables. Let $\psi(v^{(f)}_1,\dots,v^{(f)}_{c_1})$ denote a monadic second-order formula in $\mathrm{TOTO}$ given in disjunctive normal form, i.e.\
    \begin{align*} 
        \psi(v^{(f)}_1,\dots,v^{(f)}_{c_1}) = Q_1 v^{(b)}_1\dots Q_{c_2} v^{(b)}_{c_2}: \left[\left(p_{1,1} \land \dots \land p_{1,a_1}\right) \lor \dots \lor \left(p_{b,1} \land \dots  \land p_{b,a_b}\right)\right],
    \end{align*}
    where each $p_{i,j}$ is an atomic formula or its negation involving variables in $v^{(f)}_1,\dots,v^{(f)}_{c_1}$,\\$v^{(b)}_1,\dots,v^{(b)}_{c_2}$. There exists a monadic second-order formula $\phi$ of the form 
    \begin{align}  \label{eq: phi formula}
        \phi(\nu^{(f)}_1,\dots,\nu^{(f)}_{c_1}) = Q_1 \nu^{(b)}_1 \dots Q_{c_2} \nu^{(b)}_{c_2}: \left[\left(q_{1,1} \land \dots \land q_{1,a_1}\right) \lor \dots \lor \left(q_{b,1} \land \dots \land q_{b,a_b}\right)\right]
    \end{align}
    satisfying the following properties:
    \begin{enumerate}
        \item[(i)] each variable $\nu^{(f)}_i$ or $\nu^{(b)}_j$ has the same order as the respective variable $v^{(f)}_i$ or $v^{(b)}_j$;
        \item[(ii)] each formula $q_{i,j}$ depends only on the formula $p_{i,j}$; and
        \item[(iii)] for each $a_1,\dots,a_{c_1}$ and $w \in \Sigma_d^*$ satisfying $w \models \Gamma_d$, we have $w \models \phi(a_1,\dots,a_{c_1})$ if and only if $f_d(w) \models \psi(\alpha_w(a_1),\dots,\alpha_w(a_{c_1}))$.
    \end{enumerate}

    In addition, if $\psi$ is a first-order formula, then so is $\phi$.
\end{thm}
\noindent The main ingredient in the proof is constructing the formulae $q_{i,j}$. This turns out to be straightforward if $p_{i,j}$ is of the form $x = y$, $x \in z$, or a negation of one of these two type of formulae. 
\begin{lem} \label{lemma: easy cases = and in}
    Let $w \in \Sigma_d^*$ be chosen so that $w \models \Gamma_d$. 
    \begin{enumerate}
        \item Let $\psi(x,y)$ denote the $\mathrm{TOTO}$-formula given by $x = y$, and $\phi(x',y') \in \mathrm{FO}(\Sigma_{d}^*)$ given by $x' = y'$. We have $w \models \phi(a,b)$ if and only if $f_d(w) \models \psi(\alpha_w(a), \alpha_w(b))$. 
        \item Let $\psi(x,z)$ denote the $\mathrm{MSO}$-$\mathrm{TOTO}$ formula given by $x \in z$, and $\phi(x',z')$ be the monadic second-order formula given by $x' \in z'$. We have $w \models \phi(a,b)$ if and only if $f_d(w) \models \psi(\alpha_w(a), \alpha_w(b))$. Here, we write $\alpha_w(b) = \{\alpha_w(c): c \in b\}$. 
    \end{enumerate}
\end{lem}
\begin{proof}
    Suppose that $w$, $a$ and $b$ are chosen so that $w \models \Gamma_d$ and $w \models \phi(a,b)$. Since the condition $w \models \phi(a,b)$ is equivalent to $a = b$, it immediately follows that $\alpha_w(a) = \alpha_w(b)$, whence we have $f_d(w) \models \psi(\alpha_w(a),\alpha_w(b))$.   
    
    Now suppose that $w \models \Gamma_d$ and $f_d(w) \models \psi(\alpha_w(a),\alpha_w(b))$. This is equivalent to the fact that  the points $\alpha_w(a)$ and $\alpha_w(b)$ are the same point. Since $\alpha_w$ is an injection, we conclude that $a = b$, which in turn implies that $w \models \phi(a,b)$. The proof of the second statement follows analogously. 
\end{proof}
\noindent While we did not use the fact that $\alpha_w$ is a surjection in the proof, the surjectivity is used in subsequent applications. 

We now move on to constructing  first-order formulae $\xi_{s,d}(x,y)$ for $s \in \{1,2\}$ so that for each $w \in \Sigma_d^*$ with $w \models \Gamma_d$, and for each pair of indices $a$ and $b$, we have $w \models \xi_{s,d}(a,b)$ if and only if $\alpha_w(a) \leq_s \alpha_w(b)$. The construction for both $s \in \{1,2\}$ is by induction on the parameter $d$, with $k$ treated as a constant. As these constructions are similar, we present them simultaneously.

\subsubsection*{Base case: construction of $\xi_{s,1}(x,y)$}
Given $\tau \in \mathcal{Q}_k$, let $\tau_1$ denote the increasing permutation $12\dots|\tau|$ and $\tau_2 = \tau$. Define $\xi_{s,1}(x,y)$ to be the first-order formula 
\begin{align*}
 & \bigwedge_{\tau\in\mathcal{Q}_{k}}\left(\rho_{2}\left(w_{1}\right)=\tau\implies \bigvee_{\substack{r,t\in[|\tau|]\\
                    \tau_s(r)\leq\tau_s(t)}}
 (w_x=a_{\tau,r}\land w_y=a_{\tau,t})\right)\\
 & \land\left(\rho_{2}\left(w_{1}\right)=\mathrm{inc}\implies x\leq y\right)\\
 & \land\left(\left(\rho_{2}\left(w_{1}\right)=\mathrm{dec}\land s=1\right)\implies x\leq y\right)\\
 & \land\left(\left(\rho_{2}\left(w_{1}\right)=\mathrm{dec}\land s=2\right)\implies y\leq x\right).
\end{align*}
Note that the second last line only occurs in $\xi_{1,1}(x,y)$, and the last line only occurs in $\xi_{2,1}(x,y)$. 

\begin{lem} \label{lemma: base case of leq_1 correspondence}
    For every $w$ with $w \models \Gamma_1$ we have $w \models \xi_{s,1}(a,b)$ if and only if $\alpha_w(a) \leq_s \alpha_w(b)$. 
\end{lem}
\begin{proof}
    We start with the case when the profile of $w$, which we denote by $\tau$,  satisfies $\tau \in \mathcal{Q}_k$. Since $d = 1$, it follows that $w = a_{\tau,1}\dots a_{\tau,|\tau|}$, and thus $\alpha_{w}(i)$ corresponds to the point $(i,\tau(i))$ in the canonical drawing of the permutation $\tau = f_d(w)$. In particular, we have $\alpha_w(a) \leq_1 \alpha_w(b)$ if and only if $a \leq b$, which by definition is equivalent to the condition $\tau_1(a) \leq \tau_1(b)$. Thus, the first line of $\xi_{1,1}$ guarantees that this condition is equivalent to $w \models \xi_{1,1}(a,b)$, as required. When $s = 2$, the condition $\alpha_w(i) \leq_2 \alpha_w(j)$ is equivalent to $\tau(i) \leq \tau(j)$. Thus an analogous argument completes the proof in this case. 

    We now move on to the case when $\tau = \mathrm{dec}$. As $f_d(w)$ is a decreasing permutation, the construction of the function $\alpha_w$  ensures that we have $\alpha_w(a) \leq_1 \alpha_w(b)$ if and only if $a \leq b$, and $\alpha_w(a) \leq_2 \alpha_w(b)$ if and only if $b \leq a$. These two conditions clearly match the third and fourth lines in the definition of $\xi_{s,1}$. The case when $\tau = \mathrm{inc}$ follows analogously. 
\end{proof}

\subsubsection*{Inductive step: construction of $\xi_{s,d}(x,y)$}

As the final expression of $\xi_{s,d}$ is rather long, we start by defining certain auxiliary formulae. While the motivation behind the construction of  these formulae may not be evident at first sight, the reader may directly move forward to Lemma~\ref{lemma: inductive step of leq_1 correspondence}, whose proof gives some motivation for the choices of these auxiliary formulae. 

Recall that we write $\tau_1 = 12\dots |\tau|$ and $\tau_2 = \tau$, and recall that the formulae $(\xi_{s,d-1})_i$ and $(\xi_{s,d-1})_{i,j}$ are defined as in the paragraph preceding Lemma~\ref{lemma: restriction to sub-words preserves FOS}. Define auxiliary formulae $P_{s,\tau}(x,y)$, $P_{2}(x,y)$, $P_{3}(x,y)$ and $P_{4,s}(x,y)$ by setting 
\begin{align*} 
P_{s,\tau}\left(x,y\right) & :=\left(\tau_{s}\left(\rho_{3}\left(w_{x}\right)\right)<\tau_{s}\left(\rho_{3}\left(w_{y}\right)\right)\right)\\
 & \lor\left(\bigvee_{i=1}^{\left|\tau\right|}\rho_{3}\left(w_{x}\right)=i\land\rho_{3}\left(w_{y}\right)=i\land(\xi_{s,d-1})_i\left(x,y\right)\right),
\end{align*}
\begin{align*}
    P_{2}\left(x,y\right):=\exists i: \left(x\leq i\land i < y\land R\in\rho_{4}\left(w_{i}\right)\right),
\end{align*}
\begin{align*}
    P_{3}\left(x,y\right):=\exists i: \left(y\leq i\land i < x\land R\in\rho_{4}\left(w_{i}\right)\right),    
\end{align*}
and 
\begin{align*}
P_{4,s}\left(x,y\right) & :=\exists i\exists j\forall \ell:\left(i \leq x \land i \leq y \land x \leq j \land y\leq j\right)\\
& \land L \in \rho_4\left(w_i\right) \land R \in \rho_4\left(w_j\right)\\
 & \land \left(i < j \implies \left(R\not\in\rho_{4}\left(w_{i}\right)\land L\not \in\rho_{4}\left(w_{j}\right)\right)\right)\\
 & \land\left(\left(i<\ell\land \ell<j\right)\implies\rho_{4}\left(w_{\ell}\right)=\emptyset\right)\\
 & \land \left(\xi_{s,d-1}\right)_{i,j}\left(x,y\right).
\end{align*}
Given the necessary preparation, define $\xi_{s,d}(x,y)$ to be the first-order formula 
\begin{align*}
 & \bigwedge_{\tau\in\mathcal{Q}_{k}}\left(\rho_{2}\left(w_{1}\right)=\tau\implies P_{s,\tau}\left(x,y\right)\right)\\
 & \land\left(\rho_{2}\left(w_{1}\right)=\mathrm{inc}\implies\left(P_{2}\left(x,y\right)\lor P_{4,s}\left(x,y\right)\right)\right)\\
 & \land\left(\left(\rho_{2}\left(w_{1}\right)=\mathrm{dec}\land s=1\right)\implies\left(P_{2}\left(x,y\right)\lor P_{4,s}\left(x,y\right)\right)\right)\\
 & \land\left(\left(\rho_{2}\left(w_{1}\right)=\mathrm{dec}\land s=2\right)\implies\left(P_{3}\left(x,y\right)\lor P_{4,s}\left(x,y\right)\right)\right).
\end{align*}

\begin{lem} \label{lemma: inductive step of leq_1 correspondence}
    Let $d$ be a positive integer. For every $w$ with $w \models \Gamma_d$ we have $w \models \xi_{s,d}(a,b)$ if and only if $\alpha_w(a) \leq_s \alpha_w(b)$. 
\end{lem}
\begin{proof}
    The proof is by induction on $d$, with the base case $d = 1$ covered by Lemma~\ref{lemma: base case of leq_1 correspondence}. Suppose that the statement holds for $d -1$, and let $w\in \Sigma_d^*$ be chosen so that $w \models \Gamma_d$. Thus $w$ is consistent and decomposes properly. Let $u_1,\dots,u_m$ be chosen so that $w$ decomposes as $w = u_1\dots u_m$. For each $i \in \{1,\dots,m\}$ write $\sigma_i = f_{d-1}(\rho_1(u_i))$, and let $\tau'$ be chosen so that $f_d(w) = \tau'[\sigma_1,\dots,\sigma_m]$, and write $\tau$ for the profile of $w$. We split the proof into three cases based on whether we have $\tau \in \mathcal{Q}_k$, $\tau = \mathrm{dec}$ or $\tau = \mathrm{inc}$. 

    \textbf{Case 1, $\tau \in \mathcal{Q}_k$:} Since $\tau \in \mathcal{Q}_k$, we have $\tau' = \tau$. From the expression of $\xi_{s,d}$ it immediately follows that the condition $w \models \xi_{s,d}(a,b)$ is equivalent to $w \models P_{s,\tau}(a,b)$. For convenience, write $p = \alpha_w(a)$, $q = \alpha_w(b)$, $i_p = \rho_3(w_a)$ and $i_q = \rho_3(w_b)$. Thus $p$ belongs to the interval $\sigma_{i_p}$ and $q$ belongs to the interval $\sigma_{i_q}$ in the inflation 
    \begin{align}\label{eq: important inflation}
        f_d(w) = \tau'[\sigma_1,\dots,\sigma_m]. 
    \end{align}
    First suppose that $i_p \neq i_q$. As $p$ and $q$ belong to different intervals in inflation~\eqref{eq: important inflation}, the relative order of $p$ and $q$ can be determined using the indices $i_p$ and $i_q$. Indeed, the condition $p \leq_1 q$ is equivalent to $i_p < i_q$, and the condition $p \leq_2 q$ is equivalent to $\tau(i_p) < \tau(i_q)$. These conditions are clearly equivalent to the condition on the  top line of $P_{s,\tau}$. 

    Now suppose that $i_p = i_q = i$. In this case, both $p$ and $q$ belong to the same interval $\sigma_{i}$, and the indices $a$ and $b$ belong to the same sub-word $u_i$, for which we have $f_{d-1}(\rho_1(u_i)) = \sigma_i$ and $\sigma_i \in \widetilde{\mathcal{O}}_k^{[d-1]}$. Thus the inductive assumption implies that $\xi_{s,d}$ satisfies the required properties. 

    \textbf{Case 2, $\tau =\mathrm{dec} $:} Note that the equation~\eqref{eq: important inflation} remains valid, but the permutation $\tau'$ is a monotone decreasing permutation, whose (unknown) length we denote by $\ell$. As before, we write $p = \alpha_w(a)$ and $q = \alpha_w(b)$, and choose $i_p$ and $i_q$ to be the indices in $\{1,\dots,\ell\}$ so that $w_a$ and $w_b$ are contained in $u_{i_p}$ and $u_{i_q}$ respectively. In particular, $i_p$ and $i_q$ play the same role as the corresponding indices in the previous case (however, in this case we cannot define them using the function $\rho_3$ as it does not apply to words whose type is $\mathrm{inc}$ or $\mathrm{dec}$). 
    \begin{claim}\label{claim: simple P2 and P3}
        We have $w \models P_{2}(a,b)$ if and only if $i_p < i_q$, and we have $w \models P_{3}(a,b)$ if and only if $i_q < i_p$
    \end{claim}
    \begin{proof}
        First suppose that $w \models P_2(a,b)$, and let $i = i_*$ be chosen so that $a \leq i_*$, $i_* < b$ and $R \in \rho_4(w_{i_*})$. Let $s$ be chosen so that $w_{i_*}$ is contained in $u_s$. Hence, the condition $a \leq i_*$ implies that $i_p \leq s$, and the conditions $i_* < b$ and $R \in \rho_4(w_{i_*})$ imply that $s < i_q$. Note that in the second conclusion, the fact $R \in \rho_4(w_{i_*})$ ensures that $w_{i_*}$ is the last letter in the sub-word $u_s$, which in turn ensures that $w_b$ cannot be contained in $u_s$ as $b > i_*$. Hence, it follows that $i_p \leq s <  i_q$, as required.
        
        If $i_p < i_q$, we may simply choose the value of $i$ in $P_2$ to be the index of the last letter in $u_{i_p}$, which proves the converse. The second assertion follows analogously. 
    \end{proof}
    \begin{claim} \label{claim: P4}
        We have $w \models P_{4,s}(a,b)$ if and only if $i_p = i_q$, and the sub-word $u_{i_p}$ satisfies $\rho_1(u_{i_p}) \models \xi_{s,d-1}(a,b)$.
    \end{claim}
    \begin{proof}
        First suppose that $w \models P_{4,s}(a,b)$, and let $i = i_*$ and $j = j_*$ be chosen so that the remaining formula is satisfied for all values of  $\ell$. The third and fourth lines of $P_{4,s}$ imply that $i_*$ and $j_*$ are the first and last index of some $u_t$, and the first line implies that $w_a$ and $w_b$ must both belong to $u_t$. In particular, we must have $t = i_p = i_q$. Finally, the last line implies that $\rho_1(u_t) \models \xi_{s,d-1}(a,b)$, as required. In order to prove the converse direction,  choose the values of $i$ and $j$ to be the indices of the left and right endpoints of the sub-word $u_{i_p}$. 
    \end{proof}

    Suppose that $a$ and $b$ are chosen so that $p \leq_s q$. Our aim is to show that we have $w \models \xi_{s,d}(a,b)$. First consider the case when $s = 1$ and $i_p \neq i_q$. In this case, $p \leq_s q$ is equivalent to $i_p < i_q$, as $p$ belongs to the interval corresponding to $\sigma_{i_p}$ and $q$ belongs to the interval  corresponding to $\sigma_{i_q}$. By Claim~\ref{claim: simple P2 and P3}, this is equivalent to having $w\models P_2(a,b)$, which in turn implies that $w \models \xi_{1,d}(a,b)$ by considering the third line of $\xi_{1,d}$. 

    Now suppose that $s = 2$ and $i_p \neq i_q$.  In this case, $p \leq_s q$ is equivalent to $i_q < i_p$, as $\tau'$ is a monotone decreasing permutation. By Claim~\ref{claim: simple P2 and P3}, this is equivalent to having $w \models P_3(a,b)$, which due to the last line of $\xi_{s,d}$ implies that $w \models \xi_{2,d}(a,b)$. 

    Finally, suppose that $i_p = i_q = t$. In this case, both $p$ and $q$ belong to the same inflated block $\sigma_{t}$, and both letters $w_a$ and $w_b$ belong to the same sub-word $u_t$. Since $p \leq_s q$, the inductive hypothesis implies that  $\rho_1(u_t) \models \xi_{s,d-1}(a,b)$. In particular, Claim~\ref{claim: P4} implies that $w \models P_{4,s}(a,b)$. Thus, by considering the appropriate line in the construction of $\xi_{s,d}$ it is easy to verify that we have $w \models \xi_{s,d}(a,b)$, as required. 

    Proving that $w \models \xi_{s,d}(a,b)$ implies that $p \leq_s q$ follows by using a similar argument, we omit the details. 

    \textbf{Case 3, $\tau = \mathrm{inc}$:} The analysis is similar to the one presented in Case 2. We omit the details. 
\end{proof}

We are now ready to prove Theorem~\ref{thm: FO correspondence}. 

\begin{proof}[Proof of Theorem~\ref{thm: FO correspondence}]
    For each variable $v_i^{(\beta)}$ with $\beta \in \{f,b\}$ occurring in $\psi$ either as a bound or free variable, we define $\nu_i^{(\beta)}$ to be a variable of the same order, and we say that $\nu_i^{(\beta)}$ is the corresponding variable in $\phi$.  Given a variable $x_{\psi}$ occurring in $\psi$, we write $x_{\phi}$ for the corresponding variable in $\phi$. Define the formulae $q_{i,j}$ as follows:
    \begin{enumerate}
        \item[(i)] if $p_{i,j} := (x_\psi = y_\psi)$, we set $q_{i,j} := (x_{\phi} = y_\phi)$;
        \item[(ii)] if $p_{i,j} := (x_\psi \in  y_\psi)$, we set $q_{i,j} := (x_\phi \in  y_\phi)$;
        \item[(iii)] if $p_{i,j} := (x_\psi \leq_s  y_\psi)$ for $s \in \{1,2\}$, we set $q_{i,j} := \xi_{s,d}(x_\phi, y_\phi)$; and
        \item[(iv)] if $p_{i,j}$ is a negation of an atomic formula, we define $q_{i,j}$ to be the negation of the appropriate formula constructed as above. 
    \end{enumerate}
    Having defined the formulae $q_{i,j}$, define $\phi(\nu_{1}^{(f)},\dots,\nu_{c_1}^{(f)})$ to be the formula given by equation~\eqref{eq: phi formula}. For convenience, write $\psi'(v^{(f)}_1,\dots,v^{(f)}_{c_1},v^{(b)}_1,\dots,v_{c_2}^{(b)})$ for the quantifier-free part of $\psi$, i.e.\  
    \[\psi'(v^{(f)}_1,\dots,v^{(f)}_{c_1},v^{(b)}_1,\dots,v_{c_2}^{(b)}) := (p_{1,1} \land \dots \land p_{1,a_1}) \lor \dots \lor (p_{b,1} \land \dots \land p_{b,a_b}),\]
    and similarly write $\phi'(\nu ^{(f)}_1,\dots,\nu ^{(f)}_{c_1},\nu ^{(b)}_1,\dots,\nu_{c_2}^{(b)})$  for the quantifier-free part of $\phi$. 
    Given an integer $t \in \{0,\dots,c_2\}$ and constants $a_1,\dots,a_{c_1}$ and $b_{1},\dots,b_{t}$, define 
    \begin{align}\label{eq: restricted psi}
        \psi_t(a_1,\dots,a_{c_1},b_1,\dots,b_t) := Q_{t+1} v_{t+1}^{(b)} \dots Q_{c_2} v^{(b)}_{c_2}: \psi'(a_1,\dots,a_{c_1},b_1,\dots,b_t,v_{t+1},\dots,v_{c_2}),
    \end{align}
    and similarly define
    \begin{align}\label{eq: restricted phi}
        \phi_t(a_1,\dots,a_{c_1},b_1,\dots,b_t) := Q_{t+1} \nu_{t+1}^{(b)} \dots Q_{c_2} \nu^{(b)}_{c_2}: \phi'(a_1,\dots,a_{c_1},b_1,\dots,b_t,\nu_{t+1},\dots,\nu_{c_2}). 
    \end{align} 
    Write $X \subseteq \{1,\dots,c_2\}$ for the set of indices $t$ for which we have $Q_t = \forall$. Finally, we say that a constant $c_i$ chosen as the value of the variable $x_i$ has \emph{correct order} if $c_i$ and $x_i$ are first-order constant and variable respectively, or they are second-order constant and variable respectively. 
    
    Having defined the necessary notation, we move on to the main part of the proof. We only prove the reverse direction of the statement as the proof of the other direction follows analogously. That is, we prove that whenever $f_d(w) \models \psi(b_1,\dots,b_{c_1})$ for some constants $b_1,\dots,b_{c_1}$, then $w \models \phi(\alpha_w^{-1}(b_1),\dots,\alpha_w^{-1}(b_{c_1}))$. 
    
    \begin{claim} \label{claim: extending sequences}
        Let $\underline{s}$ be a $[c_1]$-tuple of constants in $f_d(w)$ with correct orders with respect to the variables $v^{(f)}_i$ chosen so that $f_d(w) \models \psi(s_1,\dots,s_{c_1})$. Let $\underline{t}$ denote the $[c_1]$-tuple of constants in $w$ chosen so that $t_i = \alpha_w^{-1}(s_i)$ for each $i \in [c_1]$, and let $\hat{\underline{u}} = (\hat{u}_i: i \in X)$ denote an arbitrary $X$-tuple of constants in $w$ of the correct orders. There exists a $[c_2]$-tuple of constants $\underline{u}$ for which the following properties hold for every $j \in \{0,1,\dots,c_2\}$:
        \begin{enumerate}
            \item[(i)] if $j \in X$, we have $u_j = \hat{u}_j$;
            \item[(ii)] if $j \not \in X$ and $j > 0$, the choice of $u_j$ only depends on the tuple $\underline{t}$, and on the values of the constants in the tuple $(\hat{u}_i: i \in X, i < j)$; and 
            \item[(iii)] we have $f_d(w) \models \psi_j(s_1,\dots,s_{c_1},\alpha_w(u_1),\dots,\alpha_w(u_j))$. 
        \end{enumerate}
    \end{claim}
    
    \noindent Before proving the claim, we note that the choice of $t_i = \alpha_w^{-1}(s_i)$ in the statement exists and is unique as $\alpha_w$ is a bijection. 
    
    \begin{proof}
        We construct the sequence $\underline{u}$ by induction on $j$. The  base case $j = 0$ follows immediately from the assumption that $f_d(w) \models \psi(s_1,\dots,s_{c_1})$ as $\psi_0 = \psi$, and since only the third condition is applicable in this case. Suppose that the statement holds for some $j = i$ with $i \in \{0,\dots,c_2-1\}$. We split the proof into cases based on whether we have $i + 1 \in X$ or $i + 1\not \in X$. 

        \textbf{Case 1, $i + 1 \in X$:} In this case, we set $u_{i+1} = \hat{u}_{i+1}$, and thus the first condition is automatically satisfied. Applying the inductive assumption with the third condition we conclude that 
        \begin{align*}
            f_d(w) \models \forall v_{i+1} Q_{i+2} v_{i+2} \dots Q_{c_2} v_{c_2}: \psi'(s_1,\dots,s_{c_1},\alpha_w(u_1),\dots,\alpha_w(u_i),v_{i+1},\dots,v_{c_2}). 
        \end{align*}
        Here, we are using the fact that $Q_{i+1} = \forall$, which follows from $i+1 \in X$. By choosing $v_{i+1} := \alpha_w(u_{i+1})$, we conclude that
        \begin{align*}
            f_d(w) \models Q_{i+2} v_{i+2} \dots Q_{c_2} v_{c_2}: \psi'(s_1,\dots,s_{c_1},\alpha_w(u_1),\dots,\alpha_w(u_{i+1}),v_{i+2},\dots,v_{c_2}).
        \end{align*}
         This coincides with the statement $f_d(w) \models \psi_{i+1}(s_1,\dots,s_{c_1},\alpha_w(u_1),\dots,\alpha_w(u_{i+1}))$, and thus completes the proof of Case 1. 

        \textbf{Case 2, $i + 1 \not \in X$:} In this case, the inductive assumption implies that
        \begin{align*}
            f_d(w) \models \exists v_{i+1} Q_{i+2} v_{i+2} \dots Q_{c_2} v_{c_2}: \psi'(s_1,\dots,s_{c_1},\alpha_w(u_1),\dots,\alpha_w(u_i),v_{i+1},\dots,v_{c_2}). 
        \end{align*}
        Let $\gamma$ denote a constant chosen so that 
        \begin{align*}
            f_d(w) \models Q_{i+2} v_{i+2} \dots Q_{c_2} v_{c_2}: \psi'(s_1,\dots,s_{c_1},\alpha_w(u_1),\dots,\alpha_w(u_i),\gamma, v_{i+2},\dots,v_{c_2}). 
        \end{align*}
        Note that $\gamma$ only depends on the tuple $\underline{s}$, and the choices of $\alpha_w(u_1),\dots,\alpha_w(u_i)$. Let $u_{i+1} = \alpha_w^{-1}(\gamma)$, and note that $\alpha_w^{-1}(\gamma)$ is well-defined since $\alpha_w$ is a bijection. Since $\alpha_w$ is a bijection, the inductive assumption implies that the tuple $(\alpha_w(u_1),\dots,\alpha_w(u_{i}))$ only depends on the tuple $\underline{t}$ and the values of $(\underline{\hat{u}}_j: j \in X, j \leq i)$. Thus $u_{i+1}$ chosen in this way satisfies both the second and the third condition, as required. 
    \end{proof}

    Fix a sequence $s_1,\dots,s_{c_1}$ and a $X$-tuple $\hat{\underline{u}} = (\hat{u}_i: i \in X)$ for which we have $f_d(w) \models \psi(s_1,\dots,s_{c_1})$, and let $\underline{t}$ and $\underline{u}$ be constructed as in the statement of Claim~\ref{claim: extending sequences}. Applying Claim~\ref{claim: extending sequences} with $j = c_2$ yields that 
    \[f_d(w) \models \psi'(s_1,\dots,s_{c_1},\alpha_w(u_1),\dots,\alpha_w(u_{c_2})).\]
    Recall that $s_i = \alpha_w(t_i)$ for each $i$. Given $i$ and $j$, let $x_\psi$  and $y_\psi$ denote the variables occurring in the formula $p_{i,j} := p_{i,j}(x_\psi,y_\psi)$. By construction, we have $q_{i,j} := q_{i,j}(x_\phi,y_\phi)$, and depending on the type of the atomic formula used to construct $p_{i,j}$, Lemma~\ref{lemma: easy cases = and in} or Lemma~\ref{lemma: inductive step of leq_1 correspondence} respectively implies that for all $a$ and $b$ we have $w \models q_{i,j}(a,b)$ if and only if $f_d(w) \models p_{i,j}(\alpha_w(a),\alpha_w(b))$. Since $\alpha_w$ is a bijection, it follows that 
    \[w \models \phi'(t_1,\dots,t_{c_1},u_1,\dots,u_{c_2}).\]

    Since the conclusion holds for an arbitrary $X$-tuple $\hat{\underline{u}}$, and the values of $u_j$ with $j \not \in X$ only depend on the values of $\hat{\underline{u}}_i$ with $i < j$, we may use these values to assign correct choices for the variables attached with $\exists$-quantifiers. Hence, we conclude that 
    \[w \models \phi(t_1,\dots,t_{c_1}).\]
    Since $t_i = \alpha_w^{-1}(s_i)$ and $\alpha_w$ is a bijection, this completes the proof. Proof of the other direction, i.e.\ of the fact that $w \models \phi(a_1,\dots,a_{c_1})$ implies $f_d(w) \models \psi(\alpha_w(a_1),\dots,\alpha_w(a_{c_1}))$ follows analogously - we omit the details.  
    
    If $\psi$ is a first-order formula, then all bound and free variables in $\psi$ are first-order variables, and none of the $p_{i,j}$'s is of the form $x_\psi \in y_\psi$ or its negation for some variables $x_\psi$ and $y_\psi$. Thus the same must be true for the variables and formulae $q_{i,j}$ in $\phi$, whence $\phi$ must be a first-order formula  as well. This completes the proof of Theorem~\ref{thm: FO correspondence}. 
\end{proof}
\subsection{Restricting to the bijection}\label{section: restricting to bijection}
Recall that the left-greedy order $<_{L}$ induces a total order on the set $D_{d,k}(\sigma)$ of inflations inducing $\sigma$, and the least element in $D_{d,k}(\sigma)$ is referred to as the left-greedy decomposition of $\sigma$. 

Our goal in this subsection is to construct a first-order sentence $\Delta_d$  so that for each $\sigma \in \widetilde{\mathcal{O}}_{k}^{[d]}$, there exists a unique $w \in \Sigma_{d}^*$ for which we have $w \models \Delta_d$ and $f_d(w) = \sigma$. The idea behind the inductive construction can be described as follows. Given $\sigma \in \widetilde{\mathcal{O}}_k^{[d]}$, suppose that $\tau' \in \widetilde{\mathcal{O}}_k$ and $\sigma_1,\dots,\sigma_\ell \in \widetilde{\mathcal{O}}_k^{[d-1]}$ are chosen so that $\tau'[\sigma_1,\dots,\sigma_\ell]$ is the left-greedy decomposition of $\sigma$. Let $u_1',\dots,u_\ell' \in \Sigma_{d-1}^*$ be the unique words chosen so that $f_{d-1}(u_i') = \sigma_i$ and $u_i' \models \Delta_{d-1}$ for each $i$. Let $\tau \in \mathcal{Q}_k \cup \{\mathrm{inc},\mathrm{dec}\}$ to be chosen so that $\tau = \tau'$ whenever $\tau'\in \mathcal{Q}_k$, $\tau = \mathrm{inc}$ whenever $\tau'$ is monotone increasing, and $\tau= \mathrm{dec}$ otherwise. For each $i$, let $u_i$ denote the unique lift of $u_i'$ in $\Sigma_d^*$ satisfying the following conditions:
\begin{enumerate}
    \item[(i)] $\rho_1(u_i)$ = $u_i'$;
    \item[(ii)] $\rho_2((u_i)_j) = \tau$ for each $j$;
    \item[(iii)] if $\tau \in \mathcal{Q}_k$, we have $\rho_3((u_i)_j) = i$ for each $j$; and
    \item[(iv)] if $\tau \in \{\mathrm{inc},\mathrm{dec}\}$, we have $L \in \rho_4((u_i)_j)$ if and only if $j = 1$, and we have $R \in \rho_4((u_i)_j)$ if and only if $j = |u_i|$. 
\end{enumerate}
It is straightforward to check that such $u_i$ always exists and is unique. Let $w^* = u_1\dots u_\ell$, and note that the construction of $w^*$ ensures that $f_d(w^*) = \sigma$. Our goal is to construct $\Delta_d \in \mathrm{FOS}$ for which we have $w \models \Delta_d$ and $f_d(w) = \sigma$ if and only if $w = w^*$. 

As the inductive construction of $\Delta_d$ is rather lengthy, we split it into smaller steps. We start with two quick observations. First, we prove that there exists a first-order sentence $\chi_{d,k}$ such that for every $\sigma \in \mathfrak{S}_n$ we have $\sigma \models \chi_{d,k}$ if and only if $\sigma \in \widetilde{\mathcal{O}}_{k}^{[d]}$. Our proof is based on the following result of Vatter.  

\begin{thm} \label{thm: finitely based}[Theorem 6.2 in \cite{vatter2011small}]
Let $\mathcal{F}$ be a small permutation class. Then $\mathcal{F}$
is finitely based, i.e.\ there exists a finite collection of permutations
$\varphi_{1},\dots,\varphi_{L}$ so that $\mathcal{F}=\bigcap_{i=1}^{L}\mathrm{Av}\left(\varphi_{i}\right)$. 
\end{thm}

As our second step, we prove that for every permutation $\tau$  there exists a $\psi_{\tau} \in \mathrm{FO}$ which can be used to check if $\sigma$ can be written as an inflation of the form $\tau[\sigma_1,\dots,\sigma_\ell]$, with $\sigma_i \in \widetilde{\mathcal{O}}_k^{[d-1]}$, and so that each $\sigma_i$ is non-empty. Furthermore, the formula $\psi_{\tau}$  constructed has convenient structural properties, in the sense that it involves variables whose values correspond to the extremal points in each $\sigma_i$. 

As a notational convenience, write $p \in [a_1,b_1]_1\times [a_2,b_2]_2$ if and only if $p$ satisfies the conditions $a_1 \leq_1 p$, $p \leq_1 b_1$, $a_2 \leq_2 p$ and $p \leq_2 b_2$. In particular, this condition can be expressed as a first-order formula, for which we use the same shorthand $p \in [a_1,b_1]_1\times [a_2,b_2]_2$. 

\begin{proposition} \label{prop: small classes FOS classifiable}
    Let $\mathcal{F}$ be a finitely based permutation class. There exists a first-order formula $\upsilon_{\mathcal{F}}(x_1,x_2,y_1,y_2)$ for which $\sigma \models \upsilon_{\mathcal{F}}(a_1,a_2,b_1,b_2)$ if and only if the sub-permutation $\tau$ of $\sigma$ is contained in $\mathcal{F}$, where $\tau$ is induced by the points $p$ in $\sigma$ which satisfy the condition $p \in [a_1,b_1]_1 \times [a_2,b_2]_2$. 

    In particular, there exists a first-order sentence $\hat{\upsilon}_{\mathcal{F}}$ for which we have  $\sigma \models \hat{\upsilon}_{\mathcal{F}}$ if and only if $\sigma \in \mathcal{F}$. 
\end{proposition}
    
\begin{proof}
    Choose $\varphi_1,\dots,\varphi_L$ so that $\mathcal{F} = \bigcap_{i=1}^{L} \mathrm{Av}(\varphi_i)$.  For each $i \in [L]$, let $A_i(x_1,x_2,y_1,y_2)$ denote the first-order formula given by 
    \begin{align*}
     \exists z_1,\dots,z_{|\varphi_i|} :& \bigwedge_{j=1}^{|\varphi_i|} \left(z_j \in [x_1,y_1]_1 \times [x_2,y_2]_2\right)\\
    & \land \bigwedge_{j = 1}^{|\varphi_i| - 1} \bigwedge_{\ell =  j + 1}^{|\varphi_i|} \left(z_j <_1 z_\ell \land \left(z_j <_2 z_\ell \iff \varphi_i(j) < \varphi_i(\ell)\right)\right),
    \end{align*}
    and set 
    \[\upsilon_{\mathcal{F}}(x_1,x_2,y_1,y_2) := \bigwedge_{i=1}^{L} \left(\neg A_i(x_1,x_2,y_1,y_2)\right).\]

    To check that $\upsilon_\mathcal{F}(x_1,x_2,y_1,y_2)$ satisfies the required properties, suppose that we have $\sigma \models A_i(a_1,a_2,b_1,b_2)$  for some $i$. The top line of $A_i$ implies that each of the points $z_1,\dots,z_{|\varphi_i|}$ is contained in $[a_1,b_1]_1 \times [a_2,b_2]_2$, while the second line implies that they form a sub-permutation that coincides with $\varphi_i$. Thus we have $\sigma \models \upsilon_{\mathcal{F}}(x_1,x_2,y_1,y_2)$ if and only if the sub-permutation $\tau$ of $\sigma$ consisting of the points in $[x_1,y_1]_1 \times [x_2,y_2]_2$ does not contain $\varphi_i$ as a sub-permutation for any $1 \leq i \leq L$, as required. The second part of the statement follows by defining $A_i$ to be the first-order sentence obtained from $A_i(x_1,x_2,y_1,y_2)$ by deleting the second part of the top line, and setting $\hat{\upsilon}_\mathcal{F} := \bigwedge_{i=1}^L \left(\neg A_i\right)$. 
\end{proof}
\noindent We remark that each of the permutation classes $\widetilde{\mathcal O}_k^{[d]}$ and hence each $\mathrm{Geom}(M)_{d,k}$ is finitely based by an easy application of \cite[Theorem~4.4]{albert2015inflations}.
\begin{lem}
    Let $\tau$ be a permutation of length $\ell$. There exists a first-order formula $\psi_{\tau}(x_{i,u,v})$ with variables $x_{i,u,v}$ for $i \in \{1,\dots,\ell\}$ and $u,v \in \{1,2\}$  which satisfies the following conditions. We have $\sigma \models \psi_{\tau}(a_{i,u,v})$ for some constants $a_{i,u,v}$ if and only if $\sigma = \tau[\sigma_1,\dots,\sigma_\ell]$ for some non-empty permutations $\sigma_i \in \widetilde{\mathcal{O}}_k^{[d-1]}$, and for each $1 \leq i\leq \ell$, the points $a_{i,u,1}$ and $a_{i,u,2}$ are respectively the least and greatest points of $\sigma_i$ in the order $\leq_u$ for $u \in \{1,2\}$, and the sub-interval corresponding to $\sigma_i$ consists of those points $p$ which satisfy the condition $p \in [a_{i,1,1},a_{i,1,2}]_1 \times [a_{i,2,1},a_{i,2,2}]_2$. 
\end{lem}
\begin{proof}
    Let $\upsilon_{d-1,k}(x_1,x_2,y_1,y_2)$ denote the first-order formula constructed in the statement of Proposition~\ref{prop: small classes FOS classifiable} with $\mathcal{F} = \widetilde{\mathcal{O}}_k^{[d-1]}$. 
    We claim that one can take $\psi_{\tau}(x_{i,u,v})$ to be the formula 
    \begin{align*}
         &\forall z: \bigvee_{i=1}^{\ell} \left(z \in [x_{i,1,1},x_{i,1,2}]_1 \times [x_{i,2,1},x_{i,2,2}]_2\right) \\
         & \land \bigwedge_{i,i'=1}^{\ell}  \bigwedge_{u,u'=1}^{2} \bigwedge_{v,v'=1}^{2} \left((i < i' \implies x_{i,u,v} <_1 x_{i',u',v'}) \land (\tau(i) < \tau(i') \implies x_{i,u,v} <_2 x_{i',u',v'})\right) \\
         & \land \bigwedge_{i=1}^{\ell} \upsilon_{d-1,k}(x_{i,1,1},x_{i,2,1},x_{i,1,2},x_{i,2,2})
\\ &\land\bigwedge_{i=1}^{\ell}\bigwedge_{u,v=1}^{2}
\left(x_{i,u,v}\in
[x_{i,1,1},x_{i,1,2}]_1\times
[x_{i,2,1},x_{i,2,2}]_2\right).
    \end{align*}
    Suppose that $\sigma \models \psi_{\tau}(a_{i,u,v})$, and for each $i \in \{1,\dots,\ell\}$ let $\sigma_i$ denote the sub-permutation induced by the points in $[a_{i,1,1},a_{i,1,2}]_1 \times [a_{i,2,1},a_{i,2,2}]_2$. The second last line of $\psi_{\tau}$ implies that we have $\sigma_i \in \widetilde{\mathcal{O}}_k^{[d-1]}$ for each $i$, and the last line implies that the points $a_{i,u,v}$ are chosen consistently so that each $\sigma_i$ is indeed non-empty. The first line implies that each point $p \in \sigma$ belongs to some $\sigma_i$. In addition, the second line implies that each $\sigma_i$ forms a sub-interval of $\sigma$, and these intervals are ordered so that $\sigma = \tau[\sigma_1,\dots,\sigma_\ell]$. The converse direction follows analogously, we omit the details. 
\end{proof}

As the first step towards the proof of Theorem~\ref{thm:Main O_K}, we prove the following Lemma which allows us to perform the construction in the inductive step for those $w$ whose profile is in $\mathcal{Q}_k$. 

\begin{lem}\label{lemma: handling Q_ks for bijection}
    \begin{enumerate}
        \item Let $\tau$ be a permutation in $\widetilde{\mathcal{O}}_k$. There exists a first-order formula $\psi_{\tau,<_L}(x_{i,u,v})$ with variables $x_{i,u,v}$ for $i \in \{1,\dots,|\tau|\}$ and $u,v \in \{1,2\}$ so that $\sigma \models \psi_{\tau,<_L}(a_{i,u,v})$ if and only if the left-greedy decomposition of $\sigma$ is given by $\tau[\sigma_1,\dots,\sigma_{|\tau|}]$, where the points $a_{i,u,1}$ and $a_{i,u,2}$ are respectively the least and greatest points of $\sigma_i$ in the order $\leq_u$ for $u \in \{1,2\}$, and each $\sigma_i$ consists of the points in $[a_{i,1,1},a_{i,1,2}]_1 \times [a_{i,2,1},a_{i,2,2}]_2$. 
        \item Let $\hat{\Delta}$ be a first-order sentence in $\Sigma_{d-1}^* $. Given $\tau \in \mathcal{Q}_k$,  there exists a first-order formula $\Delta_\tau(\nu_{i,u,v})$ with $i \in \{1,\dots,|\tau|\}$, and $u,v\in\{1,2\}$ for which $w \models \Delta_\tau(c_{i,u,v})$ if and only if the following conditions hold:
        \begin{enumerate}
            \item $f_d(w) \neq \star$ and $w$ has profile $\tau$;
            \item if $w$ decomposes as $w = u_1\dots u_{|\tau|}$, then we have $\rho_1(u_i) \models \hat{\Delta}$ for each $i$; and
            \item the left-greedy decomposition of $f_d(w)$ is given by $\tau[f_{d-1}(\rho_1(u_1)),\dots,f_{d-1}(\rho_1(u_{|\tau|}))]$, and we have $f_d(w) \models \psi_{\tau,<_L}(\alpha_w(c_{i,u,v}))$. 
        \end{enumerate}
    \end{enumerate}

    In~\cite{ourotherpaper}, we also require an analogous version of Lemma~\ref{lemma: handling Q_ks for bijection} to determine the end-points of the first $B$ blocks when the profile $\tau$ is a monotone permutation of arbitrary length. As the construction is somewhat similar to the proof of Lemma~\ref{lemma: handling Q_ks for bijection}, we state and prove this result in the Appendix. 
    \begin{proof}
    \textbf{Proof of statement 1:} Given a positive integer $t$, define the auxiliary first-order formula $\psi_{\geq,t}(x_1,\dots,x_t,y_1,\dots,y_t)$ as
        \begin{align*}
            \bigwedge_{j = 1}^{t} \left(\left(\bigwedge_{i=1}^{j-1} x_i = y_i\right) \implies  y_j \leq_1 x_j\right).
        \end{align*}
        Note that the first term corresponding to $j = 1$ is simply $y_1 \leq_1 x_1$. Let $S_{\tau} = \{\tau' \in \widetilde{\mathcal{O}}_k: |\tau'| < |\tau|\}$. Define $\psi_{\tau,<_L}(x_{i,u,v})$ to be the formula 
        \begin{align*}
            & \psi_{\tau}(x_{i,u,v}) \\
            & \land \bigwedge_{\tau' \in S_{\tau}} \left(\forall_{i \in [|\tau'|], u,v\in \{1,2\}} y_{i,u,v}: (\neg \psi_{\tau'}(y_{i,u,v}))\right) \\
            & \land \left(\forall_{i \in [|\tau|], u,v\in \{1,2\}} y_{i,u,v}: \psi_{\tau}(y_{i,u,v}) \implies \psi_{\geq, |\tau|}(x_{i,1,2},y_{i,1,2})\right).
        \end{align*}
        Suppose that $\sigma \models \psi_{\tau,<_L}(a_{i,u,v})$. For each $i$, let  $\sigma_i$ denote the sub-permutation of $\sigma$ consisting of points $p$ in $[a_{i,1,1},a_{i,1,2}]_1 \times [a_{i,2,1},a_{i,2,2}]_2$. The first line of $\psi_{\tau,<_L}$ implies that $\sigma = \tau[\sigma_1,\dots,\sigma_{|\tau|}]$, and our goal is to show that $\tau[\sigma_1,\dots,\sigma_{|\tau|}]$ is the left-greedy decomposition of $\sigma$. 

        Suppose that $\tau' \in \widetilde{\mathcal{O}}_k$ and $\sigma_1',\dots,\sigma'_{|\tau'|} \in \widetilde{\mathcal{O}}_k^{[d-1]}$ are chosen so that $\sigma = \tau'[\sigma_1',\dots,\sigma'_{|\tau'|}]$, and suppose for the sake of contradiction we have  $\tau'[\sigma'_1,\dots,\sigma'_{|\tau'|}]<_{L}\tau[\sigma_1,\dots,\sigma_{|\tau|}]$. First consider the case when $|\tau'|< |\tau|$. Let $b_{i,u,v}$ be chosen so that $\sigma'_i$ consists of the points in $[b_{i,1,1},b_{i,1,2}]_1\times[b_{i,2,1},b_{i,2,2}]_2$. Note that such $b_{i,u,v}$ exist, as each $\sigma_i'$ is a sub-interval of $\sigma$. Choosing $y_{i,u,v} = b_{i,u,v}$ for this particular choice of $\tau'$ for each $i \in [|\tau'|]$  on the second line of $\psi_{\tau,<_L}$ gives the desired contradiction. 

        Now suppose that $|\tau'| = |\tau|$. Proposition~2.3 in~\cite{albert2015inflations} implies that for any $\sigma \in \widetilde{\mathcal{O}}_{k}^{[d]}$, the set $T(\sigma)$ of possible profiles $\tau'$  of decompositions of $\sigma$ as $\sigma = \tau'[\sigma_1,\dots,\sigma_\ell]$ with $\sigma_i \in \widetilde{\mathcal{O}}_{k}^{[d-1]}$ has a unique minimum element under sub-pattern containment. In particular, since $\tau$ is an element of minimal size in $T(\sigma)$, we must have $\tau \leq \tau'$ and thus $\tau = \tau'$. 
        
        Let $t$ be chosen so that $|\sigma_i| = |\sigma_i'|$ for every $i \leq t-1$ and $|\sigma_t'| > |\sigma_t|$. Let $b_{i,u,v}$ be chosen as above, and choose $y_{i,u,v} = b_{i,u,v}$ on the last line of $\psi_{\tau, <_L}$. Since $\sigma \models \psi_{\tau}(b_{i,u,v})$, we must have $\psi_{\geq,|\tau|}(a_{i,1,2},b_{i,1,2})$. In particular, we have 
        \begin{align} \label{eq: leftgreedy contradiction}
            \sigma \models \left(\left(\bigwedge_{i=1}^{t-1} a_{i,1,2} = b_{i,1,2}\right) \implies b_{t,1,2} \leq_1 a_{t,1,2}\right). 
        \end{align}
        Note that for each $i$, $a_{i,1,2}$ is the largest point in $\sigma_i$ with respect to $\leq_1$, and analogously $b_{i,1,2}$ is the largest point in $\sigma_i'$. Thus, the condition $a_{i,1,2} = b_{i,1,2}$ is equivalent to 
        \begin{align*}
            \sum_{j=1}^{i} |\sigma_j| = \sum_{j=1}^i |\sigma_j'|
        \end{align*}
        and the condition $b_{i,1,2} \leq_1 a_{i,1,2}$ is equivalent to\begin{align*}
            \sum_{j=1}^{i} |\sigma_j| \geq \sum_{j=1}^i |\sigma_j'|. 
        \end{align*}
        In particular, the equation~\eqref{eq: leftgreedy contradiction} implies that $|\sigma_i| = |\sigma_i'|$ for every $i \leq t-1$ and $|\sigma_t'| \leq |\sigma_t|$, which gives the required contradiction. Proving the converse direction follows analogously, we omit the details. 

        \textbf{Proof of statement 2:} Given $\tau \in \mathcal{Q}_k$ we write $\ell = |\tau|$. Define the first-order formula $\chi_{\tau}(x_{i,u,v},x_i)$ involving the variables  $x_{i,u,v}$  and $x_i$ for $i\in [\ell]$, $u,v \in \{1,2\}$ by setting
        \begin{align} \label{eq: importan definition}
            \psi_{\tau,<_L}(x_{i,u,v}) \land \bigwedge_{i=1}^\ell \left(x_i \in [x_{i,1,1},x_{i,1,2}]_1\times [x_{i,2,1},x_{i,2,2}]_2\right).
        \end{align}
        Let $\hat{\Delta}_{\tau}(\nu_{i,u,v},\nu_i)$ denote the first-order formula in $\Sigma_{d}^*$ obtained by applying Theorem~\ref{thm: FO correspondence}  to $\chi_{\tau}(x_{i,u,v},x_i)$, where $\nu_i$ and $\nu_{i,u,v}$ are the variables corresponding to $x_i$ and $x_{i,u,v}$ respectively. Let $\Delta_{\tau}(\nu_{i,u,v})$ denote the first-order formula in $\Sigma_{d}^*$ given by
        \begin{align} \label{eq: Final construction}
            &\Gamma_d\land(\forall\nu:\rho_2(w_\nu)=\tau)\land\bigwedge_{i=1}^{\ell}\hat\Delta_i\\
            &\land\forall\nu_1\cdots\forall\nu_\ell:
            \left(\left(\bigwedge_{i'=1}^{\ell}\rho_3(w_{\nu_{i'}})=i'\right)
            \implies\hat\Delta_\tau(\nu_{i,u,v},\nu_i)\right),
        \end{align}
        where, as before, $\hat{\Delta}_i$ denotes the first-order sentence $\hat{\Delta}$ applied to $\rho_1(u_i)$, where $u_i$ denotes the  sub-word of $w$ which consists of those letters $w_j$ which satisfy $\rho_3(w_j) = i$. 

        To check that $\Delta_\tau$ satisfies the required properties, suppose that $w\models \Delta_\tau(c_{i,u,v})$. As the first two conditions in the expression of $\Delta_\tau$ imply that $w \models \Gamma_d$ and $\rho_2(w_i) = \tau$ for all $i$, it follows that $w$ decomposes as $w = u_1\dots u_\ell$, where $u_i$ is the sub-word of $w$ which consists of all letters $w_j$ with $\rho_3(w_j) = i$. In addition, the condition $\hat{\Delta}_i$ ensures that for each $i$ we have $\rho_1(u_i) \models \hat{\Delta}$. Thus the conditions (a) and (b) are satisfied. 

        For each $i \in \{1,\dots,\ell\}$ and $u,v\in \{1,2\}$ write $b_{i,u,v} = \alpha_w(c_{i,u,v})$, set $\sigma_i := f_{d-1}(\rho_1(u_i))$, and choose $\tau' \in \widetilde{\mathcal{O}}_k$ and $\sigma_1',\dots,\sigma_{|\tau'|}'$ so that the left-greedy decomposition of $f_d(w)$ is given by $\tau'[\sigma_1',\dots,\sigma_{|\tau'|}']$. Since each $u_i$ is non-empty, we may choose $\nu_i$ in $u_i$ for every $i$. The last condition of $\Delta_\tau$ and  Theorem~\ref{thm: FO correspondence} imply that $f_d(w) \models \psi_{\tau,<_L}(b_{i,u,v})$. Hence we must have $\tau' = \tau$, and for each $i \in \{1,\dots,\ell\}$, the permutation $\sigma_i'$ consists of points in $[b_{i,1,1},b_{i,1,2}]_1 \times [b_{i,2,1},b_{i,2,2}]_2$. On the other hand, varying a single $\nu_i$ over the indexes of $u_i$ while keeping the other choice of $\nu_j$ fixed, it is easy to see that the second condition in the construction~\eqref{eq: importan definition}, the condition $\rho_3(w_{\nu_i})= i$ in the construction~\eqref{eq: Final construction} and Theorem~\ref{thm: FO correspondence} together imply that every point corresponding to the letters of $u_i$ belongs to the rectangular block corresponding to $\sigma_i'$. Since $\alpha_w$ is a bijection, the blocks $B_i$ consisting of points corresponding to the letters of the word $u_i$ partition $f_d(w)$. Since the blocks corresponding to $\sigma_i'$ also partition $f_d(w)$, the inclusions must be equalities, and thus $B_i = \sigma_i'$ for every $i$. 
        
        Thus, we conclude that $\alpha_w(j) \in [b_{i,1,1},b_{i,1,2}]_1 \times [b_{i,2,1},b_{i,2,2}]_2$ if and only if $\rho_3(w_j) = i$.  Equivalently, the letter $w_j$ belongs to the sub-word $u_i$ if and only if the corresponding point $\alpha_w(j)$ belongs to $\sigma_i'$. By the recursive construction of the functions $\alpha_w$ we have $f_{d-1}(\rho_1(u_i)) = \sigma_i'$ for each $i \in \{1,\dots,\ell\}$. This verifies that the final condition (c) is also satisfied. 

        Proof of the converse direction follows analogously. 
    \end{proof}
\end{lem}
It is not too hard to prove an analogue of Lemma~\ref{lemma: handling Q_ks for bijection} for a given (monotone) permutation $\tau$. In fact, we present a proof of a slightly more general variant of Lemma~\ref{lemma: handling Q_ks for bijection} for monotone permutations in the Appendix, as it is needed for the purposes of \cite{ourotherpaper}. However, this would not be an useful approach for proving Theorem~\ref{thm:Main O_K}, as the set $\widetilde{\mathcal{O}}_k$ contains infinitely many permutations. Due to this reason we may not simply combine these sentences for every $\tau \in \widetilde{\mathcal{O}}_k$, as such an expression would no longer be a first-order sentence. In contrast, we may use this strategy to handle all the profiles $\tau \in \mathcal{Q}_k$ individually, as  $\mathcal{Q}_k$ is a finite set. 

Instead, we need to develop an approach which allows us to deal with monotone profiles of arbitrary length simultaneously. The approach builds on the following observation of Albert, Ruskuc and Vatter. 
\begin{lem}[Special case of Proposition 2.4 in \cite{albert2015inflations}]\label{lemma: albert left greedy}
    Given $\sigma \in \widetilde{\mathcal{O}}_k^{[d]}$, the expression $\sigma_1 \oplus \dots \oplus \sigma_\ell$ with $\sigma_i \in \widetilde{\mathcal{O}}_k^{[d-1]}$ is not the left-greedy decomposition of $\sigma$ if and only if there exists $i \in \{1,\dots,\ell-1\}$ for which $\sigma_i \oplus \sigma_{i+1} = \theta_1 \oplus \theta_2$ for some $\theta_1,\theta_2 \in \widetilde{\mathcal{O}}_k^{[d-1]}$ with $|\theta_1| > |\sigma_{i}|$ (allowing the possibility that $\theta_2$ is the empty permutation). 

    An analogous statement holds when all the instances of $\oplus$ are replaced with $\ominus$. 
\end{lem}
We start with the following preliminary lemma. 
\begin{lem}\label{lemmma: oplus}
    There exists a first-order formula $\psi_{\oplus}(x_{i,u,v})$ with variables $x_{i,u,v}$ for $i,u,v \in \{1,2\}$ which satisfies the following conditions. For $i \in \{1,2\}$, let $\theta_i$ denote the sub-permutation consisting of the points in $[a_{i,1,1},a_{i,1,2}]_1 \times [a_{i,2,1},a_{i,2,2}]_2$. 
    \begin{enumerate}
        \item [(i)] Whenever $\sigma \models \psi_{\oplus}(a_{i,u,v})$, we have $\sigma = \theta_1 \oplus \theta_2$, and $\theta_i \in \widetilde{\mathcal{O}}_k^{[d-1]}$
        \item [(ii)] Given any decomposition of $\sigma$ as $\sigma = \theta_1 \oplus \theta_2$ for non-empty $\theta_i \in \widetilde{\mathcal{O}}_k^{[d-1]}$, let $a_{i,u,v}$ denote the extremal points in $\theta_i$ chosen as above. Then $\sigma \models \psi_{\oplus}(a_{i,u,v})$. 
    \end{enumerate}
    
    Similarly, there exists a first-order formula $\psi_{\ominus}(x_{i,u,v})$ satisfying the appropriate conditions with each instance of $\oplus$ replaced with $\ominus$. 
\end{lem}
\begin{proof}
    We claim that we may take $\psi_{\oplus}(x_{i,u,v})$ to be the first-order formula 
    \begin{align*}
        & \left( \forall z: (z \in [x_{1,1,1},x_{1,1,2}]_1 \times [x_{1,2,1},x_{1,2,2}]_2 \lor z \in [x_{2,1,1},x_{2,1,2}]_1 \times [x_{2,2,1},x_{2,2,2}]_2) \right) \\
        \land & \left( \forall z_1 \forall z_2: \left( \bigwedge_{i=1}^2 z_i \in [x_{i,1,1},x_{i,1,2}]_1 \times [x_{i,2,1},x_{i,2,2}]_2\right) \implies \left((z_1 <_1 z_2) \land (z_1 <_2 z_2)\right)\right) \\
        \land & \bigwedge_{i=1}^2 \upsilon_{d-1,k}(x_{i,1,1},x_{i,2,1},x_{i,1,2},x_{i,2,2})\\ 
        \land &\bigwedge_{i,u,v=1}^{2}
        \left(x_{i,u,v}\in[x_{i,1,1},x_{i,1,2}]_1
        \times[x_{i,2,1},x_{i,2,2}]_2\right).
    \end{align*}
    Indeed, suppose that $\sigma \models \psi_{\oplus}(a_{i,u,v})$. Let $\theta_1$ and $\theta_2$ denote the sub-permutations defined as in the statement of the Lemma. The first line of $\psi_\oplus$ implies that every point $p \in \sigma$ belongs to at least one of $\theta_1$ or $\theta_2$, the second line implies that we have $\sigma = \theta_1\oplus \theta_2$, and the third line implies that we have $\theta_1,\theta_2 \in \widetilde{\mathcal{O}}_k^{[d-1]}$. Finally, the last line ensures that each of these rectangles is non-empty, and its specified endpoints are indeed its extremal points. Proving the converse follows similarly, we omit the details. 
\end{proof}
We are now ready to prove the final lemma preceding the proof of Theorem~\ref{thm:Main O_K}. This lemma provides a first-order formula that enables us to handle all monotone increasing inflations in a uniform manner.
\begin{lem}\label{lemma: sum decomposition lemma}
    There exists a first-order formula $\psi_{\mathrm{inc}}(k_{1},\ell_1,k_2,\ell_2) \in \mathrm{FO}(\Sigma_{d}^*)$ so that $w \models \psi_{\mathrm{inc}}(a_{1},b_1,a_2,b_2)$ if and only if the following conditions hold:
    \begin{enumerate}
        \item $w \models \Gamma_d$, $\rho_2(w_1) = \mathrm{inc}$, and there exists $j$ so that $u_j$ and $u_{j+1}$ coincide with the sub-words $w_{a_1}\dots w_{b_1}$ and $w_{a_2} \dots w_{b_2}$ respectively; and 
        \item if $f_{d-1}(\rho_1(u_j)) \oplus f_{d-1}(\rho_1(u_{j+1})) = \theta_1 \oplus \theta_2$ for such choice of $j$ as above, and for some (possibly empty) permutations $\theta_1, \theta_2 \in \widetilde{\mathcal{O}}_k^{[d-1]}$, we have $|\theta_1| \leq |f_{d-1}(\rho_1(u_j))|$. 
    \end{enumerate}
    Similarly, there exists a first-order formula $\psi_{\mathrm{dec}}(k_1,\ell_1,k_2,\ell_2)$ which satisfies analogous conditions with $\oplus$ replaced by $\ominus$. 
\end{lem}
\begin{proof}
    Let $\chi_{\square}(\nu,\nu_{u,v}) \in \mathrm{FO}(\Sigma_{d}^*)$  with  $u,v\in \{1,2\}$denote the formula obtained by applying Theorem~\ref{thm: FO correspondence} to the formula $x \in [x_{1,1},x_{1,2}]_1 \times [x_{2,1},x_{2,2}]_2$.
    Write $\chi_{\oplus}(\nu_{i,u,v})$ for the first-order formula which corresponds to the formula $\psi_{\oplus}(x_{i,u,v})$ constructed in Lemma~\ref{lemmma: oplus}. 
    In addition, recall that $\upsilon_{d-1,k}(x_1,x_2,y_1,y_2)$ is the first-order formula constructed in Proposition~\ref{prop: small classes FOS classifiable}, and let $\chi_v$ denote the corresponding formula on words provided by Theorem~\ref{thm: FO correspondence}. 

    Set $\psi_{\mathrm{inc}}(k_1,\ell_1,k_2,\ell_2)$ to be the first-order formula
    \begin{align*}
        &\exists_{i,u,v\in \{1,2\}} \nu_{i,u,v}: \forall_{i\in \{1,2\}} \nu_i \text{ }\forall_{i,u,v\in \{1,2\}} \nu'_{i,u,v}\text{ } \forall t:\Bigl[ \\
        & \Gamma_d \land \rho_2(w_1) = \mathrm{inc} \land k_1\leq\ell_1<k_2\leq\ell_2\land(t\leq\ell_1\lor k_2\leq t) \\
        &\land\bigwedge_{i=1}^{2}\left((k_i\leq t\land t\leq\ell_i)\implies
        \left((L\in\rho_4(w_t)\iff t=k_i)\land(R\in\rho_4(w_t)\iff t=\ell_i)\right)\right)\\
        &\land\bigwedge_{i=1}^{2}\left(\left(\bigwedge_{u,v=1}^{2}k_i\leq\nu_{i,u,v}\leq\ell_i\right)
        \land\left((k_i\leq\nu_i\land\nu_i\leq\ell_i)\iff\chi_\square(\nu_i,\nu_{i,u,v})\right)\right)\\
        & \land \left(\neg \chi_{\upsilon}(\nu_{1,1,1},\nu_{1,2,1},\nu_{2,1,2},\nu_{2,2,2})\right)\\
        &\land\left((\chi_\oplus(\nu'_{i,u,v}))_{k_1,\ell_2}
        \implies\xi_{1,d}(\nu'_{1,1,2},\nu_{1,1,2})\right)\Bigr].
    \end{align*}
    Here $(\chi_\oplus(\nu'_{i,u,v}))_{k_1,\ell_2}$  is the first-order formula on words obtained by applying Theorem~\ref{thm: FO correspondence} to the formula $\chi_{\oplus}(\nu_{i,u,v})$, and including the restriction that for every bound and free variable $z$ must satisfy $k_1 \leq z \leq \ell_2$. Since $L \in \rho_4(w_{k_1})$ and $R \in \rho_4(w_{\ell_2})$, the points corresponding to the letters of the sub-word $w_{k_1}\dots w_{\ell_2}$ form a union of blocks in the decomposition. In particular, the condition $k_1 \leq z \leq \ell_2$ simply implies that $\alpha_w(z)$ belongs to the union of these blocks. 
    
    In order to check that $\psi_{\mathrm{inc}}$ satisfies the required properties, suppose that we have $w\models \psi_{\mathrm{inc}}(a_1,b_1,a_2,b_2)$. The second line implies that the profile of $w$ is $\mathrm{inc}$, and $w$ decomposes as $w = u_1\dots u_\ell$ with each $u_i$ satisfying the conditions described by~\eqref{eq: consistent conditions for I,D - 1} and ~\eqref{eq: consistent conditions for I,D - 2}. The third line implies that both sub-words $w_{a_i} \dots w_{b_i}$ must coincide with some $u_j$ for both $i \in \{1,2\}$, and the condition $b_1 < a_2 \land (\forall k: k \leq b_1 \lor a_2 \leq k)$ implies that there exists $j$ so that $w_{a_1}\dots w_{b_1} = u_j$ and $w_{a_2} \dots w_{b_2} = u_{j+1}$ - that is, these two subwords in fact correspond to consecutive subwords $u_j$ and $u_{j+1}$. Thus the first condition of Lemma is satisfied.  

    Fix values of $\nu_{i,u,v}$ which realize the fact that $w \models \psi_{\mathrm{inc}}(a_1,b_1,a_2,b_2)$. As usual, write $\sigma = f_d(w)$ and $\sigma_i = f_{d-1}(\rho_1(u_i))$. In addition, write $\tau$ for the monotone increasing permutation of length $\ell$. 

    Theorem~\ref{thm: FO correspondence} implies that the condition on the fourth line is equivalent to the following condition: the letter $w_{\nu_i}$ is contained in the sub-word $u_s$ for $s \in \{j,j+1\}$ if and only if the corresponding point $\alpha_w(\nu_i)$ is contained in the set $[\alpha_w(\nu_{s',1,1}),\alpha_w(\nu_{s',1,2})]_1 \times [\alpha_w(\nu_{s',2,1}), \alpha_w(\nu_{s',2,2})]_2$, where $s' = s - j  +1$. Equivalently, when viewing $\sigma$ as an inflation of the form $\sigma = \tau[\sigma_1,\dots,\sigma_\ell]$, the points in $\sigma_s$ are those in  $[\alpha_w(\nu_{s',1,1}),\alpha_w(\nu_{s',1,2})]_1 \times [\alpha_w(\nu_{s',2,1}), \alpha_w(\nu_{s',2,2})]_2$ for both $s \in \{j,j+1\}$, where $s' = s - j +1$. 

    Suppose that $\theta_1, \theta_2 \in \widetilde{\mathcal{O}}_k^{[d-1]}$ are chosen so that $\sigma_{j} \oplus \sigma_{j+1} = \theta_1 \oplus \theta_2$. The fifth line is equivalent to the condition that $\sigma_j \oplus \sigma_{j+1} \not \in \widetilde{\mathcal{O}}_k^{[d-1]}$, thus implying that both $\theta_1$ and $\theta_2$ must be non-empty. Note that the condition $\chi_{\oplus}(\nu'_{i,u,v})_{k_1,\ell_2}$ is satisfied exactly when the values of the variables $\nu'_{i,u,v}$'s are chosen so that $p \in \theta_i$ if and only if $p \in [\alpha_w(\nu'_{i,1,1}),\alpha_w(\nu'_{i,1,2})]_1 \times [\alpha_w(\nu'_{i,2,1}),\alpha_w(\nu'_{i,2,2})]_2$. This follows from the additional restriction imposed on the variable occurring in $\chi_{\oplus}(\nu'_{i,u,v})_{k_1,\ell_2}$, which ensures that the formula $\chi_\oplus$ is applied to the sub-permutation $\sigma_j \oplus \sigma_{j+1}$ (which consists of the points $\alpha_w(s)$ for $s \in [a_1,b_2]$). 
    
    For any such choice of variables the condition $\xi_{1,d}(\nu'_{1,1,2},\nu_{1,1,2})$ must be satisfied. By Theorem~\ref{thm: FO correspondence}, this is equivalent to having $\alpha_w(\nu'_{1,1,2})\leq_1\alpha_w(\nu_{1,1,2})$, which is equivalent to $|\theta_1|\leq|\sigma_j|$, as these points are the largest points in their respective blocks with respect to $\leq_1$. The proof of the converse direction and the construction for profile $\mathrm{dec}$ follow analogously. 
\end{proof}

\begin{proof}[Proof of Theorem~\ref{thm:Main O_K}]
    We construct the sentences $\Delta_d$ by induction. As the base
    case, note that $\Delta_1=\Gamma_1$ satisfies the required property. Suppose that $d\geq2$ and that $\Delta_{d-1}$ satisfies the conditions of the Theorem. For each $\tau\in\mathcal{Q}_k$, let $\Delta_{\tau,d}$ denote the first-order sentence obtained by applying Lemma~\ref{lemma: handling Q_ks for bijection} with $\hat\Delta:=\Delta_{d-1}$, and attaching $\exists$-quantifier with every free variable. 

    To handle with words whose profile is $\mathrm{inc}$ or $\mathrm{dec}$, we let $\delta(x,y)$ be the formula
    \begin{align*}
        x\leq y\ \land\ \forall c:\Bigl((x\leq c\land c\leq y)
        \implies\bigl((L\in\rho_4(w_c)\iff c=x)
        \land(R\in\rho_4(w_c)\iff c=y)\bigr)\Bigr).
    \end{align*}
    Note that $\delta(x,y)$ corresponds to detecting whether the sub-word $w_x\dots w_y$ corresponds to one of the words $u_i$ in the proper decomposition of $w$. We also write $\delta_2(a_1,b_1,a_2,b_2)$ for
    \[
        \delta(a_1,b_1)\land\delta(a_2,b_2)\land b_1<a_2
        \land\neg\exists c:\ b_1<c<a_2,
    \]
    which similarly corresponds to identifying that $w_{a_1}\dots w_{b_1}$ and $w_{a_2}\dots w_{b_2}$ correspond to two consecutive such sub-words $u_i$ and $u_{i+1}$. 
    
    Recall that we write $(\Delta_{d-1})_{a,b}$ for $\Delta_{d-1}$ applied to the projected sub-word $\rho_1(w_a\dots w_b)$. For $s\in\{\mathrm{inc},\mathrm{dec}\}$ define
    \begin{align*}
        \Delta_{s,d}:={}&
        \forall a\forall b:\bigl(\delta(a,b)
            \implies(\Delta_{d-1})_{a,b}\bigr)\\
        &\land\ \forall a_1\forall b_1\forall a_2\forall b_2:
        \bigl(\delta_2(a_1,b_1,a_2,b_2)
            \implies\psi_s(a_1,b_1,a_2,b_2)\bigr).
    \end{align*}
     Finally, we define
    \begin{align*}
        \Delta_d:={}&\Gamma_d
        \land\bigl(\rho_2(w_1)=\mathrm{inc}\implies\Delta_{\mathrm{inc},d}\bigr)\\
        &\land\bigl(\rho_2(w_1)=\mathrm{dec}\implies\Delta_{\mathrm{dec},d}\bigr)\\
        &\land\bigwedge_{\tau\in\mathcal{Q}_k}
            \bigl(\rho_2(w_1)=\tau\implies\Delta_{\tau,d}\bigr).
    \end{align*}

    Given $w$ with $w\models\Gamma_d$, suppose that $w$ has profile $\tau$ and properly decomposes as  $w=u_1\dots u_\ell$. Write $\sigma_i=f_{d-1}(\rho_1(u_i))$, and we choose $\tau'$ so that $\tau' = \tau$ when $\tau\in\mathcal{Q}_k$, and otherwise $\tau'$ is the monotone permutations of length $\ell$ whose direction is given by $\tau$. 
    \begin{claim}\label{claim: central for bijection}
        We have $w\models\Delta_d$ if and only if 
        \begin{enumerate}
            \item for every $i$ we have $\rho_1(u_i)\models\Delta_{d-1}$,
            \item the expression $\tau'[\sigma_1,\dots,\sigma_\ell]$ is the left-greedy decomposition of the respective permutation, and
        \item $\tau\neq \mathrm{dec}$ when $\ell=1$.
        \end{enumerate} 
    \end{claim}
    \begin{proof}
        For $\tau\in\mathcal{Q}_k$ the statement follows immediately from Lemma~\ref{lemma: handling Q_ks for bijection}. For $\tau=\mathrm{inc}$, the second line of $\Delta_{\mathrm{inc},d}$ ensures that for every $i$, we cannot write $\sigma_i \oplus \sigma_{i+1} = \theta_1 \oplus \theta_2$ with $|\theta_1| > |\sigma_i|$ (with all these permutations contained in $\widetilde{\mathcal{O}}_{k}^{[d-1]}$), and thus Lemma~\ref{lemma: albert left greedy} ensures that the expression given in the second condition indeed corresponds to the left-greedy decomposition. Finally, the first line of $\Delta_{\mathrm{inc},d}$ ensures that for each $i$, the sub-word $u_i$ is uniquely determined for the given unique $\sigma_i$. An analogous argument applies when $\tau= \mathrm{dec}$; in this case the third condition follows from the last line of $\Gamma_d$.
    \end{proof}
    \noindent Claim~\ref{claim: central for bijection} implies the uniqueness and existence of $w \in \Sigma_d^*$ with $w \models \Delta_d$ and $f_d(w) = \sigma$ for any $\sigma \in \widetilde{\mathcal{O}}_k^{[d]}$. Set $L_{d} = \{w \in \Sigma_{d}^*: w \models \Delta_{d}\}$. This language is regular by Theorem~\ref{thm: closure for regular languages}, and Theorem~\ref{thm: FO correspondence} gives the required logical correspondences for the restriction of $f_{d}$ to $L_{d}$. 
\end{proof}
\subsection{Constructing a finite deterministic automaton capturing non-$\star$ values}
For each $d \geq 1$, we construct an automaton $\mathcal{A}_d=(Q_d,\Sigma_d,\delta_d,e_d,F_d)$ which recognizes precisely those words $w$ with $g_d(w) \neq \star$. The construction is by induction on $d$. Before we start with the base case, we make the following observation which turns out to be handy during the inductive step. 
 \begin{lem} \label{lemma: unique single letter nonstar}
     For each $d \geq 1$ there exists a unique letter $v \in \Sigma_d$ for which $g_d(v) \neq \star$. 
 \end{lem}
 \begin{proof}
     For the base case $d = 1$, we note that $v = a_{\mathrm{inc}}$ is the unique such letter. For the inductive step, suppose that $g_d(v) \neq \star$. Note that $\tau:= \rho_2(v) \in \mathcal{Q}_k$ is impossible, as then $v$ (viewed as a word) must contain a letter with $\rho_3$-value $i$ for each $1 \leq i \leq |\tau|$. This is clearly impossible as $|\tau| \geq 2$ as $v$ consists of just a single letter. If $\rho_2(v) = \mathrm{dec}$, then $v$ viewed as a word must contain at least $2$ distinct positions whose $\rho_4$-image contains $L$, which is again impossible.
     
     Finally, if $\rho_2(v) = \mathrm{inc}$, we must have $\rho_4(v) =\{L,R\}$. Thus $g_d(v) \neq \star$ if and only if $g_{d-1}(\rho_1(v)) \neq \star $ by the construction of $g_d$. Thus $\rho_1(v)$ is uniquely determined, and given that $\rho_2(v) = \mathrm{inc}$ and $\rho_4(v) = \{L,R\}$, the letter $v$ itself is uniquely determined.
 \end{proof}

 We set $Q_1 = \{e_1,\star,q_{\mathrm{inc}},q_{\mathrm{dec}}^*,q_{\mathrm{dec}}\} \cup \bigcup_{\tau \in \mathcal{Q}_k} \{q_{\tau,i}: 1\leq i \leq |\tau|\}$ to be the set of states with $e_1$ to be the designated starting state, $\star$ corresponding to a rejection state (to which we transition if some prefix of $w$ already implies that $g_d(w) = \star$), and the set of accepting states is given by $F_1 = \{q_{\mathrm{inc}},q_{\mathrm{dec}}\} \cup \{q_{\tau,|\tau|}: \tau \in \mathcal{Q}_k\}$. We define the transition function $\delta_1$ by setting 
 \[\begin{aligned}
 \delta_1(q_{\tau,i-1},a_{\tau,i})&=q_{\tau,i}\\
 \delta_1(e_1,a_{\mathrm{inc}})=\delta_1(q_{\mathrm{inc}},a_{\mathrm{inc}})&=q_{\mathrm{inc}},\\
 \delta_1(e_1,a_{\mathrm{dec}})&=q_{\mathrm{dec}}^*,&
 \delta_1(q_{\mathrm{dec}}^*,a_{\mathrm{dec}})=\delta_1(q_{\mathrm{dec}},a_{\mathrm{dec}})&=q_{\mathrm{dec}},
\end{aligned}\]
where we use the convention that $q_{\tau,0}=e_1$. Finally, we define all other transition relations $\delta_1(q,v)$ not specified above to map to $\star$. See Figure~\ref{fig: A1 figure} for an illustration of $\mathcal{A}_1$. 
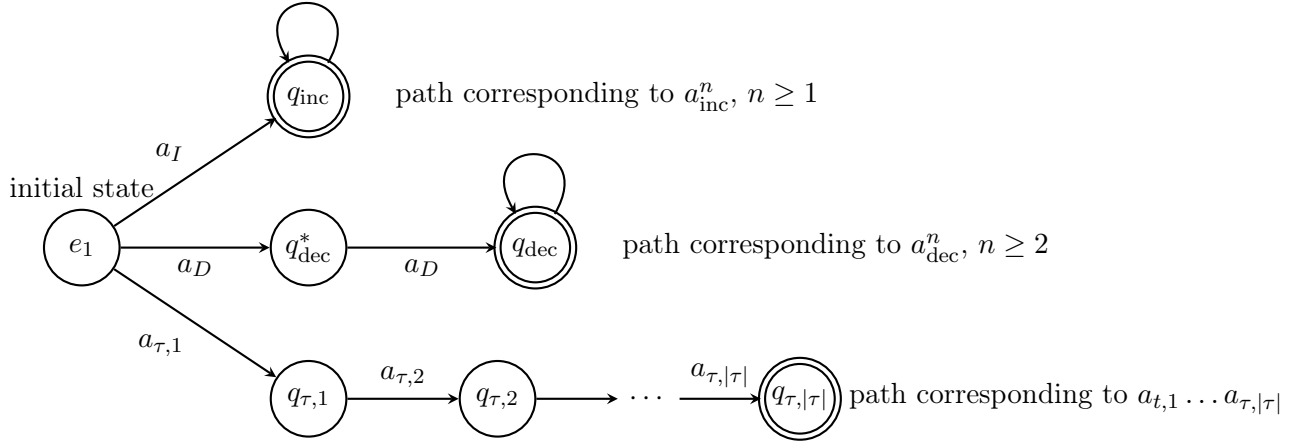
\begin{figure} 
    \centering
    \begin{tikzpicture}[>=stealth,thick,font=\small,
  state/.style={circle,draw,minimum size=10mm,inner sep=2pt},
  final/.style={state,double,double distance=1.5pt}]
  \node[state] (e) at (0,0) {$e_1$};
  \node[above] at (e.north) {initial state};
  \node[final] (inc) at (3,2) {$q_{\mathrm{inc}}$};
  \node[state] (dec0) at (3,0) {$q_{\mathrm{dec}}^*$};
  \node[final] (dec) at (6,0) {$q_{\mathrm{dec}}$};
  \node[state] (t1) at (3,-2) {$q_{\tau,1}$};
  \node[state] (t2) at (5.5,-2) {$q_{\tau,2}$};
  \node (dots) at (7.5,-2) {$\cdots$};
  \node[final] (tm) at (9.5,-2) {$q_{\tau,|\tau|}$};
  \node[right] at (10,-2)
    {path corresponding to $a_{t,1}\dots a_{\tau,|\tau|}$};
  \draw[->] (e) -- node[above left]{$a_I$} (inc);
  \draw[->] (inc) to[out=60,in=120,looseness=6] (inc);
  \node[right] at (4,2)
    {path corresponding to $a_{\mathrm{inc}}^n$, $n\geq1$};

  \draw[->] (e) -- node[below]{$a_D$} (dec0);
  \draw[->] (dec0) -- node[below]{$a_D$} (dec);
  \draw[->] (dec) to[out=60,in=120,looseness=6] (dec);
  \node[right] at (7,0) {path corresponding to $a_{\mathrm{dec}}^n$, $n\geq2$};

  \draw[->] (e) -- node[below left] {$a_{\tau,1}$} (t1);
  \draw[->] (t1) -- node[above] {$a_{\tau,2}$} (t2);
  \draw[->] (t2) -- (dots);
  \draw[->] (dots) -- node[above] {$a_{\tau,|\tau|}$} (tm);

\end{tikzpicture}
    \caption{Illustration of $\mathcal{A}_1$. Double circles are accepting states and $\star$ is omitted.}
    \label{fig: A1 figure}
\end{figure}

We now move on to the inductive step of the construction. Suppose that $d \geq 2$ and that we have constructed the automaton $\mathcal{A}_{d-1}$. Write $Q_{d-1}' = Q_{d-1}\setminus \{e_{d-1},\star\}$. Define the sets of states 
\[
\begin{aligned}
 B_{\tau,i}&=\{q_{r,\tau,i}:r\in Q_{d-1}'\}
 &&(\tau\in\mathcal Q_k,\ 1\le i\le|\tau|),\\
 B_{\alpha,\beta}&=\{q_{r,\alpha,\beta}:r\in Q_{d-1}'\}
 &&(\alpha\in\{\mathrm{inc},\mathrm{dec}\},\ \beta\subseteq \{L,R\}, \beta \neq \{L,R\}),\\
 B_{\alpha,\{L,R\}}&=\{q_{\alpha,LR}\}
 &&(\alpha\in\{\mathrm{inc},\mathrm{dec}\}).
\end{aligned}
\]
Note that $B_{\alpha,\{L,R\}}$ in particular consists of a single state. In addition, for every $\beta\subseteq\{L,R\}$, let $B_{\mathrm{dec},\beta}^*$ be a disjoint copy of $B_{\mathrm{dec},\beta}$; in particular, $B_{\mathrm{dec},\{L,R\}}^*=\{q_{\mathrm{dec},LR}^*\}$.

Write $H_{\tau} = \bigcup_{i=1}^{|\tau|} B_{\tau,i}$ and $H_{\alpha} = \bigcup_{\beta \subseteq \{L,R\}} B_{\alpha,\beta}$. In addition, put $H_{\mathrm{dec}}^*=\bigcup_{\beta\subseteq\{L,R\}}B_{\mathrm{dec},\beta}^*$ and finally we define
\[
Q_d=\{e_d,\star\}\cup\bigcup_{\tau\in\mathcal Q_k}H_\tau
       \cup H_{\mathrm{inc}}\cup H_{\mathrm{dec}}\cup H_{\mathrm{dec}}^*.
\]

The structure of the automaton splits naturally into parts based on the profile $\rho_2(w) \in \mathcal{Q}_k \cup \{\mathrm{inc},\mathrm{dec}\}$, and thus after the starting state, all the transitions occur inside the blocks $H_\alpha$ with $\alpha \in \mathcal{Q}_k \cup \{\mathrm{inc},\mathrm{dec}\}$, inside $H_{\mathrm{dec}}^*$ or from $H_{\mathrm{dec}}^*$ to $H_{\mathrm{dec}}$. We now describe the transition inside the blocks $H_{\alpha}$, and describe the accepting states $F_d$. 

We start by describing the transition in the case when $\alpha \in \mathcal{Q}_k$. First, if $v = a_{\tau,i,j}$, we set $\delta_d(e_d,v) = q_{u,\tau,1}$ with $u = \delta_{d-1}(e_{d-1},\rho_1(v))$ whenever $i = 1$ and $u \neq \star$, and $\delta_d(e_d,v) = \star$ whenever $i \geq 2$ or $u = \star$. 

Now consider a state $q:=q_{r,\tau,i} \in B_{\tau,i}$ with $1 \leq i \leq |\tau|$. Given a letter $v$, we set $\delta_d(q,v) = \star$ whenever $\rho_2(v) \neq \tau$ or $\rho_3(v) \not \in \{i,i+1\}$. If $\rho_2(v) = \tau$ and $\rho_3(v) = i$, we write $u =\delta_{d-1}(r,\rho_1(v))$ and set $\delta_d(q,v) = q_{u,\tau,i}$ whenever $u \neq \star$, and $\delta_d(q,v) = \star$ if $u = \star$. If $\rho_2(v) = \tau$ and $\rho_3(v) = i+1$, we write $u =\delta_{d-1}(e_{d-1},\rho_1(v))$. If $r \in F_{d-1}$ (i.e.\ $r$ is an accepting state in $\mathcal{A}_{d-1}$) and $u \neq \star$, we set $\delta_d(q,v) = q_{u,\tau,{i+1}}$. Finally, we set $F_{d} \cap H_{\tau}$ to consist of those states $q_{r,\tau,|\tau|}$ with $r \in F_{d-1}$. In particular, $F_{d} \cap H_{\tau} \subseteq B_{\tau,|\tau|}$. 

Now suppose that $\alpha = \mathrm{inc}$. We write $E_{\mathrm{inc}} = \{q_{\mathrm{inc},LR},e_d\} \cup B_{\mathrm{inc},\{R\}}$, and we call the set $E_{\mathrm{inc}}$ end states. We set $F_d \cap H_{\mathrm{inc}}$ to be the set consisting of $q_{\mathrm{inc},LR}$ together with the states that are of the form $q_{r,\mathrm{inc},\{R\}}$ for some $r \in F_{d-1}$. 

We now move on to define the transition rules. While the definition of the transition rules below is somewhat tedious, we remark that all the relations are inherited from depth $d-1$, and with the relations between the blocks corresponding to $\emptyset$, $L$, $R$ and $\{L,R\}$ being consistent so that each accepted $w$ decomposes as $w = u_1\dots u_s$ so that each $\rho_1(u_i)$ is accepted by $\mathcal{A}_{d-1}$, the left end-point of each $u_i$ is the unique letter receiving label $L$ and the right end-point is the unique letter receiving label $R$. This matches exactly with the definition of $g_d$. For illustration, see Figure~\ref{fig:Ad increasing block}. 

\begin{figure}
\centering
\begin{tikzpicture}[>=stealth,thick,font=\small,
  block/.style={draw,rounded corners,minimum width=25mm,
                minimum height=10mm},
  lab/.style={font=\scriptsize,fill=white,inner sep=1.5pt}]

  \node[draw,circle] (e) at (0,0) {$e_d$};
  \node[below] at (e.south) {initial state};
  \node[draw,circle] (LR) at (3,0) {$q_{\mathrm{inc},LR}$};

  \node[block] (L) at (7,2.3) {$B_{\mathrm{inc},\{L\}}$};
  \node[block] (O) at (7,0) {$B_{\mathrm{inc},\emptyset}$};
  \node[block] (R) at (7,-2.3) {$B_{\mathrm{inc},R}$};

  \draw[->] (e.north)
    .. controls (0,3.2) and (4,3.2) ..
    node[lab,above,pos=.55] {$L$} (L.west);
  \draw[->] (e) -- node[lab,above] {$LR$} (LR);

  \draw[->] (L) -- node[lab,right] {$\emptyset$} (O);
  \draw[->] (O) -- node[lab,right] {$R$} (R);

  \draw[->] (O.north east)
    .. controls +(0.45,0.3) and +(0.45,-0.3) ..
    node[lab,right,pos=.5] {$\emptyset$} (O.south east);

  \draw[->] (L.east)
    .. controls (10,2.3) and (10,-2.3) ..
    node[lab,right,pos=.5] {$R$} (R.east);

  \draw[->] (LR) -- node[lab,above left] {$L$} (L);
  \draw[->] (R.south east)
    .. controls (12,-3.3) and (12,3.3) ..
    node[lab,right,pos=.5] {$L$} (L.north east);
  \draw[->] (R) -- node[lab,below left] {$LR$} (LR);
  \draw[->] (LR) to[out=65,in=115,looseness=3]
    node[lab,above] {$LR$} (LR);
\end{tikzpicture}
    \caption{Illustration of the states of $H_\mathrm{inc}$ in $\mathcal{A}_d$. The labels on the directed edges between the blocks are chosen to represent the $\rho_4$-images of the labels along the directed edges between the blocks.}
    \label{fig:Ad increasing block}
\end{figure}
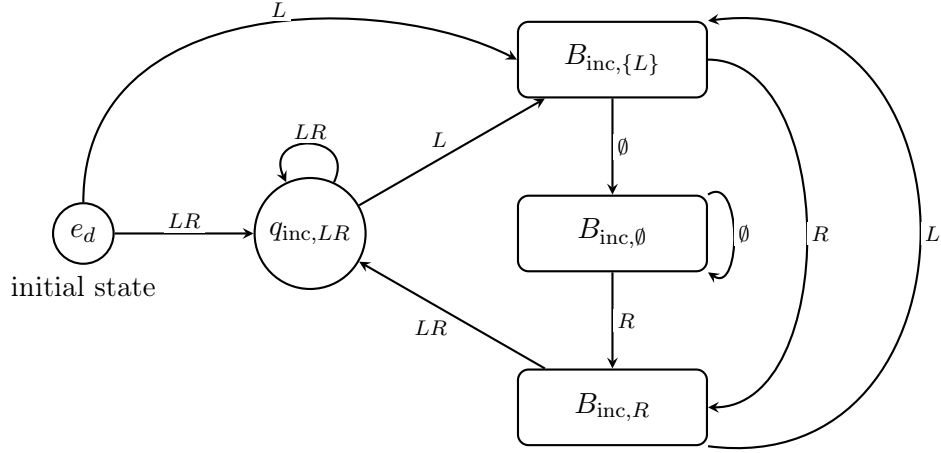

Throughout the definition, we use the convention that whenever $u = \star$, the state $q_{u,\alpha,\beta}$ is defined to be the rejecting state $\star$. If $q \in H_{\mathrm{inc}}$ and $v\in \Sigma_d$ with $\rho_2(v) \neq \mathrm{inc}$, we set $\delta_d(q,v) = \star$. 

First consider the case when $q\in E_{\mathrm{inc}}$. If $q\not \in (F_d\cap E_{\mathrm{inc}})\cup\{e_d\}$, we set $\delta_d(q,v) = \star$ for every $v$. Now suppose that $q \in (F_d \cap E_{\mathrm{inc}}) \cup \{e_d\}$. In this case, for letters $v$ with $\rho_2(v) = \mathrm{inc}$, we set $u = \delta_{d-1}(e_{d-1},\rho_1(v))$ and define 
\begin{align}\label{eq: monotone delta_d definition}
\delta_d(q,v)=
\begin{cases}
q_{u,\mathrm{inc},\{L\}},&\rho_4(v)=\{L\},\quad u\neq\star,\\
q_{\mathrm{inc},LR},&\rho_4(v)=\{L,R\},\quad u\in F_{d-1},\\
\star,&\text{otherwise}.
\end{cases}        
\end{align}

Now suppose that $q = q_{r,\mathrm{inc},\beta} \in B_{\mathrm{inc},\emptyset} \cup B_{\mathrm{inc},\{L\}}$. Given $v \in \Sigma_d$ with $\rho_2(v) = \mathrm{inc}$, if $L \not \in \rho_4(v)$ we set $u = \delta_{d-1}(r,\rho_1(v))$ and define $\delta_d(q,v) = q_{u,\mathrm{inc},\rho_4(v)}$. If $L\in \rho_4(v)$, we define $\delta_d(q,v) =\star$.

Finally, we consider $H_{\mathrm{dec}}^*$ and $H_{\mathrm{dec}}$. For an illustration, see Figure~\ref{fig: decreasing Ad}. The transition rules inside $H_{\mathrm{dec}}$ are defined identically as in $H_{\mathrm{inc}}$. The transition $\delta_d(e_d,v)$ for letters $v$ with $\rho_2(v) =\mathrm{dec}$ is defined as in~\eqref{eq: monotone delta_d definition}, so that the transition occurs to the appropriate copy of the state located in $H_{\mathrm{dec}}^*$. Inside $H_{\mathrm{dec}}^*$, we keep all the directed edges corresponding to those in $H_{\mathrm{inc}}$ whose starting point does not belong to the set $E_{\mathrm{dec}}^*:= \{q_{\mathrm{dec},LR}^*\}\cup B^*_{\mathrm{dec},\{R\}}$. In addition, we add a copy of every directed edge corresponding to one in $H_{\mathrm{inc}}$ whose corresponding starting state belongs to $E_{\mathrm{dec}}^*$, and with the end point chosen to be the corresponding state in $H_{\mathrm{dec}}$. Finally, we set $F_d \cap H_{\mathrm{dec}}$ to consist of the state $q_{\mathrm{dec},LR}$ together with all the states $q_{r,\mathrm{dec},\{R\}}$ with $r \in F_{d-1}$. We also set $F_d \cap H_{\mathrm{dec}}^* = \emptyset$. Finally, no state that we have not yet specified as an accepting state is an accepting state, and every transition not specified above leads to $\star$, including every transition from $\star$.
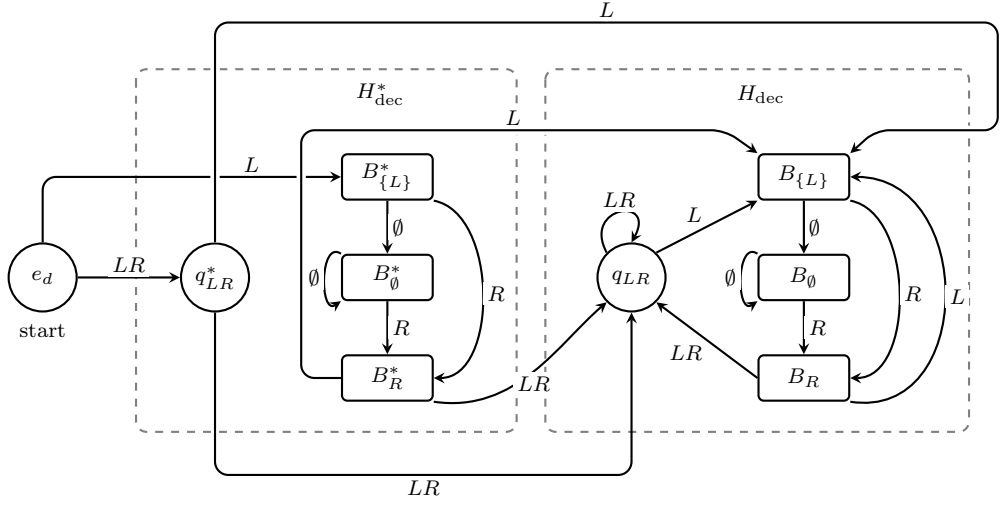
\begin{figure}
    \centering
    
\begin{tikzpicture}[>=stealth,thick,font=\scriptsize,
  x=.95cm,y=.95cm,
  block/.style={draw,rounded corners=2pt,
                minimum width=12mm,minimum height=6mm,
                inner sep=2pt},
  singleton/.style={draw,circle,minimum size=9mm,inner sep=2pt},
  lab/.style={fill=white,inner sep=1.5pt},
  clear/.style={preaction={draw=white,line width=3.5pt}}]

  \draw[dashed,rounded corners,gray]
    (.5,2.9) rectangle (5.8,-2.15);
  \draw[dashed,rounded corners,gray]
    (6.2,2.9) rectangle (12.1,-2.15);
  \node at (3.9,2.55) {$H_{\mathrm{dec}}^*$};
  \node at (9.2,2.55) {$H_{\mathrm{dec}}$};

  \node[singleton] (e) at (-.8,0) {$e_d$};
  \node[below] at (e.south) {start};

  \node[singleton] (LRs) at (1.6,0) {$q_{LR}^{*}$};
  \node[block] (Ls) at (4,1.4) {$B_{\{L\}}^*$};
  \node[block] (Os) at (4,0) {$B_{\emptyset}^*$};
  \node[block] (Rs) at (4,-1.4) {$B_R^*$};

  \node[singleton] (LRt) at (7.4,0) {$q_{LR}$};
  \node[block] (Lt) at (9.8,1.4) {$B_{\{L\}}$};
  \node[block] (Ot) at (9.8,0) {$B_{\emptyset}$};
  \node[block] (Rt) at (9.8,-1.4) {$B_R$};

  \foreach \s/\x in {s/4,t/9.8} {
    \draw[->] (L\s.south) --
      node[lab,right] {$\emptyset$} (O\s.north);
    \draw[->] (O\s.south) --
      node[lab,right] {$R$} (R\s.north);
    \draw[->] (O\s.north west)
      .. controls +(-.3,.2) and +(-.3,-.2) ..
      node[lab,left,pos=.5] {$\emptyset$}
      (O\s.south west);
    \draw[->] (L\s.south east)
      .. controls (\x+1.55,.9) and (\x+1.55,-1.4) ..
      node[lab,right,pos=.5] {$R$} (R\s.east);
  }

  \draw[->] (e.east) --
    node[lab,above] {$LR$} (LRs.west);
  \draw[->,rounded corners=5pt]
    (e.north) -- (-.8,1.4) --
    node[lab,above,pos=.7] {$L$} (Ls.west);

  \draw[->,rounded corners=5pt]
    (LRs.south) -- (1.6,-2.75) --
    node[lab,below] {$LR$}
    (7.4,-2.75) -- (LRt.south);

  \draw[->] (Rs.south east)
    .. controls (5.8,-1.9) and (6.4,-.9) ..
    node[lab,below,pos=.5] {$LR$} (LRt.south west);

  \draw[->,clear,rounded corners=5pt]
    (Rs.west) -- (2.8,-1.4) -- (2.8,2.05) --
    node[lab,above] {$L$}
    (8.7,2.05) -- (Lt.north west);

  \draw[->,clear,rounded corners=5pt]
    (LRs.north) -- (1.6,3.55) --
    node[lab,above] {$L$}
    (12.5,3.55) -- (12.5,2.05) --
    (10.8,2.05) -- (Lt.north east);

  \draw[->] (LRt.north east) --
    node[lab,above left] {$L$} (Lt.south west);

  \draw[->] (Rt.west) --
    node[lab,below left] {$LR$} (LRt.south east);

  \draw[->] (LRt.north west)
    .. controls +(-.35,.65) and +(.35,.65) ..
    node[lab,above,pos=.5] {$LR$} (LRt.north);

  \draw[->] (Rt.south east)
    .. controls (12.2,-2) and (12.2,1.4) ..
    node[lab,right,pos=.5] {$L$} (Lt.east);

\end{tikzpicture}
    \caption{Illustration of the states of $H_\mathrm{dec}^*$ and $H_\mathrm{dec}$ in $\mathcal{A}_d$. Again, the labels on the directed edges between the blocks are chosen to represent the $\rho_4$-images of the labels along the directed edges between the blocks.}
    \label{fig: decreasing Ad}
\end{figure}

Intuitively speaking, given a word $w$ whose letters have type $\mathrm{dec}$, we first enter $H_{\mathrm{dec}}^*$. Inside $H_{\mathrm{dec}}^*$, we cycle through once starting from a state labelled with $L$, and ending with state labelled with $R$ (which includes possibly visiting only the state $q_{\mathrm{dec},LR}^*$). If the state labelled with $R$ corresponds to an accepting state of $\mathcal{A}_{d-1}$ (or is the state $q_{\mathrm{dec},LR}^*$), and the next letter is labelled with $\{L\}$ (or $\{L,R\}$), we move to the appropriate state in $H_{\mathrm{dec}}$. This corresponds precisely to the automata accepting precisely those words $w$ all of whose letters have type $\mathrm{dec}$, which decompose properly as $w = u_1\dots u_s$, and passing through the part corresponding to $H_{\mathrm{dec}}^*$ ensures that we must have $s \geq 2$. 

We state that the automaton $\mathcal{A}_d$ accepts precisely those words $w \in \Sigma_d^*$ with $g_d(w) \neq \star$.  
\begin{lem}\label{lem:Ad-nonstar}
The automaton $\mathcal A_d$ accepts exactly those $w$ with $g_d(w) \neq \star$.
\end{lem}
\begin{proof}
It is easy to check that the statement is true for $d = 1$; in particular, note that the state $q_{\mathrm{dec}}^*$ ensures that $a_{\mathrm{dec}}^n$ is accepted only when $n \geq 2$. 

When $d \geq 2$, the inductive structure ensures that every $w$ accepted by the automaton must decompose properly. If the type of $w$ coincides with $\tau$, the statement follows easily: note that each sub-word $u_i$ must be non-empty as the states in $B_{\tau,i}$ have only directed edges to states in $B_{\tau,j}$ with $j \in \{i,i+1\}$, and in both of these cases these  edges are labelled with letters $v$ satisfying $\rho_3(v) = j$. If the type of $w$ equals $\mathrm{inc}$ or $\mathrm{dec}$, then $\mathcal{A}_d$ accepts $w$ if and only if $w$ decomposes as $w = u_1\dots u_s$ into non-empty sub-words $u_i$ which satisfy the following properties: 
\begin{enumerate}
    \item $\rho_1(u_t)$ is accepted by $\mathcal{A}_{d-1}$ for every $t$,
    \item the only letter in $u_i$ whose $\rho_4$-image contains $L$ is its left end-point, and the only letter in $u_i$ whose $\rho_4$-image contains $R$ is its right end-point. 
\end{enumerate}
In addition, the block $H_{\mathrm{dec}}^*$ guarantees that we must have $s \geq 2$ when the type of $w$ is $\mathrm{dec}$. This completes the proof. 
\end{proof}

Note that the automaton $\mathcal{A}_d$ splits into disjoint parts $H_\tau$ for $\tau \in \mathcal{Q}_k$, $H_{\mathrm{inc}}$, $H_{\mathrm{dec}}^*$ and $H_{\mathrm{dec}}$ with no directed cycles crossing between the parts. Thus every strongly connected component with at least two vertices is contained in one of these blocks. Since each such block can be further decomposed into disjoint parts that closely resemble $\mathcal{A}_{d-1}$ (in fact, with only the initial state and the rejection state removed), and these parts interact only in a very delicate way, we may backtrack through the inductive structure and find well-behaved copies of these blocks in which all the strongly connected components are contained. To make this process formal, we introduce the notion of iterative sequences. 
\begin{definition}
    An iterative sequence of depth $j \in \{1,\dots,d\}$ is a sequence $\underline{\tau} = (\tau_s,i_s)_{s = j}^d$ with $\tau_s \in \mathcal{Q}_k \cup \{\mathrm{inc},\mathrm{dec},\mathrm{dec}^*\}$ for all $j +1 \leq s \leq d$, $\tau_j \in \{\mathrm{inc},\mathrm{dec}\}$, and $i_s \in \{1,\dots,|\tau_s|\}$ for every $s$ with $\tau_s \in \mathcal{Q}_k$, and with $i_s$ otherwise undefined. 
\end{definition}
Given an iterative sequence $\underline{\tau}$ of depth $j$, we also identify a subpart of the automaton $\mathcal{A}_d$ on depth $j$ which corresponds to $\underline{\tau}$. 
\begin{definition}
    Let $\underline{\tau} = (\tau_s,i_s)_{s = j}^d$ denote an iterative sequence of depth $j$. We define the set $H_{\underline{\tau}}$ of states in $Q_d'$ by induction on $d$ as follows. If $d = 1$, $\underline{\tau}$ only consists of $\tau_1 \in \{\mathrm{inc},\mathrm{dec}\}$ as $i_1$ is undefined. In this case, we define $H_{\underline{\tau}}$ to be the set $\{q_{\tau_1}\}$ in $\mathcal{A}_1$. 

    If $d \geq 2$ and $j = d$, we define $H_{\underline{\tau}} = H_{\tau_d}$. Note that this is well-defined as $\tau_d \in \{\mathrm{inc},\mathrm{dec}\}$. If $j < d$, let $\underline{\tau}'=(\tau_s,i_s)_{s = j}^{d-1}$ denote the sequence obtained by deleting the last element of $\underline{\tau}$.

    Let $Q$ be the copy of $Q_{d-1}'$ given by $B_{\tau_d,i_d}$ if $\tau_d \in \mathcal{Q}_k$, by $B_{\tau_d,\emptyset}$ if $\tau_d \in \{\mathrm{inc},\mathrm{dec}\}$, and by $B^*_{\mathrm{dec},\emptyset}$ if $\tau_d = \mathrm{dec}^*$. We then define $H_{\underline{\tau}}$ to be the set of states $H_{\underline{\tau'}}$ obtained when we apply the definition to the sequence $\underline{\tau}'$ in $Q$. 
\end{definition}
Note that since $Q$ corresponds to a copy of $Q_{d-1}'$, the inductive hypothesis can be applied to $Q$. It is easy to verify by induction that if $\underline{\tau}$ has depth $j$, then $H_{\underline{\tau}}$ coincides with $H^{j}_{\alpha}$ for $\alpha = \tau_j\in \{\mathrm{inc},\mathrm{dec}\}$, where $H^{j}_\alpha$ denotes the construction of $H_\alpha$ inside $\mathcal{A}_j$. Here, we slightly abuse the notation and use the convention that $H^{1}_\alpha = \{q_\alpha\}$.

For $j \geq 2$, decompose $H^{j}_\alpha$ as $\bigcup_{\beta \subseteq \{L,R\}} B^{j-1}_{\alpha,\beta}$, where for $\beta \neq \{L,R\}$ these sets correspond to their natural counterparts in the decomposition of this particular $H^{j}_\alpha$ and are thus copies of $Q_{j-1}'$, and $B^{j-1}_{\alpha,\{L,R\}}$ consists of the single state labelled with $LR$ and direction $\alpha$. When $j = 1$, we use the natural convention that $H_{\underline{\tau},\beta} = \emptyset$ for every $\beta \subseteq \{L,R\}$. For $j \geq 2$, we set $H_{\underline{\tau},\beta} = B^{j-1}_{\alpha,\beta}$ for every $\beta \subseteq \{L,R\}$.

We now prove that the strongly connected component $C$ in $\mathcal{A}_d$ are naturally related to the iterative sequences. By strongly connected sets of states we refer to set $H$ which is strongly connected in the transition digraph, but not necessarily maximal with respect to such property.

\begin{lem}\label{lemma: iterative sequence lemma}
    Let $C$ be a strongly connected set of states in $Q'_d$ which is located along some directed walk from the initial state to an accepting state, and which   either contain at least $2$ states or a self-loop. Then there exists an iterative sequence $\underline{\tau}$ of depth $j$ so that $C \subseteq H_{\underline{\tau}}$ and $C \not \subseteq H_{\underline{\tau},\emptyset}$.
\end{lem}
\begin{proof}
    We prove the statement by induction on $d$. When $d = 1$, the only strongly connected sets $C$ satisfying the conditions are $\{q_{\mathrm{inc}}\}$ and $\{q_{\mathrm{dec}}\}$, and both clearly satisfy the statement of the Lemma.
    
    Since $C$ either contains a self-loop or at least $2$ states, it must  contain a directed cycle. Note that a simple inductive argument proves that any directed edge in $\mathcal{A}_d$ which does not end at $\star$ is labelled with a unique letter. 
    
    Let $v$ be letter attached as a label of an edge in $C$, and let $\tau = \rho_2(v)$. First consider the case when $\tau \in \mathcal{Q}_k$. Then $C$ must be clearly contained in $H_{\tau}$. Furthermore, since all the directed edges whose starting point is in $B_{\tau,i}$ have their end-point  in $B_{\tau,i} \cup B_{\tau,i+1}$, we conclude that all the states in $C$ must belong to a single $B_{\tau,i}$. Since $B_{\tau,i}$  corresponds to a copy of $Q'_{d-1}$, the inductive hypothesis implies that there exists an $\underline{\tau}' =(\tau_s,i_s)_{s = j}^{d-1}$ satisfying the conclusion of the Lemma. Thus we may choose $\underline{\tau}$ to be the sequence obtained by appending the sequence $\underline{\tau}'$ by $(\tau,i)$. 

    Now suppose that $\tau = \mathrm{inc}$. Since all the directed edges whose starting point is in $H_{\mathrm{inc}}$ have their end-point in $H_{\mathrm{inc}}\cup \{\star\}$, we conclude that $C \subseteq H_{\mathrm{inc}}$. If $C \subseteq B_{\mathrm{inc},\emptyset}$, we may apply the inductive hypothesis as above to complete the proof (and by setting $\tau_d = \mathrm{inc}$). If $C \not \subseteq B_{\mathrm{inc},\emptyset}$, we may simply choose $j = d$ and $\tau_d = \mathrm{inc}$. An analogous argument applies when $\tau = \mathrm{dec}$ and $C \subseteq H_{\mathrm{dec}}$. 
    
    Finally, suppose that $C \subseteq H_{\mathrm{dec}}^*$. Recall that all the directed edges whose starting point is in $B_{\mathrm{dec},\{R\}}^*$ or $B_{\mathrm{dec},\{L,R\}}^*$ have their end-point in $H_{\mathrm{dec}}$. Similarly, all the directed edges whose end-point is in $B_{\mathrm{dec},\{L\}}^*$ must have $e_d$ as their starting point - and note that $e_d$ is not contained in any directed cycle. Thus $C$ cannot contain vertices from any of these sets, and thus we must have $C \subseteq B_{\mathrm{dec},\emptyset}^*$. Thus we may set $\tau_d = \mathrm{dec}^*$, and again an analogous argument completes the proof. 
\end{proof}
\section{Small permutation classes}\label{section: small classes}

Having constructed the necessary tools to deal with $\star$-entries in the previous section, we move on to generalize the construction of Albert et.\ al.\ from \cite{albert2013geometric} presented in section~\ref{subsetion: albert construction}. Throughout this section, we assume that $M$ is a $[\ell]\times [m]$ $\star$-isolated forest matrix, and we also recall that $M$ is hence multiplicative. For convenience, we write $\mathcal{S}_+ = \{(i,j): M_{i,j} = +1\}$, $\mathcal{S}_- = \{(i,j): M_{i,j} = -1\}$,  $\mathcal{S}_\star = \{(i,j): M_{i,j} = \star \}$,  $\mathcal{S}_{\pm} = \mathcal{S}_+ \sqcup \mathcal{S}_-$ and finally $\mathcal{S} = \mathcal{S}_\pm \sqcup \mathcal{S}_\star$. Given positive integers $d$ and $k$, enumerate the elements of $\Sigma_{d,k}$ as $\{b_1,\dots,b_N\}$ for some $N$.  Let $\Sigma_{M,\pm}$ denote an alphabet containing a letter $a_{i,j}$ for every $(i,j) \in \mathcal{S}_\pm$. In addition, for each $(i,j) \in \mathcal{S}_\star$ write $\Sigma_{\star,(i,j)}$ for the alphabet $\left\{b_{(i,j),1},\dots,b_{(i,j),N}\right\}$, where $b_{(i,j),u}$ denotes a copy of letter $b_u$ attached to the $(i,j)$-th cell. Finally, we set $\Sigma_M = \Sigma_{M,\pm} \cup \bigcup_{(i,j): M_{i,j} = \star} \Sigma_{\star, (i,j)}$. 

Having constructed the alphabet $\Sigma_M$, we define the maps $\rho_5: \bigcup_{(i,j) \in \mathcal{S}_\star} \Sigma_{\star,(i,j)} \rightarrow \Sigma_{d,k}$ and $\rho_6: \Sigma_M \rightarrow [\ell] \times [m]$ as follows. For each $b_{(i,j),u} \in \Sigma_{\star, (i,j)}$, define $\rho_5(b_{(i,j),u}) = b_u$, and extend the map $\rho_5$ to words by setting $\rho_5(w_1\dots w_s) = \rho_5(w_1) \dots \rho_5(w_s)$. Finally, we set $\rho_6(a_{i,j}) = (i,j)$ for every $a_{i,j} \in \Sigma_{M,\pm}$ and $\rho_6(b_{(i,j),u}) = (i,j)$ for every $b_{(i,j),u} \in \Sigma_{\star,(i,j)}$. 

Enumerate the elements of $\mathcal{S}_\star$ as $(i_1,j_1),\dots,(i_r,j_r)$ with $i_1 < \dots <i_r$, and note that such a choice is possible as $M$ is $\star$-isolated. For $w \in \Sigma_M^*$, define the sub-words $u_0$ and $u_{(i_1,j_1)},\dots,u_{(i_r,j_r)}$ as follows. We write $u_0$ for the sub-word which consists of every letter $w_s$ with $w_s \in \Sigma_{M,\pm}$. For each $(i,j)\in \mathcal{S}_\star$, define $u_{(i,j)}$ to be the sub-word consisting of every letter $w_s$ with $w_s \in \Sigma_{\star,(i,j)}$. We say that $w$ \emph{decomposes properly} if $w$ coincides with the concatenation $u_0 u_{(i_1,j_1)} \dots u_{(i_r,j_r)}$, where we allow the possibility that some of these sub-words are empty. 
 
Having defined the necessary notions on the alphabets $\Sigma_M$, we move on to construct the functions $g_M: \Sigma_M^* \rightarrow \mathrm{Geom}^\#(M)_{d,k} \cup \{\star\}$ and $f_M := \pi_M^\# \circ g_M$. Given $w \in \Sigma_M^*$, if $w$  does not decompose properly, we set $g_M(w) = \star$. If $w$ decomposes properly as $w = u_0 u_{(i_1,j_1)}\dots u_{(i_r,j_r)}$, define $\sigma_{i,j} := f_d(\rho_5(u_{(i,j)}))$ for every $(i,j) \in \mathcal{S}_\star$, with the convention that $\sigma_{i,j}$ is the empty permutation whenever $u_{(i,j)}$ is the empty word. If $\sigma_{i,j} = \star$ for any $(i,j) \in \mathcal{S}_\star$, we set $g_M(w) = \star$. Otherwise, we proceed as follows. 

Let $M_{\pm 1}$ denote the matrix obtained from $M$ by replacing all the $\star$-entries with $0$'s. If $u_0$ is the empty word, we set $F_w'$ to be the empty figure (in the standard gridding $G$). Otherwise, we choose $F_w'$ to be the gridded figure $g_{M_{\pm 1}}(u_0)$ in the standard gridding $G$, where $g_{M_{\pm 1}}$ denotes the function $g_{M'}$ constructed by Albert et. al. \cite{albert2013geometric} for a $\{-1,0,1\}$-matrix, whose construction is also presented in section~\ref{subsetion: albert construction}. 

Let $F_w$ be the gridded figure obtained by starting with $F_{w}'$, and for each $(i,j) \in \mathcal{S}_\star$ adding the permutation $\sigma_{i,j}$ in the $(i,j)$-th cell of $(F_w',G)$. Since $M$ is $\star$-isolated, it is easy to see that the equivalence class of $(F_w,G)$ is independent of the construction, and hence setting $g_M(w)$ to be the equivalence class containing $(F_w,G)$ defines a well-defined function $g_M$. Finally, define the function $f_M: \Sigma_M^* \rightarrow \mathrm{Geom}(M)_{d,k}\cup \{\star\}$ by setting $f_M := \pi_M^\# \circ g_M$, with the convention that $\pi_M^\#(\star) = \star$. 

To complete the construction, for each $w$ with $g_M(w) \neq \star$, we define the functions $\alpha_{w,M}: [|w|] \rightarrow f_M(w)$ and $\beta_{w,M}: [|w|] \rightarrow g_M(w)$ analogously as in section~\ref{sect: construction}. Given $i \in [|w|]$, first consider the case when $w_i$ is contained in $u_0$. In this case,  we define $\beta_{w,M}(i)$ to be the point $\beta_{u_0,M_{\pm 1}}(i')$, where $i'$ is the index of $w_i$ in $u_0$. We also define $\alpha_{w,M}(i)$ analogously. If $w_i$ is contained in  $u_{(i,j)}$ for some $(i,j) \in \mathcal{S}_\star$, we set $\beta_{w,M}(i)$ to be the point corresponding to $\alpha_{\rho_5(u_{(i,j)})}(i') \in \sigma_{i,j}$, where $i'$ is the index of $w_i$ in $u_{(i,j)}$, and the function $\alpha_{\rho_5(u_{(i,j)})}(i')$ is constructed as in section~\ref{sect: construction} for $\widetilde{\mathcal{O}}_k^{[d]}$. Finally, we define $\alpha_{w,M}(i)$ analogously. 

It is straightforward to verify that both of the functions $\alpha_{w,M}$ and $\beta_{w,M}$ are bijections for each $w$ with $g_M(w) \neq \star$. In addition, the construction ensures that for each $(i,j)$ we have 
\begin{align} \label{eq: important observation construction beta}
    \beta_{w,M}(u) \in g_M(w) \cap C_{(i,j)} \text{ if and only if } \rho_6(w_u) = (i,j).     
\end{align}
Having constructed the functions $g_M$ and $f_M$, the rest of the argument splits into two main steps. As a first step, we prove that $f_M$ is both a first-order and a monadic second-order homomorphism. This is carried out in section~\ref{section: MSO correspondence in M}. The proof is relatively straightforward, as the main technical ingredients have already been developed in section~\ref{section: FO and MSO correspondence okd}.

As a second step, we prove the existence of a first-order sentence $\Delta_M$ with the following property: for every $\sigma \in \mathrm{Geom}(M)_{d,k}$, there exists a unique word $w \in \Sigma_M^*$ such that $f_M(w) = \sigma$ and $w \models \Delta_M$. The argument proceeds in two stages. We first prove an analogous uniqueness statement for the function $g_M$. Once this intermediate result has been obtained, the construction of $\Delta_M$ reduces to translating the defining conditions of the left-greedy order into first-order formulae. This part of the argument is presented in section~\ref{section: M bijection restriction}. Although the overall strategy is conceptually natural, some of the proofs are rather lengthy and require somewhat technical details. 
\subsection{Establishing logical correspondences in $\mathrm{Geom}(M)_{d,k}$} \label{section: MSO correspondence in M}

As in section~\ref{section: FO and MSO correspondence okd}, our first goal is to translate atomic formulae and their negations in $\mathrm{TOTO}$ to formulae in $\Sigma_M^*$. Finding the corresponding formulae for  $x = y$ and $x \in z$ is again easy. Finding the corresponding formulae for $x \leq_1 y$ and $x \leq_2 y$ is slightly trickier, but it is a substantially easier having already done most of the work in the construction of the formulae $\xi_{s,d}(x,y)$. Before establishing the logical correspondences, we start with the following analogue of Lemma~\ref{lemma:gamma}.

\begin{lem}\label{lemma:gammaM}
    There exists a first-order sentence $\Gamma_M$ so that $w \models \Gamma_M$ if and only if $f_M(w) \neq \star$. 
\end{lem}
\begin{proof}
Let $\Gamma_M$ be the first-order sentence 
\begin{align*}
&\forall i\forall j\,\Bigl[
 \bigwedge_{(a_1,b_1)\in\mathcal S_\pm}
 \bigwedge_{(a_2,b_2)\in\mathcal S_\star}
 ((\rho_6(w_i)=(a_1,b_1)\land\rho_6(w_j)=(a_2,b_2))\implies i<j)
 \Bigr]\\
&\land\forall i\forall j\,\Bigl[
 \bigwedge_{(a_1,b_1),(a_2,b_2)\in\mathcal S_\star}
 ((\rho_6(w_i)=(a_1,b_1)\land\rho_6(w_j)=(a_2,b_2)\land i<j)
   \implies a_1\leq a_2)\Bigr]\\
&\land\bigwedge_{(a,b)\in\mathcal S_\star}
 \Bigl((\exists u\colon\rho_6(w_u)=(a,b))
       \implies(\Gamma_d)_{(a,b)}\Bigr).
\end{align*}
As before, $(\Gamma_d)_{(a,b)}$ denotes the first-order sentence $\Gamma_d$ applied to the word $\rho_5(u_{(a,b)})$. In order to check that $\Gamma_M$ satisfies the required properties, suppose that $w \models \Gamma_M$. The first line of $\Gamma_M$ implies that $w$ is of the form $w = u_0w'$ for some $w' \in \left(\Sigma_M \setminus \Sigma_{M,\pm}\right)^*$, and the second line implies that each letter in $u_{(i,j)}$ must precede the letters of $u_{(i',j')}$ for every $(i,j),(i',j') \in \mathcal{S}_\star$ with $i < i'$. In particular, we must have $w = u_0u_{(i_1,j_1)}\dots u_{(i_r,j_r)}$. Finally, the last line implies that for every $1 \leq s \leq r$, either $u_{(i_s,j_s)}$ is empty or $\rho_5(u_{(i_s,j_s)}) \models \Gamma_{d}$. The proof that for every $w \in \Sigma_M^*$ with $f_M(w) \neq \star$ we have $w  \models \Gamma_M$ follows analogously. 
\end{proof}
\begin{lem} \label{lemma: easy cases = and in M}
    Let $w \in \Sigma_M^*$ be chosen so that $w \models \Gamma_M$. 
    \begin{enumerate}
        \item Let $\psi(x,y)$ denote the $\mathrm{MSO}$-$\mathrm{TOTO}$ formula given by $x = y$, and $\phi(x',y')$ denote the formula in $\mathrm{MSO}(\Sigma_M^*)$  given by $x' = y'$. We have $w \models \phi(a,b)$ if and only if $f_M(w) \models \psi(\alpha_{w,M}(a), \alpha_{w,M}(b))$. 
        \item Let $\psi(x,z)$ denote the $\mathrm{MSO}$-$\mathrm{TOTO}$ -formula given by $x \in z$, and $\phi(x',z')$ denote the formula in $\mathrm{MSO}(\Sigma_M^*)$ given by $x' \in z'$. We have $w \models \phi(a,c)$ if and only if $f_M(w) \models \psi(\alpha_{w,M}(a), \alpha_{w,M}(c))$. Here, we define $\alpha_{w,M}(c) := \{\alpha_{w,M}(t): t\in c\}$. 
    \end{enumerate}
\end{lem}
\begin{proof}
    The proof is identical to the proof of Lemma~\ref{lemma: easy cases = and in}. 
\end{proof}

We start by giving a brief outline for the construction of the formulae $\xi_{s,M}(x,y)$. Here, it turns out to be convenient to use the function $g_M$ instead of $f_M$. Suppose that $p'$ and $q'$ are chosen in $f_M(w)$ so that $p' \leq_1 q'$, and let $p $ and $q$ denote the corresponding points in $g_M(w)$. If $p$ and $q$ are contained in different columns $i$ and $i'$, then the condition $p \leq_1 q$ is equivalent to $i < i'$. Since the information on $i$ and $i'$ is already encoded in the corresponding letters $\beta_{w,M}^{-1}(p)$ and $\beta_{w,M}^{-1}(q)$, this case can be handled easily. 

Now suppose that $i = i'$. If $p$ and $q$ are contained in the same cell $C_{i,j}$ which corresponds to a $\star$-entry, then both $\beta_{w,M}^{-1}(p)$ and $\beta_{w,M}^{-1}(q)$ belong to the same sub-word $u_{(i,j)}$. Thus, we may use the formulae $\xi_{s,d}$ and Lemma~\ref{lemma: inductive step of leq_1 correspondence} to take care of this case. 

Finally, consider the case when $i = i'$ and the $i$-th column contains (possibly multiple) $\pm 1$-entries. In this case, the construction of $g_{M_{\pm 1}}(u_0)$ ensures that the condition $p \leq_1 q$ can be uniquely determined by the parameter $\delta_i$ and the relative order of the indices $\beta_{w,M}^{-1}(p)$ and $\beta_{w,M}^{-1}(q)$. Indeed, if $\delta_i = 1$, the points are placed in the $i$-th column from `left to right', while if $\delta_i = -1$, the points are placed from `right to left'. 

We now present the formal construction of $\xi_{1,M}(x,y)$ and $\xi_{2,M}(x,y)$. Given $(a,b) \in \mathcal{S}_\star$, write $\xi_{s,d,(a,b)}(x,y)$ for the first-order formula $\xi_{s,d}$ applied to the word $\rho_5(u_{(a,b)})$ (here, we assume that $u_{(a,b)}$ is non-empty and satisfies $f_d(\rho_5(u_{(a,b)})) \neq \star$). We also write $\delta_{=,i}(x,y)$ and $\varepsilon_{=,j}(x,y)$ for the first-order formulae 
\begin{align*}
    (\delta_i = 1 \land x \leq y) \lor (\delta_i=-1 \land y \leq x)
\end{align*}
and 
\begin{align*}
    (\epsilon_j = 1 \land x \leq y) \lor (\epsilon_j = -1 \land y \leq x)
\end{align*}
respectively. Define $\xi_{1,M}(x,y)$ to be the formula 
\begin{align*}
    & \bigvee_{(a_1,b_1) \in \mathcal{S}} \bigvee_{(a_2,b_2) \in \mathcal{S}} \left(\rho_6(w_x) = (a_1,b_1) \land \rho_6(w_y) = (a_2,b_2) \land a_1 <a_2\right) \\
    \lor &  \bigvee_{(a_1,b_1) \in \mathcal{S}_\pm} \bigvee_{(a_2,b_2) \in \mathcal{S}_\pm}  \left(\rho_6(w_x) = (a_1,b_1) \land \rho_6(w_y) = (a_2,b_2) \land a_1 = a_2 \land \delta_{=,a_1}(x,y)\right) \\
    \lor &  \bigvee_{(a,b) \in \mathcal{S}_\star} \left(\rho_6(w_x) = (a,b) \land \rho_6(w_y) = (a,b) \land \xi_{1,d,(a,b)}(x,y)\right),
\end{align*}
and $\xi_{2,M}(x,y)$ to be the formula 
\begin{align*}
    & \bigvee_{(a_1,b_1) \in \mathcal{S}} \bigvee_{(a_2,b_2) \in \mathcal{S}} \left(\rho_6(w_x) = (a_1,b_1) \land \rho_6(w_y) = (a_2,b_2) \land b_1 <b_2\right) \\
    \lor &  \bigvee_{(a_1,b_1) \in \mathcal{S}_\pm} \bigvee_{(a_2,b_2) \in \mathcal{S}_\pm}  \left(\rho_6(w_x) = (a_1,b_1) \land \rho_6(w_y) = (a_2,b_2) \land b_1 = b_2 \land \varepsilon_{=,b_1}(x,y)\right) \\
    \lor &  \bigvee_{(a,b) \in \mathcal{S}_\star} \left(\rho_6(w_x) = (a,b) \land \rho_6(w_y) = (a,b) \land \xi_{2,d,(a,b)}(x,y)\right).
\end{align*}

\begin{lem} \label{lem: leq1 for M}
    For every $w \in \Sigma_M^*$ with $w \models \Gamma_M$ we have $w \models \xi_{s,M}(c_1,c_2)$ if and only if $\alpha_{w,M}(c_1) \leq_s \alpha_{w,M}(c_2)$.  
\end{lem}
\begin{proof}
    Suppose that $w$, $c_1$ and $c_2$ are chosen so that $w \models \Gamma_M$ and $w \models \xi_{1,M}(c_1,c_2)$. Note that the condition $\alpha_{w,M}(c_1) \leq_1 \alpha_{w,M}(c_2)$ is trivially equivalent to the condition $\beta_{w,M}(c_1) \leq_1 \beta_{w,M}(c_2)$. 

    First consider the case when the formula on the first line is satisfied for $x = c_1, y= c_2$ and for some $(a_1,b_1)$ and $(a_2,b_2) \in \mathcal{S}$. Using observation~\eqref{eq: important observation construction beta}, $\beta_{w,M}(c_1)$ belongs to the $a_1$-th column and $\beta_{w,M}(c_2)$ belongs to the $a_2$-th column of the gridded permutation $g_M(w)$. Since $a_1 < a_2$, we must have $\beta_{w,M}(c_1) \leq_1 \beta_{w,M}(c_2)$, as required. 

    We now consider the case when the formula on the second line is satisfied for some $(a_1,b_1)$ and $(a_2,b_2)$, and write $a_* = a_1 = a_2$. Thus, the points $\beta_{w,M}(c_1)$ and $\beta_{w,M}(c_2)$ belong to the same column of $g_M(w)$, whose index we denote by $a_*$, and every non-zero entry in the column is a $\pm 1$-entry. First suppose that $\delta_{a_*} = 1$. Since $w \models \delta_{=,a_1}(c_1,c_2)$, it follows that $c_1 \leq c_2$. By the construction of the function $g_M$, this is clearly equivalent to $\beta_{w,M}(c_1) \leq_1 \beta_{w,M}(c_2)$. The case $\delta_{a_*} = -1$ follows analogously from the construction of the formula $\delta_{=,a_*}(x,y)$. 

    Finally, we consider the case when the formula on the last line is satisfied for some $(a',b')$. Thus the points $\beta_{w,M}(c_1)$ and $\beta_{w,M}(c_2)$ belong to the same cell $C_{a',b'}$ of $g_M(w)$, which is indexed by a $\star$-entry, and which contains the permutation $f_{d}(\rho_5(u_{(a',b')}))$. Since $w \models \xi_{1,d,(a',b')}(c_1,c_2)$, Lemma~\ref{lemma: inductive step of leq_1 correspondence} implies that $\beta_{w,M}(c_1) \leq_1 \beta_{w,M}(c_2)$, as required. The proof of the converse direction and the proof for $\xi_{2,M}$ follow analogously. 
\end{proof}

\begin{thm}\label{thm: FO correspondence for M}
    Let $v^{(f)}_1,\dots,v^{(f)}_{c_1},v^{(b)}_1,\dots,v^{(b)}_{c_2}$ denote variables (either first or second-order variables), with $v^{(f)}_i$'s corresponding to free and $v^{(b)}_i$'s corresponding to bound variables. Let $\psi(v^{(f)}_1,\dots,v^{(f)}_{c_1})$ denote a monadic second-order formula in $\mathrm{TOTO}$ given in disjunctive normal form, i.e.\
    \begin{align*} 
        \psi(v^{(f)}_1,\dots,v^{(f)}_{c_1}) = Q_1 v^{(b)}_1\dots Q_{c_2} v^{(b)}_{c_2} \left(\left(p_{1,1} \land \dots \land p_{1,a_1}\right) \lor \dots \lor \left(p_{b,1} \land \dots  \land p_{b,a_b}\right)\right),
    \end{align*}
    where each $p_{i,j}$ is an atomic formula or its negation using variables in $v^{(f)}_1,\dots,v^{(f)}_{c_1},v^{(b)}_1,\dots,v^{(b)}_{c_2}$. There exists a monadic second-order formula $\phi$ in $\Sigma_M^*$ of the form 
    \begin{align*}  
        \phi(\nu^{(f)}_1,\dots,\nu^{(f)}_{c_1}) = Q_1 \nu^{(b)}_1 \dots Q_{c_2} \nu^{(b)}_{c_2}: \left(q_{1,1} \land \dots \land q_{1,a_1}\right) \lor \dots \lor \left(q_{b,1} \land \dots \land q_{b,a_b}\right)
    \end{align*}
    satisfying the following properties:
    \begin{enumerate}
        \item each variable $\nu^{(f)}_i$ and $\nu^{(b)}_j$ has the same order as the respective variable $v^{(f)}_i$ and $v^{(b)}_j$;
        \item each formula $q_{i,j}$ depends only on the formula $p_{i,j}$; and 
        \item for each $a_1,\dots,a_{c_1}$ and $w \in \Sigma_M^*$ satisfying $w \models \Gamma_M$, we have $w \models \phi(a_1,\dots,a_{c_1})$ if and only if $f_M(w) \models \psi(\alpha_{w,M}(a_1),\dots,\alpha_{w,M}(a_{c_1}))$. 
    \end{enumerate}

    In addition, if $\psi$ is a first-order formula, then so is $\phi$.
\end{thm}

\begin{proof}
    The proof of Theorem~\ref{thm: FO correspondence for M} is identical to the proof of Theorem~\ref{thm: FO correspondence}, apart from the following trivial modifications. The constructions of the formulae $q_{i,j}$ are identical except for using the formulae $\xi_{s,M}$ in place of $\xi_{s,d}$. In addition, Lemmas~\ref{lemma: easy cases = and in M} and~\ref{lem: leq1 for M} are used instead of Lemmas~\ref{lemma: easy cases = and in} and~\ref{lemma: inductive step of leq_1 correspondence} respectively. We omit the details. 
\end{proof}
\subsection{Restricting to the bijection} \label{section: M bijection restriction}

The purpose of this sub-section is to prove the following result. 
\begin{thm} \label{thm: fM is a bijection}
    There exists a first-order sentence $\Delta_M$ so that for every $\sigma \in \mathrm{Geom}(M)_{d,k}$, there exists a unique $w \in \Sigma_M^*$ with $w \models \Delta_M$ and $f_M(w) = \sigma$. 
\end{thm}
As an initial step, we start by proving the gridded variant of Theorem~\ref{thm: fM is a bijection}, i.e.\ the variant in which the  function $f_M$ is replaced with $g_M$. We say that two letters $a_{i,j}$ and $a_{i',j'} \in \Sigma_{M,\pm}$ are \emph{unrelated} if and only if $i \neq i'$ and $j \neq j'$. We say that a word $w'$ can be obtained from $w \in \Sigma_{M,\pm}^*$ by applying a \emph{local swap} if there exist indices $i$ and $i+1$ so that $w_i$ and $w_{i+1}$ are unrelated, and $w'$ is obtained by swapping these two letters. 

It is evident that for any $w'$ which can be obtained from $w$ by applying a local swap we have $g_M(w) = g_M(w')$. Define an equivalence relation $\sim$ on $\Sigma_{M,\pm}^*$ so that $w\sim w'$  if and only if there exists a finite sequence $w_0,\dots,w_k$ with $w_0 = w$, $w_k = w'$, and $w_{i+1}$  can be obtained from $w_i$ by applying a local swap for every $0 \leq i \leq k-1$. We start by recalling the following result of Albert et.\ al.\ from \cite{albert2013geometric}, which we have rephrased slightly. 
\begin{lem}\cite[Proposition 7.1]{albert2013geometric}\label{lemma: albert equivalence relation}
    For every $w,w' \in \Sigma_{M,\pm}^*$ we have $g_M(w) = g_M(w')$ if and only if $w \sim w'$. 
\end{lem}
In \cite{albert2013geometric}, the authors observed that by using standard properties of trace monoids, one may construct a regular language $L$ which contains a unique representative from each equivalence class of $\sim$. For our purposes, it turns out to be convenient to phrase their result in a way that is easier to translate to a first-order sentence. As the proof of the following Lemma is analogous to the proof of Proposition 7.1 in \cite{albert2013geometric}, we omit the proof. 
\begin{lem}\label{lemma: no stars preimages given by equivalence relation}
    Suppose that $M$ is a multiplicative  matrix whose entries are in $\{-1,0,1\}$. For every $\sigma_\# \in \mathrm{Geom}^\#(M)$ there exists a unique $w$ for which $g_M(w) = \sigma_\#$, and for every pair of letters $w_{\ell_1} = a_{i,j}$ and $w_{\ell_2} = a_{i',j'}$ with $\ell_1 < \ell_2$ at least one of the following conditions holds: (i) $i = i'$, (ii) $j = j'$, or (iii) if for every $\ell_3 \in (\ell_1,\ell_2) \cap \mathbb{Z}$ the letter $w_{\ell_3}$ is unrelated with $w_{\ell_2}$, then $(i,j) <_{lex} (i',j')$. We call such $w$ the lexicographically minimal $w$. 
\end{lem}

\begin{lem}\label{lemma: key thm for gM instead}
    There exists a first-order sentence $\Delta_{M,\#}$ so that for every $\sigma_\# \in \mathrm{Geom}(M)^\#_{d,k}$, there exists a unique $w \in \Sigma_M^*$ with $w \models \Delta_{M,\#}$ and $g_M(w) = \sigma_\#$. 
\end{lem}
\begin{proof}
Write $u(a,b)$ as a shorthand for the statement that is true if and only if the letters $a$ and $b$ are unrelated. Recall that $\mathcal{S}_\pm$ denotes the indexes of the $\pm 1$ entries of $M$. For simplicity, for variables $x$, $y$ and $z$ we write $x' = \rho_6(w_x)$, $y' = \rho_6(w_y)$ and $z' = \rho_6(w_z)$. Let $\eta$ denote the first-order sentence
\[
\neg \exists x \exists y \left(x' \in \mathcal{S}_\pm \land y' \in \mathcal{S}_\pm \land x < y \land x' >_{lex} y'\land \forall z\left(x \leq z < y \implies u\left(w_z,w_y\right)\right) \right).
\]
It is easy to check that $\eta$ matches the condition of Lemma~\ref{lemma: no stars preimages given by equivalence relation}.

Let $\Delta_{M,\#}$ be the first-order sentence defined by 
\begin{align*}
\Delta_{M,\#} = \Gamma_M \land \eta \land \bigwedge_{(i,j)\in\mathcal{S}_\star}
\left(
\left(\exists x\colon \rho_6(w_x)=(i,j)\right)
\implies (\Delta_d)_{(i,j)}
\right).
\end{align*}
It is easy to check that $\Delta_{M,\#}$  satisfies the required properties. Indeed, $\eta$ (and $\Gamma_M$) ensure that whenever $u_0$ is non-empty, the word $u_0$ corresponds to the unique representative word in the trace monoid by Lemma~\ref{lemma: no stars preimages given by equivalence relation}. 

For each $(i,j) \in \mathcal{S}_\star$, if the respective permutation in $C_{i,j}$ is non-empty, then it has a unique left-greedy representative by Theorem~\ref{thm:Main O_K}. Thus the second line ensures that either the respective sub-word $u_{(i,j)}$ is empty, or is the unique word which satisfies $\Delta_d$ and whose image coincides with $(\sigma_\#)_{i,j}$. This completes the proof. 
\end{proof}
As different elements of $\mathrm{Geom}(M)_{d,k}$  can have a varying number of gridded figures corresponding to them,  Lemma~\ref{lemma: key thm for gM instead} on its own is not particularly useful for proving  Theorem~\ref{thm: fM is a bijection}, nor even using it as a tool for proving a convergence law.  To overcome this difficulty, we recall from Definition~\ref{def: left greedy} the notion of left-greedy order $<_{LG}$ introduced by Albert et al, which for every $\sigma$ induces a total order on the set $G(\sigma)$. In particular, the left-greedy order gives a unique minimal representative in $G(\sigma)$ for each $\sigma$, which is called the \emph{left-greedy representation} of $\sigma$. 

Our next goal is to construct first-order formulae which correspond to the conditions in the definition of the left-greedy order, analogously as in section~\ref{section: restricting to bijection}. Using this approach, we construct a first-order sentence $\hat{\Delta}_M$ for which $w \models \hat{\Delta}_M$ if and only if $g_M(w)$ is the left-greedy representation of the underlying permutation $f_M(w)$. Together with Lemma~\ref{lemma: key thm for gM instead}, we conclude that $\Delta_M := \hat{\Delta}_M \land \Delta_{M,\#}$ is a suitable first-order sentence to satisfy the conditions of Theorem~\ref{thm: fM is a bijection}. 

We move on to proving a sequence of preliminary Lemmas that will be required in the proof of Theorem~\ref{thm: fM is a bijection}. 

\begin{lem} \label{lemma: cells indexed by A}
    Let $A \subseteq \mathcal{S}$ be a collection of non-zero entries of $M$. There exists a first-order formula $\psi_A(x_{(i,j),u,v})$ involving variables $x_{(i,j),u,v}$ for $(i,j) \in A$, $u,v\in \{1,2\}$ so that $\sigma \models \psi_A(c_{(i,j),u,v})$ for some constants $c_{(i,j),u,v}$ if and only if there exists a gridded permutation $\sigma_\# \in G(\sigma)$  so that the following conditions are satisfied:
    \begin{enumerate}
        \item[(i)] the non-empty cells of $\sigma_\#$ are indexed exactly by $A$; and 
        \item[(ii)] for every $(i,j) \in A$, we have $p \in \sigma_\# \cap C_{i,j}$ if and only if $p \in [c_{(i,j),1,1},c_{(i,j),1,2}]_1 \times [c_{(i,j),2,1},c_{(i,j),2,2}]_2$. Moreover, each $c_{(i,j),u,v}$ belongs to the cell $\sigma_\# \cap C_{i,j}$. 
    \end{enumerate}
\end{lem}
\begin{proof}
    Given variables $x_{(i,j),u,v}$ for $(i,j) \in A$ and $u,v\in \{1,2\}$, write $(\upsilon_{d,k})_{i,j}$ for the first-order formula $\upsilon_{d,k}(x_{(i,j),1,1},x_{(i,j),2,1},x_{(i,j),1,2},x_{(i,j),2,2})$, which can be used to check whether the sub-permutation induced by the points in the  rectangle $[x_{(i,j),1,1},x_{(i,j),1,2}]_1\times[x_{(i,j),2,1},x_{(i,j),2,2}]_2$ is contained in $\widetilde{\mathcal O}_k^{[d]}$. Let $\psi_A(x_{(i,j),u,v})$ denote the first-order formula
    \begin{align*}
        & \forall p: \left(\bigvee_{(i,j) \in A} \left(p \in [x_{(i,j),1,1}, x_{(i,j),1,2}]_1 \times [x_{(i,j),2,1}, x_{(i,j),2,2}]_2\right)\right) \\
        & \land \left(\bigwedge_{(i,j) \in A, u,v\in \{1,2\}} \left(x_{(i,j),u,v} \in [x_{(i,j),1,1}, x_{(i,j),1,2}]_1 \times [x_{(i,j),2,1}, x_{(i,j),2,2}]_2\right)\right) \\
        & \land \bigwedge_{(i,j) \in \mathcal{S}_\star \cap A} (\upsilon_{d,k})_{i,j} \\
        & \land \bigwedge_{(i,j) \in \mathcal{S}_+ \cap A} \left(\forall p\forall q: \left(p,q \in [x_{(i,j),1,1}, x_{(i,j),1,2}]_1 \times [x_{(i,j),2,1}, x_{(i,j),2,2}]_2 \right) \implies (p \leq_1 q \iff p  \leq_2 q)\right)\\
        & \land \bigwedge_{(i,j) \in \mathcal{S}_- \cap A} \left(\forall p\forall q: \left(p,q \in [x_{(i,j),1,1}, x_{(i,j),1,2}]_1 \times [x_{(i,j),2,1}, x_{(i,j),2,2}]_2 \right) \implies (p \leq_1 q \iff q  \leq_2 p)\right)\\
        & \land \bigwedge_{(i,j),(i',j') \in A} \left(i < i' \implies x_{(i,j),1,2} <_1 x_{(i',j'),1,1}\right) \\
        & \land \bigwedge_{(i,j),(i',j') \in A} \left(j < j' \implies x_{(i,j),2,2} <_2 x_{(i',j'),2,1}\right).
    \end{align*}
    Here, we use shorthand $p,q \in [x,y]_1\times [z,w]_2$ for the formula $(p \in [x,y]_1 \times [z,w]_2) \land (q \in [x,y]_1 \times [z,w]_2)$. 
    In order to check that $\psi_A$ satisfies the required properties, suppose that $\sigma \models \psi_{A}(c_{(i,j),u,v})$. For each $(i,j) \in A$, let $\theta_{i,j}$ denote the set of points $p \in [c_{(i,j),1,1}, c_{(i,j),1,2}]_1 \times [c_{(i,j),2,1}, c_{(i,j),2,2}]_2$. Note that at this stage we have not verified that $\theta_{i,j}$'s are disjoint nor that they are sub-intervals of $\sigma$.

    The first line implies that every $p$ belongs to at least one $\theta_{i,j}$, and the second line implies that $c_{(i,j),u,v}$'s are the smallest and largest points in $\theta_{i,j}$ under both orders respectively. In addition, the last two lines imply that whenever $i < i'$, every point $p \in \theta_{i,j}$ and every point $p' \in \theta_{i',j'}$ satisfy the condition $p <_1 p'$, and an analogous condition holds for the $\leq_2$-order whenever $j < j'$. In particular, every point $p$ belongs to exactly one $\theta_{i,j}$. 

    The third line implies that $\theta_{i,j} \in \widetilde{\mathcal{O}}_k^{[d]}$ for each $(i,j)$ with $M_{i,j} = \star$, and the fourth and fifth lines imply that the points in $\theta_{i,j}$'s with $M_{i,j} \in \{\pm 1\}$ form a monotone sub-permutation of correct direction. As these points are only guaranteed to form a monotone sub-permutation, but they do not necessarily lie on the main diagonal, it follows that there exists $\sigma_{\#} \in \mathrm{Grid}^\#(M)_{d,k}$ which satisfies the conditions of the Lemma.  To complete the proof, recall the result of Albert et. al. \cite{albert2013geometric} (stated earlier as Proposition~\ref{prop: equality grid and geometric grid when HM is a forest}), which states that $\mathrm{Grid}^\#(M)_{d,k} = \mathrm{Geom}^\#(M)_{d,k}$ whenever $M$ is a forest, and thus applies to $M_{\pm 1}$. Thus we may insert the points in the monotone cells on appropriate diagonals, which completes the proof. The proof of the converse direction follows analogously, we omit the details. 
\end{proof}
From now on, we write $\lambda_{\Box}(\nu, \nu_{u,v})$ for the first-order formula obtained by applying Theorem~\ref{thm: FO correspondence for M} to the first-order formula $x \in [x_{1,1},x_{1,2}]_1 \times [x_{2,1},x_{2,2}]_2$, with the variable $\nu$ corresponding to $x$, and each $\nu_{u,v}$ corresponding to $x_{u,v}$. 
\begin{lem} \label{lemma: word equivalent of cells indexed by A}
    Let $\phi_A^{(1)}(\boldsymbol\nu)$ be the first-order formula on words that is obtained by applying Theorem~\ref{thm: FO correspondence for M} to $\psi_A$, where $\boldsymbol{\nu} = (\nu_{(i,j),u,v}: (i,j) \in A,\ u,v  \in \{1,2\})$. Set
    \[
    \phi_A(\boldsymbol\nu):=
    \Gamma_M\land\phi_A^{(1)}(\boldsymbol\nu) \land\bigwedge_{(i,j)\in A}\forall t\,\bigl(\lambda_\Box(t;\nu_{(i,j),u,v}) \iff\rho_6(w_t)=(i,j)\bigr).
    \]
    We have $w \models \phi_A(c_{(i,j),u,v})$ if and only if 
    \begin{enumerate}
        \item[(i)] the non-empty cells of $g_M(w)$ are exactly those indexed by $A$; and 
        \item[(ii)] for each $(i,j) \in A$, the $g_M(w) \cap C_{i,j}$ consists of  those points $p$ that are contained in $[\beta_{w,M}(c_{(i,j),1,1}),\beta_{w,M}(c_{(i,j),1,2})]_1 \times [\beta_{w,M}(c_{(i,j),2,1}),\beta_{w,M}(c_{(i,j),2,2})]_2$. Moreover, each \\$\beta_{w,M}(c_{(i,j),u,v})$ belongs to $g_M(w) \cap C_{i,j}$. 
    \end{enumerate}
\end{lem}
\begin{proof}
    The proof follows easily from Theorem~\ref{thm: FO correspondence for M} and  Lemma~\ref{lemma: cells indexed by A}, and by noticing that the last part of the formula $\phi_A$ ensures that the points $\beta_{w,M}(c_{(i,j),u,v})$ are always contained in $g_M(w) \cap C_{i,j}$ for every $(i,j) \in A$, and $u,v\in \{1,2\}$. 
\end{proof}
\begin{lem}\label{lemma: left-greedy comparison lemma}  
    Let $A, B \subseteq \mathcal{S}$. There exists a first-order formula $\lambda_{A,B}(x_{(i,j),u,v}, y_{(i',j'),u',v'})$ with variables $x_{(i,j),u,v}$ for $(i,j) \in A$, $u,v \in \{1,2\}$ and $y_{(i',j'),u',v'}$ for $(i',j')\in B$, $u',v' \in \{1,2\}$ so that  $\sigma \models \lambda_{A,B}(a_{(i,j),u,v},b_{(i',j'),u',v'})$ for some constants $a_{(i,j),u,v}$ and $b_{(i',j'),u',v'}$ if and only if both of the following conditions are satisfied:
    \begin{enumerate}
        \item[(i)] there exist gridded permutations $\sigma_{\#,a}$, $\sigma_{\#,b} \in G(\sigma)$ so that $\sigma_{\#,a}$ satisfies the conditions of Lemma~\ref{lemma: cells indexed by A} for the set $A$ and for constants $a_{(i,j),u,v}$, and  $\sigma_{\#,b}$ satisfies the conditions for the set $B$ and for constants $b_{(i,j),u,v}$; and 
        \item[(ii)] we have $\sigma_{\#,a} \leq_{LG} \sigma_{\#,b}$. 
    \end{enumerate}
\end{lem}
\begin{proof}
     Write $A_i = \{(i,j): (i,j) \in A\}$ and define $B_i$ analogously. For notational convenience, write $\underline{x}_A : = (x_{(i,j),u,v}: (i,j) \in A, u,v \in \{1,2\})$ and $\underline{y}_B = (y_{(i,j),u,v}: (i,j) \in B, u,v\in \{1,2\})$.  Given a finite set of variables $x_1,\dots,x_k$ write $y = \max_1(x_1,\dots,x_k)$ for the point $y \in \{x_1,\dots,x_k\}$ chosen so that $x_i \leq_1 y$ for all $i \in [k]$. It is easy to prove that $y$ can be expressed as a first-order formula. We also write $m_i(\underline{x}_A)$ and $m_i(\underline{y}_B)$ as shorthands for $\max_1(x_{(i,j),1,2}: j \text{ chosen so that } (i,j) \in A)$ and $\max_1(y_{(i,j),1,2}: j\text{ chosen so that } (i,j) \in B)$ respectively. Note that we define $m_i(\cdot)$ only when the maximum is taken over a non-empty set. 
     
     For each pair $(t,s)$, write$A_{<(t,s)} = \{(i,j)\in A: i < t, \text{ or } i=t\text{ and } j<s\}$, and define $B_{<(t,s)}$ analogously. For convenience, we slightly abuse the notation by allowing $s = \infty$, whence $A_{<(t,\infty)} = \{(i,j) \in A: i \leq t\}$. 
     
     Given sets $A$ and $B$ with $A_{<(t,s)} = B_{<(t,s)}$ define auxiliary formula $\omega_{<(t,s),A,B}(\underline{x}_A,\underline{y}_B)$ as 
    \begin{align*}
        \bigwedge_{(i,j) \in A_{<(t,s)}} \bigwedge_{u,v\in \{1,2\}} (x_{(i,j),u,v} = y_{(i,j),u,v}). 
    \end{align*}
    For an integer $i'$ with $A_{<(i',\infty)} = B_{<(i',\infty)}$, define $\lambda_{i',A,B}(\underline{x}_A,\underline{y}_B)$ to be the formula 
    \begin{align*}
        & \bigvee_{\substack{1\leq i\leq i'\\ A_i=B_i\ne\emptyset}} \left(\omega_{<(i-1,\infty),A,B}(\underline{x}_A, \underline{y}_B) \land \left(m_i(\underline{y}_B) <_1 m_i(\underline{x}_A)\right)\right)\\
        & \lor \bigvee_{\substack{1\leq i\leq i'\\ A_i=B_i\ne\emptyset}}\left(\omega_{<(i-1,\infty),A,B}(\underline{x}_A, \underline{y}_B) \land \left(m_i(\underline{y}_B) = m_i(\underline{x}_A)\right)\right.\\ &\qquad\left.\land \bigvee_{j: (i,j) \in A} \left(\omega_{<(i,j),A,B}(\underline{x}_A,\underline{y}_B)\land y_{(i,j),2,2}<_2 x_{(i,j),2,2}\right)\right).
    \end{align*}
    We also use the convention that $\lambda_{0,A,B}(\underline{x}_A,\underline{y}_B)$ is the false sentence. We start by proving the following preliminary claim, which describes how $\lambda_{i',A,B}$ can  be used to determine the relative order of $\sigma_{\#,a}$ and $\sigma_{\#,b}$ with respect to $<_{LG}$. 
    \begin{claim}\label{claim: lambda claim}
        Let $i'$ be chosen so that $A_{<(i',\infty)} = B_{<(i',\infty)}$. Given constants $\underline{a}_A$ and $\underline{b}_B$, suppose that we have $\sigma \models \psi_A(\underline{a}_A)$ and $\sigma \models \psi_B(\underline{b}_B)$. Then we have $\sigma \models \lambda_{i',A,B}(\underline{a}_A,\underline{b}_B)$ if and only if $\sigma_{\#,a} <_{LG} \sigma_{\#,b}$, and if at least one of the following conditions is satisfied: 
        \begin{enumerate}
            \item[(i)] the relation $\sigma_{\#,a} <_{LG} \sigma_{\#,b}$ is witnessed when condition 1 in Definition~\ref{def: left greedy}  is satisfied for some $i_* \leq i'$; or
            \item[(ii)] the relation $\sigma_{\#,a} <_{LG} \sigma_{\#,b}$ is witnessed when condition 2 in Definition~\ref{def: left greedy}  is satisfied for some $(i_*,j_*)$ with $i_* \leq i'$. 
        \end{enumerate}
    \end{claim}
    \begin{proof}
        The proof follows by observing that 
        \begin{enumerate}
            \item[(i)]$\sigma$ satisfying the condition on the first row of $\lambda_{i',A,B}$ with $i = i_*$ is equivalent to the first condition being satisfied in Definition~\ref{def: left greedy}  when $i = i_*$; and 
            \item[(ii)] $\sigma$ satisfying the condition on the second row of $\lambda_{i',A,B}$ with $i = i_*$ and $j = j_*$ is equivalent to the second condition being satisfied in Definition~\ref{def: left greedy}  with $i = i_*$ and $j = j_*$.  
        \end{enumerate}
    \end{proof}
    
    \noindent We now construct the formulae $C_{A,B}(\underline{x}_A, \underline{y}_B)$. When $A = B$, we set 
    \begin{align*}
        C_{A,A}(\underline{x}_A, \underline{y}_A) = \lambda_{\ell,A,A}(\underline{x}_A,\underline{y}_A) \lor \omega_{<(\ell,\infty),A,A}(\underline{x}_A,\underline{y}_A).
    \end{align*}
    By Claim~\ref{claim: lambda claim}, the first part of $C_{A,A}$ is equivalent to the condition that $\sigma_{\#,a} <_{LG} \sigma_{\#,b}$, while the second part is equivalent to the condition that $\sigma_{\#,a} = \sigma_{\#,b}$. 

    Now suppose that $A \neq B$, and let $(i',j')$ be the lex-minimal element in $A \Delta B$. In particular, we have $A_{<(i',j')} = B_{<(i',j')}$. Note that for any choice of constants $\underline{a}_A$ and $\underline{b}_B$ we must have $\sigma_{\#,a} \neq \sigma_{\#,b}$. In addition, the left-greedy order between $\sigma_{\#,a}$ and $\sigma_{\#,b}$ must be determined in column $i_*$ or cell $(i_*,j_*)$ for some $i_* \leq i'$. We split the construction into three cases based on the structure of $A$ and $B$ in the $i'$-th column of $M$.  We write $A_{i'} = \{(i',j): (i',j) \in A\}$, and define $B_{i'}$ analogously. 

    \textbf{Case 1, $A_{i'} = \emptyset$: } Since $(i',j') \in A\Delta B$, we must have $B_{i'} \neq \emptyset$. We claim that we may take $C_{A,B}(\underline{x}_A, \underline{y}_B) = \lambda_{i'-1,A,B}(\underline{x}_A, \underline{y}_B)$. Note that whenever the condition $\sigma_{\#,a} <_{LG} \sigma_{\#,b}$ is satisfied, either the first condition in Definition~\ref{def: left greedy} holds in some column $i_*$ with $i_* \leq i'-1$, or the second condition holds in some cell $(i_*,j_*)$ with $i_* \leq i'-1$. Indeed, by our previous observation,  both of these statements are true when the upper bound on $i_*$ is replaced with $i'$ instead of $i'-1$. However, the case $i_* = i'$ is impossible as $A_{i'} = \emptyset$, whence $\sigma_{\#,a}$ cannot contain any points in its $i'$-th column (while $\sigma_{\#,b}$ is guaranteed to contain at least one point in its $i'$-th column as $B_{i'} \neq \emptyset$). Thus Claim~\ref{claim: lambda claim} implies that the formula $C_{A,B}$ satisfies the required properties. 

    \textbf{Case 2, $B_{i'} = \emptyset$:} As above, we must have $A_{i'} \neq \emptyset$. We claim that 
    \begin{align*}
         C_{A,B}(\underline{x}_A,\underline{y}_B) = \lambda_{i'-1,A,B}(\underline{x}_A,\underline{y}_B) \lor \omega_{<(i'-1,\infty),A,B}(\underline{x}_A,\underline{y}_B)
    \end{align*}
    satisfies the required properties. As above, the first part of $C_{A,B}$ covers the case when $\sigma_{\#,a} <_{LG} \sigma_{\#,b}$ is witnessed when the first or second condition in Definition~\ref{def: left greedy} is satisfied either in column $i_*$ or cell $(i_*,j_*)$ respectively for some $i_* < i'$. Note that the second part is equivalent to the statement that both $\sigma_{\#,a}$ and $\sigma_{\#,b}$ contain equal number of points in each cell $(i,j)$ with $i \leq i'-1$, which immediately implies that $\sigma_{\#,a} <_{LG} \sigma_{\#,b}$, as we have $B_{i'} = \emptyset$ and $A_{i'} \neq \emptyset$. 

    \textbf{Case 3, Both $A_{i'}, B_{i'}$ are non-empty:} Since both $A_{i'}$ and $B_{i'}$ are non-empty, both $\sigma_{\#,a}$ and $\sigma_{\#,b}$ must contain a point in their $i'$-th columns. Define $C_{A,B}(\underline{x},\underline{y})$ to be the formula 
    \begin{align*}
        & \lambda_{i'-1,A,B}(\underline{x}_A,\underline{y}_B) \\
        & \lor \left(\omega_{<(i'-1,\infty),A,B}(\underline{x}_A, \underline{y}_B) \land m_{i'}(\underline{y}_B) <_1 m_{i'}(\underline{x}_A)\right) \\
        & \lor \bigvee_{j:\,(i',j)\in A,\ j<j'} \left(\omega_{<(i',j),A,B}(\underline{x}_A, \underline{y}_B) \land m_{i'}(\underline{y}_B) =m_{i'}(\underline{x}_A)\land  y_{(i',j), 2,2} <_2  x_{(i',j),2,2}\right) \\
        & \lor \left(\omega_{<(i',j'),A,B}(\underline{x}_A, \underline{y}_B) \land m_{i'}(\underline{y}_B) =m_{i'}(\underline{x}_A)\land (i',j') \in A\right),
    \end{align*}
    where the last line is only included if $(i',j') \in A$. As before, the first line of $C_{A,B}$ corresponds to the case when  $\sigma_{\#,a} <_{LG} \sigma_{\#,b}$ is witnessed either in column $i_*$ or cell $(i_*,j_*)$ with $i_* < i'$. In addition, the second line corresponds to the case when $\sigma_{\#,a} <_{LG} \sigma_{\#,b}$ is witnessed when the first condition in Definition~\ref{def: left greedy} holds in column $i'$, while the third line corresponds to the case when the second condition holds in cell $(i',j)$ with $j < j'$. Finally, the last line corresponds to the case when the second condition holds in the cell $(i',j')$, which may only happen when $(i',j') \in A$. 

    Finally, we define 
    \[\lambda_{A,B}(\underline{x}_A,\underline{y}_B) = \psi_{A}(\underline{x}_A) \land \psi_{B}(\underline{y}_B) \land C_{A,B}(\underline{x}_A,\underline{y}_B).\]
    Then $\lambda_{A,B}$ satisfies the required properties as $\psi_A$ and $\psi_B$ ensure that there exist gridded permutations associated to the tuples $x_A$ and $y_B$, and $C_{A,B}$ ensures that ¨$\sigma_{\#,a} \leq_{LG} \sigma_{\#,b}$. This completes the proof. 
\end{proof}
\begin{lem}\label{lemma: left-greedy drawing, which has its extremal points in A}
    Let $A \subseteq{\mathcal{S}}$. There exists a first-order formula $\Delta_{A}(x_{(i,j),u,v})$ with variables $x_{(i,j),u,v}$ for $(i,j) \in A$,  $u,v \in \{1,2\}$ for which the following holds. We have $\sigma \models \Delta_{A}(a_{(i,j),u,v})$ for some constants $a_{(i,j),u,v}$ if and only if $\sigma \models \psi_{A}(a_{(i,j),u,v})$, where $\psi_A$ is the formula given by Lemma~\ref{lemma: cells indexed by A}, and the gridded permutation $\sigma_\#$ satisfying the conclusions of Lemma~\ref{lemma: cells indexed by A} is the left-greedy representation of $\sigma$. 
\end{lem}
\begin{proof}
    Given $B \subseteq \mathcal{S}$ and variables $y_{(i',j'), u',v'}$ with $(i',j') \in B$, $u',v' \in \{1,2\}$, write 
    \[\hat{\Delta}_{B}(x_{(i,j),u,v},y_{(i',j'), u',v'}),\] 
    where $(i,j) \in A$, $u,v\in \{1,2\}$, for the first-order formula 
    \begin{align*}
        \psi_B\left(y_{(i',j'),u',v'}\right) \implies \lambda_{A,B}\left(x_{(i,j),u,v}, y_{(i',j'),u',v'}\right),
    \end{align*}
    and set $\Delta_A(x_{(i,j),u,v})$ to be the formula  
    \begin{align*}
        & \psi_{A}(x_{(i,j),u,v}) \\
        & \land \bigwedge_{B \subseteq \mathcal{S}} \left(\forall_{(i',j') \in B, u',v' \in \{1,2\}} \text{ } y_{(i',j'), u',v'}: \psi_{B}(y_{(i',j'), u',v'}) \implies \lambda_{A,B}\left(x_{(i,j),u,v}, y_{(i',j'), u',v'}\right)\right).
    \end{align*}
    To check that $\Delta_A$ satisfies the required properties, suppose that $\sigma \models \Delta_A(a_{(i,j),u,v})$. Let $\sigma_\#$ be the gridded permutation satisfying the conclusion of Lemma~\ref{lemma: cells indexed by A} (whose existence is provided by the top row of $\Delta_A$). We claim that $\sigma_\#$ is the left-greedy representation of $\sigma$. 

    Let $\phi_\#$ be an arbitrary element in $G(\sigma)$, and let $B \subseteq \mathcal{S}$ be chosen so that $\phi_\# \cap C_{i,j} \neq \emptyset$ if and only if $(i,j) \in B$. For each $(i',j') \in B$ and $u',v' \in \{1,2\}$,  let $b_{(i',j'),u',v'}$ be chosen so that $\phi_\# \cap C_{i',j'} = [b_{(i',j'),1,1}, b_{(i',j'),1,2}]_1 \times [b_{(i',j'),2,1}, b_{(i',j'),2,2}]_2$. 

    These choices ensure that  $\sigma \models \psi_B(b_{(i',j'),u',v'})$, and hence the second line of $\Delta_A$ implies that $\sigma \models \lambda_{A,B}\left(a_{(i,j),u,v},b_{(i',j'),u',v'}\right)$. Thus Lemma~\ref{lemma: left-greedy comparison lemma} implies that $\sigma_\# \leq_{LG} \phi_\#$, and since $\phi_\#$ is an arbitrary member of $G(\sigma)$, it follows that $\sigma_\#$ is the left-greedy representation of $\sigma$. The converse direction follows analogously, we omit the details. 
\end{proof}

\begin{proposition}\label{prop: penultimate step}
    There exists a first-order sentence $\hat{\Delta}_M$ for which $w \models \hat{\Delta}_M$ if and only if $g_M(w)$ is the left-greedy representation of the underlying permutation $f_M(w)$.
\end{proposition}

\begin{proof}
    Recall that $\lambda_{\Box}(\nu_{i,j},\nu_{(i,j),u,v})$ is defined to be the first-order formula obtained by applying Theorem~\ref{thm: FO correspondence for M} to the formula  $x_{i,j} \in [x_{(i,j),1,1},x_{(i,j),1,2}]_1 \times [x_{(i,j),2,1},x_{(i,j),2,2}]_2.$ We also write $\hat{\Delta}_A(\nu_{(i,j),u,v})$ for the first-order formula obtained by applying Theorem~\ref{thm: FO correspondence for M} to the formula $\Delta_A(x_{(i,j),u,v})$. 

    Let $\Delta_A'$ denote the first-order sentence 
    \begin{align*}
        \Delta_A':=&\exists_{(i,j)\in A,\ u,v\in\{1,2\}} \nu_{(i,j),u,v}\,\Bigl(\hat\Delta_A(\boldsymbol\nu)\\
        &\land\bigwedge_{(i,j)\in A}\forall t\, \bigl(\lambda_\Box(t;\nu_{(i,j),u,v}) \iff\rho_6(w_t)=(i,j)\bigr)\Bigr),
    \end{align*}
    and set 
    \begin{align*}
        \hat\Delta_M:=&\Gamma_M\land \bigvee_{A\subseteq\mathcal S}\Delta_A'.
    \end{align*}
    We claim that  $\hat\Delta_M$ satisfies the required properties. 
    
    Suppose that $w \models \Delta_A' \land \Gamma_M$ for some  $A \subseteq S$. Since $w \models \Gamma_M$, $w$ decomposes as $w = u_0u_{(i_1,j_1)} \dots u_{(i_r,j_r)}$. In addition, the second line of $\Delta_A'$ implies that for every $i$, the letter $w_i$ belongs to the sub-word $u_{(i_s,j_s)}$ if and only if $\beta_{w,M}(i) \in g_M(w) \cap C_{i_s,j_s}$. Thus for every $(i,j) \in A \cap \mathcal{S}_\pm$  we have $\beta_{w,M}(\ell) \in g_M(w) \cap C_{i,j}$ if and only if $\rho_6(w_\ell) = (i,j)$. In particular, the letter $w_\ell$ must belong to the sub-word $u_0$. 

    Thus, for every such $w$ we have $g_{M}(w) \cap C_{(i_s,j_s)} = f_d(\rho_5(u_{(i_s,j_s)}))$ and $g_M(w) \setminus \left(\bigcup_{s=1}^r C_{(i_s,j_s)}\right) = g_{M_{\pm 1}}(u_0)$. Hence the result follows from the first-order correspondence provided by Theorem~\ref{thm: FO correspondence for M}.

\end{proof}

\begin{proof}[Proof of Theorem~\ref{thm: fM is a bijection}]
    Define  $\Delta_M := \hat{\Delta}_M \land \Delta_{M,\#}$, where $\hat{\Delta}_M$ is the first-order sentence constructed in Proposition~\ref{prop: penultimate step} and $\Delta_{M,\#}$ is the first-order sentence constructed in Lemma~\ref{lemma: key thm for gM instead}. We claim that $\Delta_M$ satisfies the required properties. 

    Given $\sigma \in \mathrm{Geom}(M)_{d,k}$, let $\sigma_\#$ be the left-greedy representation of $\sigma$, and let $w$ be the unique word chosen so that $w \models \Delta_{M,\#}$ and $g_M(w) =\sigma_{\#}$. Since $\sigma_\#$ is the left-greedy representation of $\sigma$, Proposition~\ref{prop: penultimate step} implies that $w \models \hat{\Delta}_M$. Hence $w \models \Delta_M$, as required.

    If $f_M(w) = f_M(w') = \sigma$ and $w, w' \models \Delta_M$ for some $\sigma \in \mathrm{Geom}(M)_{d,k}$, then the fact that $w,w' \models \hat{\Delta}_M$ implies that  $g_M(w) = g_M(w')$, as both $g_M(w)$ and $g_M(w')$ coincide with the (unique) left-greedy representation of $\sigma$. Since $w,w' \models \Delta_{M,\#}$, Lemma~\ref{lemma: key thm for gM instead} implies that $w = w'$, which completes the proof. 
\end{proof}
\subsection{Constructing a finite deterministic automaton}

For convenience, we follow the convention that $g_M(\varepsilon) = \varepsilon$, where $\varepsilon$ denotes both the empty word and the empty permutation (and thus in particular the empty word is not mapped to $\star$). This somewhat inconsistent choice of notation does not affect the rest of our argument, and makes it slightly easier to construct the corresponding automata.

Recall that we enumerate the elements of $\mathcal{S}_\star$ as $(i_1,j_1),\dots,(i_r,j_r)$. Let $\mathcal{A}_M$ denote the finite deterministic automaton consisting of initial state $e_M$, states 
\[Q_M = \{e_M,\star\} \cup \bigcup_{t = 1}^r \left\{q_{(i_t,j_t)}: q \in Q_d\setminus \{e_d,\star\}\right\},\] 
and accepting states 
\[F_M = \{e_M\} \cup \bigcup_{t=1}^r \left\{q_{(i_t,j_t)}: q \in F_d\right\}.\]
For convenience, we write $B_t$ to denote the restriction of $\mathcal{A}_M$ onto states $q_{(i_t,j_t)}$ where $q$ ranges over all states in $Q_d\setminus \{e_d,\star\}$. 

In addition, the transition rules are given as follows. If $v = a_{i,j} \in \Sigma_{M,\pm}$, we have $\delta_M(e_M,v) = e_M$ and $\delta_M(q,v) = \star$ whenever $q \neq e_M$. If  $v \in \Sigma_M\setminus \Sigma_{M,\pm}$, we write $u = \delta_d(e_{d},\rho_5(v))$, and set $\delta_M(e_M,v) = u_{\rho_6(v)}$, with the usual convention that $\delta_M(e_M,v) = \star$ whenever $u = \star$. 

Now suppose that $q = r_{(i_t,j_t)}$ for some $r \in Q_d$. If $v \in \Sigma_{M,\pm}$, we set $\delta_M(q,v) = \star$. If $\rho_6(v) = (i_t,j_t)$, we write $u = \delta_d(r,\rho_5(v))$, and set $\delta_M(q,v) = \star$ if $u = \star$, and $\delta_M(q,v) = u_{(i_t,j_t)}$ otherwise. If $\rho_6(v) = (i_{t'},j_{t'})$ with $t' \neq t$, then 
\begin{enumerate}
    \item if $t' < t$, we set $\delta_M(q,v) = \star$,
    \item if $r \not \in F_d$, we set $\delta_M(q,v) = \star$, and 
    \item if $t' > t$ and $r \in F_d$, we write $u = \delta_{d}(e_d,\rho_5(v))$ and set $\delta_M(q,v) = u_{(i_{t'},j_{t'})}$. 
\end{enumerate}
In the last case, we use the convention that if $u = \star$ then $\delta_M(q,v) = \star$. Finally, we set $\delta_M(\star,v) = \star$.
We now check that $\mathcal{A}_M$ satisfies the required properties. 
\begin{lem} \label{A_pm lemma}
    $\mathcal{A}_M$ accepts precisely those words $w$ for which $g_M(w) \neq \star$. 
\end{lem}
\begin{proof}
    Write $B_0 = \{e_M\}$. Note that for each $i < j$, there are only directed edges from $B_i$ to $B_j$. Thus, it is straightforward to see that any $w$ which is accepted by $\mathcal{A}_M$ decomposes as $w = u_0u_{(i_1,j_1)}\dots u_{(i_r,j_r)}$, with $u_0$ and $u_{(i_t,j_t)}$'s containing the letters of $w$ with appropriate $\rho_6$-images, allowing any of these sub-words to be empty. Note that any such $u_0$ is accepted by $\mathcal{A}_M$, and for $w$ to be accepted, we need each $u_{(i_t,j_t)}$ to be either empty or its $\rho_5$-image to be accepted by $\mathcal{A}_d$. This coincides with the condition that $g_M(w) \neq \star$. 
\end{proof}

Having constructed the automaton $\mathcal{A}_M$, we remark that the component $B_0 = \{e_M\}$ corresponding to the sub-words $u_0 \in \Sigma_{M,\pm}^*$ fails to capture all the relevant properties, as we do not yet restrict our attention to injectivity. To complete this sub-section, we construct a finite deterministic automaton $\mathcal{A}_\pm$ which accepts precisely those $w \in \Sigma_{M,\pm}^*$ that correspond to lexicographically minimal words. The construction heavily uses Lemma~\ref{lemma: no stars preimages given by equivalence relation} and thus builds on the work of Albert et. al.~\cite{albert2013geometric}. While the construction of $\mathcal{A}_\pm$ does not directly appear in~\cite{albert2013geometric}, we emphasize that it is in a sense implicit from their work, and we simply produce the construction here for completeness.

Define $\mathcal{A}_\pm = \left(Q_\pm, \Sigma_{M,\pm}, \delta_\pm, e_\pm, F_\pm \right)$ where 
\[Q_\pm = \{q_A:A \subseteq \Sigma_{M,\pm}\} \cup \{\star\},\]
the initial state $e_\pm$ is chosen to be the state $q_\emptyset$, 
\[F_\pm = Q_\pm \setminus \{\star\},\]
and the transition rules are given as follows. For each $a =a_{i,j}\in \Sigma_{M,\pm}$  define the sets $U(a_{i,j})$ by setting 
\[U(a_{i,j}) = \{a_{i',j'}\in\Sigma_{M,\pm}:i'\neq i,\ j'\neq j\}.\]
Hence $U(a)$ consists of letters that do not correspond to the same row nor column as $a$. For simplicity, we write $a < b$ if $a = a_{i,j}$ and $b = a_{i',j'}$ where $(i,j) <_{lex} (i',j')$. Given a state $q_A \neq \star$, we set $\delta_\pm(q_A,v) = \star$ whenever $v \in A$, and otherwise we set 
$\delta_\pm(q_A,v) = q_B$, where 
\[B = (A \cap U(v)) \cup \{c \in U(v):c < v\}.\]
In addition, we define $\delta_\pm (\star,v) = \star$ for every $v$. 

We prove that $\mathcal{A}_\pm$ satisfies the required properties. 
\begin{lem}\label{lemma: directed edge creates self-loop}
    $\mathcal{A}_{\pm}$ accepts exactly those words $w$ which correspond to the lexicographically minimal words. In addition, whenever $\delta_{\pm}(q_A,v) = q_B$, then $\delta_\pm(q_B,v) = q_B$. Thus every strongly connected component with at least two states contains a vertex with a self-loop. 
\end{lem}
\begin{proof}
    Suppose for the sake of contradiction that there exists $w = w_1\dots w_\ell$ of minimal length that is lexicographically minimal, and for which $w$ is rejected by $\mathcal{A}_\pm$. Note that every proper prefix of $w$ must also be lexicographically minimal, and hence each proper prefix of $w$ is accepted by $\mathcal{A}_\pm$ by the minimality of $w$. For each $0 \leq i \leq \ell-1$, let $P_i$ denote the set chosen so that $q_{P_i} = \delta_{\pm}(q_\emptyset ,w_1\dots w_{i})$ with the convention that $P_0 = \emptyset$. Since $w$ is rejected, we have $w_\ell \in P_{\ell-1}$. 
    
    Note that for each $1 \leq i \leq \ell-1$, we have $w_\ell \in P_{i}\setminus P_{i-1}$ only when  $w_\ell \in U(w_i)$ and $w_\ell < w_i$. Let $j$ be the largest index for which we have $w_\ell \in U(w_j)$ and $w_\ell < w_j$. Thus $w_\ell$ must be contained in each of the sets $P_j,\dots,P_{\ell-1}$, and thus we must have $w_\ell \in U(w_i)$ for every $j \leq i \leq \ell-1$. However, this contradicts the third condition of Lemma~\ref{lemma: no stars preimages given by equivalence relation}. 

    Conversely, suppose that $w$ is accepted, and define $P_i$ as above for all the prefixes of $w$. If $w$ were not lexicographically minimal, then  by Lemma~\ref{lemma: no stars preimages given by equivalence relation} there exists indices $j<t$ with $w_t < w_j$ and $w_t \in U(w_i)$ for every $j \leq i < t$. Thus by the transition rules of $\mathcal{A}_\pm$ we conclude that $w_t \in P_{t-1}$, which contradicts the fact that $\mathcal{A}_\pm$ accepts $w$. Thus $w$ must be lexicographically minimal.

    Finally, suppose that $A$, $B$ and $v$ are chosen so that $\delta_\pm(q_A,v) = q_B$. Then 
    \[B= (A \cap U(v)) \cup \{c \in U(v): c < v\},\]
    and thus in particular we have $B \subseteq U(v)$. Note that we also have $v \not \in B$: indeed, as $v \not \in A$ we have $v \not \in A \cap U(v)$, and clearly $v$ is not lexicographically strictly smaller than $v$ itself, so $v$ is not contained in the second set either. 

    Let $B' = (B \cap U(v)) \cup \{c \in U(v): c < v\}$. Since $B \subseteq U(v)$ and $\{c \in U(v): c < v\} \subseteq B$, we conclude that $B' = B$. Thus, the rules of the transition function imply that $\delta_\pm(q_B,v) = q_B$. Since every strongly connected component with at least two states contains a directed edge, the existence of a vertex with a self-loop follows immediately. 
\end{proof}
\section{Constructing a finite deterministic automaton}
We recall some useful terminology from Section~\ref{sec: automata theory}. Given an automaton $\mathcal{A}$ with transition matrix $A$, and a strongly connected component $C$, we write $A_C$ for the restriction of $A$ on the states in $C$. Let $A'$ denote the transition matrix of the automaton $\mathcal{A'}$ from which all the non-productive states are removed. We write $\rho(\mathcal{A})$ for the spectral radius of $A'$, and a component $C$ is said to be dominant if it is productive and we have $\rho(A_C) = \rho(\mathcal{A})$.

 Throughout this section we assume that $\mathcal{F}$ is infinite. Thus, in particular every dominant component in an automaton that accepts $L_{\mathcal{F}}$ contains a directed cycle.  Our goal in this section is to prove that for every infinite small permutation class $\mathcal{F}$ there exists a finite deterministic automaton $A_{\mathcal{F}}$ which recognizes precisely the words in $L_{\mathcal{F}}$, and all of whose dominant components are aperiodic. We start by  outlining the argument on how we construct $A_{\mathcal{F}}$. 

Recall that there exists a suitable matrix $M$ so that $\mathcal{F} \subseteq \mathrm{Geom}(M)_{d,k}$, and there exists a first-order sentence $\chi_{\mathcal F}$ on $\Sigma_M^*$ for which we have $w \models \chi_{\mathcal{F}}$ if and only if $f_M(w) \in \mathcal{F}$. Such a sentence exists by Proposition~\ref{prop: small classes FOS classifiable} and
Theorem~\ref{thm: FO correspondence for M}.

Write $\mathcal F^\#$ for all admissible $M$-griddings of members of $\mathcal F$. Note that $\mathcal{F}^\#$ is a down-set for gridded pattern containment. For convenience reasons, we again allow the empty permutation to be counted in $\mathcal{F}$ and $\mathcal{F}^\#$.

Let $\Delta_d$ denote the first-order sentence constructed in Section~\ref{section: restricting to bijection} and let $\mathcal{B}_d$ denote a finite deterministic automaton accepting precisely those words in $\Sigma_{d,k}^*$ for which $w \models \Delta_d$, and which satisfies the conclusion of Lemma~\ref{lemma: iterate}. By replacing each copy of $\mathcal{A}_d$ with $\mathcal{B}_d$, and the initial state with a copy of $\mathcal{A}_\pm$, we can modify the construction of $\mathcal{A}_{M}$ to obtain $\mathcal{A}_{M,\#}$ which accepts precisely those words $w$ for which $w \models \Delta_{M,\#}$. In particular, each gridded permutation has a unique pre-image accepted by $\mathcal{A}_{M,\#}$. For an illustration of $\mathcal{A}_{M,\#}$, see Figure~\ref{fig: M and hash}. 

\begin{figure}
    \centering
\begin{tikzpicture}[>=stealth,font=\small,
  block/.style={draw,rounded corners,align=center,
                minimum height=15mm},
  lab/.style={font=\scriptsize,below,inner sep=3pt}]

  \node[block,minimum width=23mm] (pm) at (1,0)
    {$\mathcal A_\pm$};
  \node[block,minimum width=23mm] (first) at (4.5,0)
    {$\mathcal B_{d,(i_1,j_1)}$};
  \node (dots1) at (6.5,0) {$\dots$};
  \node[block,minimum width=23mm] (middle) at (8.5,0)
    {$\mathcal B_{d,(i_s,j_s)}$};
  \node (dots2) at (10.5,0) {$\dots$};
  \node[block,minimum width=23mm] (last) at (12.5,0)
    {$\mathcal B_{d,(i_r,j_r)}$};

  \draw[->] ([xshift=-3.15mm]pm.south)
  arc[start angle=125,end angle=415,radius=5.5mm]
  node[lab,pos=.5] {$\rho_6(v)\in\Sigma_{M,\pm}$};

  \foreach \name/\x/\idx in
    {first/4.5/1,middle/8.5/s,last/12.5/r} {
    \draw[->] (\name.60)
      .. controls ({\x+.8},1.6) and ({\x-.8},1.6) ..
      node[above,font=\scriptsize,inner sep=3pt]
        {$\rho_6(v)=(i_{\idx},j_{\idx})$}
      (\name.120);
  }

  \draw[->] (pm.south east)
    .. controls (2.5,-1.5) and (3,-1.5) ..
    node[lab] {$\rho_6(v)=(i_1,j_1)$}
    (first.south west);

  \draw[->] (first.east) -- (dots1.west);
  \draw[->] (dots1.east) -- (middle.west);
  \draw[->] (middle.east) -- (dots2.west);
  \draw[->] (dots2.east) -- (last.west);

  \draw[->] (pm.north east)
  .. controls (3,3) and (7.35,3) ..
  node[lab,pos=.48] {$\rho_6(v)=(i_s,j_s)$}
  (middle.north west);

  \draw[->] (pm.north)
    .. controls (1,4.6) and (11.35,4.6) ..
    node[lab] {$\rho_6(v)=(i_r,j_r)$}
    (last.north west);

  \draw[->] (first.south east)
    .. controls (5.8,-1.6) and (7.2,-1.6) ..
    node[lab] {$\rho_6(v)=(i_s,j_s)$}
    (middle.south west);

  \draw[->] (first.south)
    .. controls (4.5,-3.2) and (12.5,-3.2) ..
    node[lab] {$\rho_6(v)=(i_r,j_r)$}
    (last.south);
\end{tikzpicture}
    \caption{Illustration of the automaton $\mathcal{A}_{M,\#}$ accepting words $w$ with $w \models \Delta_{M,\#}$. The directed edges between the blocks are labelled with the corresponding $\rho_6$-images of the associated letters.}
    \label{fig: M and hash}
\end{figure}
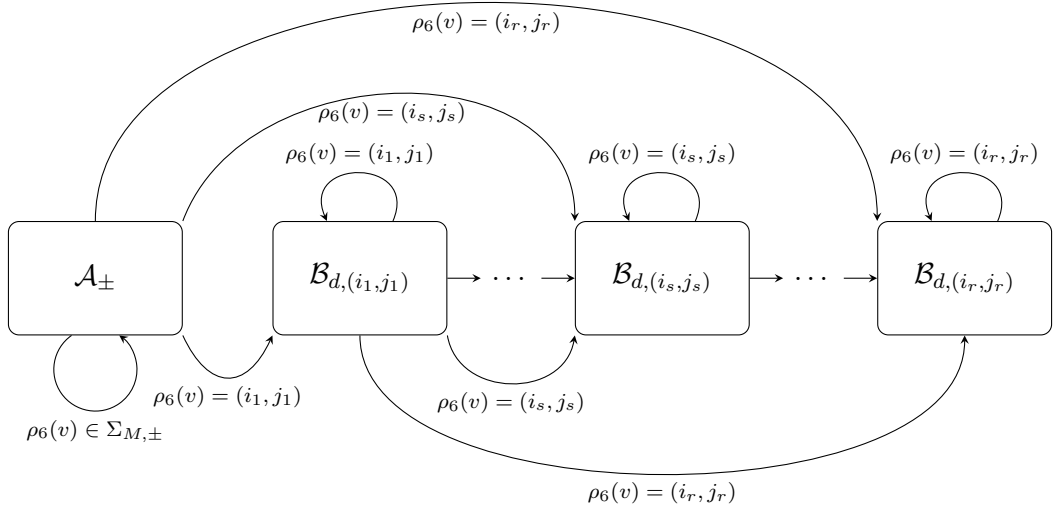

As our next goal we construct an automaton $\mathcal{A}_{\mathcal{F},\#}$ so that for each $\sigma \in \mathcal{F}$ and for any gridded realization $\theta^\# \in \mathrm{Geom}^\#(M)_{d,k}$ with $\pi_M^\#(\theta^\#)=  \sigma$ there exists a unique $w$ that is accepted by $\mathcal{A}_{\mathcal{F},\#}$  and for which $g_M(w) = \theta^\#$. We do so by simply taking the product of $\mathcal{A}_{M,\#}$ together with an automaton obtained by applying Lemma~\ref{lemma: iterate} to the first-order sentence $\chi_\mathcal{F}$. 

The main challenge at this stage is to prove that the dominant strongly connected components in $\mathcal{A}_{\mathcal{F},\#}$ are aperiodic. This is especially tricky as when passing to the permutation subclass, the spectral radius of the transition matrix may decrease. We also highlight that the statement is not true if $\chi_{\mathcal{F}}$ is replaced with a general first-order sentence, and we heavily use the fact that $\mathcal{F}$ is a down-set. 

Finally, we let $\mathcal{D}$ denote the finite deterministic automaton which accepts precisely those words $w$ with $w \models \Delta_M$, and which satisfies the conclusion of Lemma~\ref{lemma: iterate}. Then $\mathcal{A}_{\mathcal{F}} := \mathcal{A}_{\mathcal{F},\#} \times \mathcal{D}$ clearly satisfies the required properties. Again, we are left with proving the fact that the dominant strongly connected components in $\mathcal{A}_{\mathcal{F}}$ are aperiodic. A simple pegs and holes -type of argument due to Vatter~\cite{vatter2011small} implies that the ratio $|\mathcal{F}^\#_n|/|\mathcal{F}_n|$ is bounded above by a function that grows at most polynomially in $n$. Thus, the spectral radius of the dominant strongly connected components must be the same in both $\mathcal A_{\mathcal F,\#}$ and $\mathcal A_{\mathcal F}$. In Subsection~\ref{sec:ungridded-aperiodicity} we prove that under such assumptions we must maintain the aperiodicity in the dominant strongly connected components in the automata $\mathcal{A}_{\mathcal{F}}$ as well. 

The remainder of this section is organized as follows. In Section~\ref{sec:spectral-comparison} we start by proving several preliminary lemmas which allow us to compare the spectral radii of two somewhat structurally related strongly connected components. In Section~\ref{sec:geometric-component-aperiodicity} we construct the automata $\mathcal{A}_\mathcal{F,\#}$, and analyse the dominant components whose projection lies inside $\mathcal{A}_\pm$. In section~\ref{sec:recursive-component-aperiodicity} we in turn analyse the dominant components whose projection lies inside $\mathcal{A}_d$. Finally, we pass from $\mathcal{F}^\#$ to $\mathcal{F}$ in Section~\ref{sec:ungridded-aperiodicity}. 
\subsection{Comparing spectral radii}\label{sec:spectral-comparison}
We need the following standard version of the Perron-Frobenius theorem. Given an $n \times n$ matrix $A$ with non-negative entries, we define the associated directed graph $G$ with $V(G) = [n]$ and for every $i$ and $j$, the edge $ij$ is included as a directed edge if and only if $A_{i,j} \neq 0$. A matrix $A$ is said to be $\textit{irreducible}$ if the associated directed graph $G$ is strongly connected, i.e.\ for every $i$ and $j$ there exists a directed path from $i$ to $j$. In addition, $A$ is said to be \textit{primitive} if there exists a positive integer $m$ so that $A^m$ has only positive entries. 
\begin{thm}[Perron-Frobenius theorem]\label{thm: Perron frobenius irreducible}
    Let $A$ be an irreducible matrix with non-negative entries, and let $\lambda = \rho(A)$ denote its spectral radius. Then the following conditions are satisfied.
    \begin{enumerate}
        \item Both the left and right eigenspace associated with $\lambda$ have dimension $1$. 
        \item There exists both left and right eigenvectors $v$ and $w$ respectively associated with $\lambda$ which have positive entries. 
    \end{enumerate}
    Furthermore, if $A$ is primitive, then any eigenvalue $\mu \neq \lambda$ satisfies $|\mu| < |\lambda|$. 
\end{thm}
For a proof, see e.g.\ \cite[Note V.34 and proof of Theorem V.7]{FS}. We are now ready to prove Lemma~\ref{lemma: loose radius or maintain aperiodicity}. 

Note that if $C$ is a strongly connected component of an automaton either containing a self-loop or with at least two vertices, and $A_C$ denotes its transition matrix, then $A_C$ is irreducible. Furthermore, if $C$ is aperiodic, it is easy to verify that $A_C$ is primitive. In this section, we prove three preliminary Lemmas. The first one of these proves that by slightly modifying a periodic strongly connected component of an automaton, we may create another automaton with a larger spectral radius. By cleverly constructing the second automaton to match a subset of our permutation class, we can use this result to prove that dominant strongly connected components must be aperiodic. 

\begin{lem}\label{lemma: first spectral drops lemma}
    Let $n$ and $m$ be positive integers, and let $A$ and $B$ be $(n+m)\times(n+m)$ matrices with non-negative entries satisfying the following properties.
    \begin{enumerate}
        \item $A$ is irreducible,
        \item $A_{i,j} \geq B_{i,j}$ for all $i,j$, and $A_{i,j} = B_{i,j}$ for every $1\leq i,j \leq n$, and 
        \item The $n \times n$ sub-matrix of $B$ corresponding to the top-left corner is irreducible, and $B_{i,j} = 0$ whenever $\max(i,j) \geq n+1$. 
    \end{enumerate}
    Let $\lambda = \rho(A)$ and $\mu = \rho(B)$. Then $\mu < \lambda$. 
\end{lem}

The second Lemma concerns a more general property of first-order restrictions of automata. It states that a strongly connected component $C$ in a first-order restriction of a strongly connected aperiodic component $C'$ is either aperiodic, or has strictly smaller spectral radius. While we felt that such a result should be known and 'folklore', we  failed to locate such a statement in the existing literature. Thus we chose to include its simple proof. The third result concerns asymptotic growth of a finite deterministic automaton whose dominant strongly connected components are aperiodic. The result is fairly standard, and thus we state it only at the end of this subsection, and give a brief sketch of a proof.

\begin{lem} \label{lemma: loose radius or maintain aperiodicity}
    Let $\mathcal{A}$ and $\mathcal{B}$ be finite deterministic automata over the same alphabet so that $\mathcal{B}$ is complete and satisfies~\eqref{eq: iterate}. Let $D$ be a strongly connected component in the automaton $\mathcal{A} \times \mathcal{B}$ containing at least one edge, and let $C$  be the strongly connected component of $\mathcal{A}$ containing the projection of $D$. 
    
    Write $\lambda_C = \rho(A_C)$ and $\lambda_D= \rho(A_D)$. Then $\lambda_D \leq \lambda_C$. If equality holds, $C$ contains a cycle and is aperiodic, then $D$ must also be aperiodic. 
\end{lem}

We now prove these two auxiliary results. 

\begin{proof}[Proof of Lemma~\ref{lemma: first spectral drops lemma}]
    By Perron-Frobenius (Theorem~\ref{thm: Perron frobenius irreducible}), $\lambda$ is an eigenvalue of $A$, and $A$ has a corresponding left eigenvector $v$ with positive entries. Let $B'$ denote the top-left corner of $B$. Since the last $m$ rows and columns of $B$ are identically $0$, it is easy to see that $\rho(B) = \rho(B')$, and all the eigenvectors of $B$ associated with non-zero eigenvalues are of the form $(w_1,\dots,w_n,0,\dots,0)$, where $(w_1,\dots,w_n)$ is an eigenvector of $B'$ associated with the same eigenvalue. Thus, by Perron-Frobenius there exists a right eigenvector $w$ of $B$ corresponding to $\mu$ which satisfies $w_i > 0$ for every $1 \leq i \leq n$, and $w_i = 0$ for $n+1 \leq i\leq n+m$. 

    Since $A_{i,j} \geq B_{i,j}$ for every $(i,j)$, we conclude that 
    \begin{align}\label{eq: pointwise comp}
        (Aw)_i = \sum_{j=1}^{n+m}A_{i,j}w_j \geq \sum_{j=1}^{n+m} B_{i,j}w_j = (Bw)_i = \mu \cdot w_i.
    \end{align}
    Hence, we conclude that 
    \begin{align} \label{eq: critical inequality}
        \lambda \cdot v^Tw = v^TAw =\sum_{i=1}^{n+m} v_i\cdot (Aw)_i \geq \sum_{i=1}^{n+m} \mu \cdot v_i\cdot w_i = \mu\cdot v^Tw,
    \end{align}
    where the inequality in the middle follows from~\eqref{eq: pointwise comp}. Thus, we conclude that $\lambda \geq \mu$. 

    If $\lambda = \mu$, then inequality~\eqref{eq: critical inequality} must be tight. Since $v_i > 0$ for every $i$, we conclude that the inequality~\eqref{eq: pointwise comp} in turn must be tight for every $i$. Since $w_j > 0$ for every $1 \leq j \leq n$, this in turn implies that we must have $A_{i,j} = B_{i,j}$ for every $1 \leq j \leq n$ and every $1 \leq i \leq n+m$. In particular, this means that $A_{i,j} = 0$ for every  $i \in \{n+1,\dots,n+m\}$ and $j\in \{1,\dots,n\}$, which contradicts the fact that $A$ is irreducible. Thus we must have $\mu < \lambda$. 
\end{proof}

\begin{proof}[Proof of Lemma~\ref{lemma: loose radius or maintain aperiodicity}]

    The proof follows a similar idea to the proof of Lemma~\ref{lemma: first spectral drops lemma}. Let $C$ be the strongly connected component of $\mathcal{A}$ containing the projection of $D$, and put $\lambda = \rho(A_C)$. Since $D$ contains at least one directed edge, it follows that $C$ also contains at least  one directed edge, and hence $A_C$ is irreducible. Thus Perron-Frobenius (Theorem~\ref{thm: Perron frobenius irreducible}) gives a right eigenvector $v$ associated with $\lambda$ whose all entries are positive. 

    Write $A$ for the transition matrix of the automaton $\mathcal{A} \times \mathcal{B}$, and note that $(A_D)_{(p,r),(p',r')} = A_{(p,r),(p',r')}$ for every $(p,r),(p',r') \in D$ Define a vector $v' \in \mathbb{R}^D$ by setting $v'_{(p,r)} = v_{p}$ for each $(p,r) \in D$. Note that for any $(p,r) \in D$ and $p' \in C$ we have 
    \begin{align} \label{eq:projection sum}
        \sum_{r': (p',r') \in D} (A_D)_{(p,r),(p',r')} \leq \sum_{r' \in \mathcal{B}} A_{(p,r),(p',r')} = (A_C)_{(p,p')},
    \end{align}
    as there are exactly $(A_C)_{p,p'}$ possible transitions from the state $p$ to state $p'$ in $\mathcal{A}$, and thus fixing the initial state $r \in \mathcal{B}$ while allowing the final $r'$ to be arbitrary does not restrict the choice.

    Thus for every $(p,r) \in D$ we have 
    \begin{align} \label{eq: coordwise comparison}
        (A_Dv')_{(p,r)} = \sum_{(p',r') \in D }  (A_D)_{(p,r),(p',r')} \cdot v'_{(p',r')} \leq \sum_{p' \in C} \left(v_{p'} \cdot \sum_{r' \in \mathcal{B}} A_{(p,r),(p',r')} \right)
    \end{align}
    \begin{align*}
        = \sum_{p' \in C} v_{p'} \cdot (A_C)_{p,p'} = (A_C v)_p = \lambda \cdot v_p = \lambda \cdot v'_{(p,r)},
    \end{align*}
    where the inequality between the first and second line follows from~\eqref{eq:projection sum}, and throughout the inequalities we are using the fact that all the entries of the respective matrices are non-negative, and all the entries of the vector $v$ are positive. 

    Let $w$ be the left eigenvector of $A_D$ corresponding to the  eigenvalue $\mu = \rho(A_D)$, i.e.\ $w^T A_D = \mu w^T$, and recall that all the entries of $w$ can be chosen to be positive. Then 
    \begin{align}\label{eq: between evalues}
        \mu \cdot w^T v' = w^T A_D v' = \sum_{(p,r) \in D} w_{(p,r)} \cdot (A_Dv')_{(p,r)} \leq \sum_{(p,r) \in D} w_{(p,r)} \cdot (\lambda \cdot v'_{(p,r)}) = \lambda \cdot w^T v',
    \end{align}
    where the inequality follows from inequality~\eqref{eq: coordwise comparison} and the fact that all the coordinates of $w$ and $A_Dv'$ are non-negative. In particular, we conclude that $\mu \leq \lambda$. 

    Now suppose that $\mu = \lambda$. Since each coordinate of both $v'$ and $w$ is strictly positive, it follows that inequality~\eqref{eq: coordwise comparison} is in fact equality for every $(p,r) \in D$, and hence in turn we conclude that the inequality~\eqref{eq:projection sum} must be equality for every $(p,r) \in D$ and for every $p' \in C$. In particular, every labelled directed edge in $C$ must lift to a labelled directed edge in $D$, as omitting such an edge would make the corresponding sum strictly decrease. Since $C$ is a strongly connected component, we conclude that the projection of $D$ must equal $C$.  
    
    We now prove that $D$ must be aperiodic whenever $C$ is. Suppose that $p_1,\dots,p_l$ are vertices of a closed walk of length $l$ in $C$. We claim that this lifts to a closed walk of length $l$ in $D$. Let $w = v_1\dots v_l$  denote a word which witnesses the closed walk $p_1,\dots,p_l$ with respect to $\mathcal{A}$, i.e.\ chosen so that $p_i = \delta_{\mathcal{A}}(p_1,v_1\dots v_{i-1})$ for every $2 \leq i \leq l+1$ (with $p_{l+1} = p_1$). Since $\mathcal{B}$ satisfies the condition~\eqref{eq: iterate}, it follows that there exists $N$ so that 
    $\delta_{\mathcal B}(r,w^N) = \delta_{\mathcal B}(r,w^{N+1})$ for every $r$.

    Choose $r_1'$ arbitrarily so that $(p_1,r_1') \in D$, and let $r_1,\dots,r_l$ be chosen so that $r_1 = \delta_{\mathcal{B}}(r_1',w^N)$ and $r_i = \delta_{\mathcal{B}}(r_1',w^Nv_1\dots v_{i-1})$ for every $2 \leq i \leq l$. Since
    \[
    \delta_{\mathcal{A}\times \mathcal{B}} \left((p_i,r_i), v_i\right) = (p_{i+1},r_{i+1})
    \]
    holds trivially for $1 \leq i \leq l-1$ by construction, and 
    \[\delta_{\mathcal{B}}(r_{l},v_l) = \delta_{\mathcal{B}}(r_1,w) = \delta_{\mathcal{B}}(r_1',w^{N+1}) = \delta_{\mathcal{B}}(r_1',w^N) = r_1,\]
    it follows that $\delta_{\mathcal{A}\times \mathcal{B}} \left((p_l,r_l), v_l\right) = (p_1,r_1)$. Furthermore, as all the relevant inequalities~\eqref{eq: coordwise comparison} and~\eqref{eq:projection sum} must hold with equality, we conclude that all of these points are contained in $D$. Hence $D$ contains a closed walk of length $l$. Since the same conclusion holds for all the closed walks in $C$, we conclude that $D$ must be aperiodic as well. 
\end{proof}

\begin{lem}\label{lemma: primitive dominant asymptotics}
    Let $\mathcal A$ be a finite deterministic automaton accepting   infinitely many words, and let $\mathcal{A'}$ denote the automaton obtained by deleting those states that are not on any path from the initial state to some accepting state. 

    Let $\lambda$ be the largest spectral radius of a strongly connected component in $\mathcal{A'}$, and suppose that every component with spectral radius $\lambda$ is primitive. Let $k +1$ denote the maximum number of such components along any path from the initial state to an accepting state, where $k$ is a non-negative integer. Then there exists $c > 0$ so that 
    \[
        a_n=(c+o(1))\lambda^n n^k,
    \]
    where $a_n$ denotes the number of words accepted by $\mathcal{A}$ of length $n$. 
\end{lem}
We briefly sketch the argument. Write the transition matrix of $\mathcal{A'}$ in the block triangular form, with blocks corresponding to the strongly connected components. Then Theorem V.7 in~\cite{FS} implies that $1/\lambda$ is the unique dominant singularity in each component with spectral radius $\lambda$. Thus, if the largest number of such strongly connected components along any path from the initial state to an accepting state is $k+1$, then it follows that $a_n=(c+o(1))\lambda^n n^k$.
\subsection{$C$ is contained in a copy of $\mathcal{A}_\pm$} \label{sec:geometric-component-aperiodicity}

Recall that $\mathcal{B}_{\mathcal{F}}$ is chosen so that it satisfies the conclusion of Lemma~\ref{lemma: iterate} for the first-order sentence $\Delta_{M,\#} \land \chi_\mathcal{F}$. The goal of this sub-section is to prove the following result. 
\begin{lem} \label{lemma: aperiodicity or drops for plusminus}
    Suppose that $C$ is a dominant strongly connected component in $\mathcal{A}_{\mathcal{F},\#}$ which lies in the part corresponding to $\mathcal{A}_\pm \times \mathcal{B}_\mathcal{F}$. Then $C$ is aperiodic. 
\end{lem}
We start with a preliminary observation that is useful when dealing with both $\mathcal{A}_\pm$ and $\mathcal{A}_d$. 
\begin{lem} \label{lemma: too long directed paths with same label}
    Let $N \geq 2$ be an integer satisfying the conclusion of Lemma~\ref{lemma: iterate} for $\mathcal{B}_\mathcal{F}$. Suppose that $C$ is a strongly connected component of $\mathcal{A}_{\mathcal{F},\#}$ which contains a directed walk of length at least $N$ whose edges all have the same fixed letter $v$ among their labels. Then $C$ contains a self-loop, and in particular, is aperiodic. 
\end{lem}
\begin{proof}
     Suppose for the sake of contradiction that there exists a directed walk $(a_0,b_0),\dots,(a_N,b_N)$ so that $\delta_{\mathcal A_{\mathcal F,\#}}((a_i,b_i),v) = (a_{i+1},b_{i+1})$ for every $0 \leq i \leq N-1$. Writing $(a,b) = (a_N,b_N)$, we claim that $\delta_{\mathcal A_{\mathcal F,\#}}((a,b),v) = (a,b)$. 

    Observe that the choice of $N$ implies that 
    \[\delta_{\mathcal{B}_{\mathcal{F}}}(b,v) = \delta_{\mathcal{B}_{\mathcal{F}}}(b_0,v^{N+1}) = \delta_{\mathcal{B}_{\mathcal{F}}}(b_0,v^{N}) = b. \]
    First suppose that the projection of $C$ is contained in $\mathcal{A}_\pm$. Note that there is an incoming directed edge labelled with $v$ to $(a,b)$, and thus the same must be true for the state $a$ in $\mathcal{A}_\pm$. Thus, by Lemma~\ref{lemma: directed edge creates self-loop} we must have $\delta_{\mathcal{A}_\pm}(a,v) =a$. 
    
    Now suppose that the projection of $C$ is contained in $\mathcal{A}_{d,i',j'}$ for some $(i',j') \in \mathcal{S}_\star$. We claim that for every state $x$ in $\mathcal{A}_d$ and for every letter $v$ we have 
    \[
    \delta_{\mathcal{A}_d}(x,v^2)=\delta_{\mathcal{A}_d}(x,v^3).
    \]
    Indeed, the statement is trivial when $d = 1$. For $d \geq 2$, we note that if $\rho_2(v) = \tau \in \mathcal{Q}_k$, the statement either follows immediately by induction, or both sides equal $\star$. For $\rho_2(v) \in \{\mathrm{inc},\mathrm{dec}\}$, we note that if $\rho_4(v) = \{L,R\}$, then both sides equal the state $q_{\rho_2(v), \{L,R\}}$ or $\star$. If $\rho_4(v) \in \{\{L\},\{R\}\}$, then both sides are clearly $\star$. Finally, if $\rho_4(v) = \emptyset$, the result follows by induction. Since $N \geq 2$, this observation implies that $\delta_{\mathcal{A}_d}(a,v) = a$.  In particular, in either case we must have $\delta_{\mathcal A_{\mathcal F,\#}}((a,b),v) = (a,b)$. Thus $C$ contains a self-loop based at the vertex $(a,b)$, and hence is aperiodic.
    \end{proof}

\begin{proof}[Proof of Lemma~\ref{lemma: aperiodicity or drops for plusminus}] 
    Choose $N \geq 2$ as in Lemma~\ref{lemma: too long directed paths with same label}. The statement follows immediately from this lemma if $C$ contains a directed walk of length $N$ whose edges are labelled with the same fixed letter $v$. Thus assume that such a directed walk does not exist for any $v$. Since $C$ is dominant, it contains a directed cycle and thus at least one edge. Fix a letter $v$ for which there exists a directed edge in $C$ labelled with $v$, and consider the end-point of the longest directed walk in $C$ whose edges are labelled with $v$. Such a walk must have length at most $N$ by Lemma~\ref{lemma: too long directed paths with same label}, and thus its endpoint $q := (a,b)$ has an incoming directed edge labelled with $v$, but no outgoing directed edge labelled with $v$ inside $C$. 

    Consider automaton $C'$ whose states are the states in $C$, together with additional states $q_1',\dots,q_{N}'$, and with the following transition rules. 
    \begin{enumerate}
        \item The transitions between states inside $C$ are inherited by $C'$. 
        \item Writing $q_0' = q$, we have $\delta_{C'}(q_i',v) = q_{i+1}'$ for every $0 \leq i \leq N-1$. 
        \item For every letter $v$ and state $p$ with $\delta_C(q,v) = p$, we set $\delta_{C'}(q_N',v) = p$.
    \end{enumerate}
    Since $q$ has no outgoing edge labelled with $v$ inside $C$, the automaton $C'$ is deterministic. It is also straightforward to check that $C'$ is strongly connected. 

    Let $\mathcal{C}$ denote the words corresponding to the labels of the closed walks that are contained in $C$, based at $q$, and which do not revisit $q$ in the middle. Note that the same directed cycle can contribute multiple words if its edges are attached with several labels. We define $\mathcal{C}'$ analogously. 

    Let $G$ and $G'$ denote the sets of all the words $u_1\dots u_s$ with $u_i \in \mathcal{C}$ or $u_i \in \mathcal{C'}$ respectively, including the empty word in both cases. Note that
    \[
    |G_n| = (A_C^n)_{q,q} \qquad\text{and}\qquad |G_n'| = (A_{C'}^n)_{q,q}.
    \]
    Write $g$ for the function $g_{M_\pm}:\Sigma_{M,\pm}^*\rightarrow \mathrm{Geom}^\#(M_\pm)$, and let  $\mathcal{G}^\# = \{g(w): w \in G'\}$. 

    \begin{claim} \label{claim: important injectivity claim}
        $g$ restricted to $G'$ induces an injection, and $\mathcal{G}^\# \subseteq \mathcal{F}^\#$. 
    \end{claim}
    \begin{proof}
        Given $w' \in G'$, let $p_0,\dots,p_l$ denote the corresponding closed walk based at $q$ in $C'$. Let $a_0,\dots,a_l$ be chosen so that $p_i = (a_i,b_i)$ whenever $p_i \in C$, and $a_i = a$ whenever $p_i = q_t'$ for some $1 \leq t \leq N$. Thus $a_0 = a_l = a$, and we claim that 
        \begin{align} \label{eq: directed path}
            \delta_\pm(a_{i-1},w'_i) = a_i 
        \end{align}
        holds for every $1 \leq i \leq l$. Indeed, if both $p_{i-1}$ and $p_i$ are contained in $C$ then~\eqref{eq: directed path} follows from the fact that $w_i'$ is a label on the directed edge $p_{i-1} \rightarrow p_i$. 

        If $p_i = q_t'$ for some $1 \leq t \leq N$, then $a_{i-1} = a_i=a$ and $w_i'=v$, so the statement follows easily. Indeed, there is an incoming edge inside $C$ to $q$ labelled with $v$, which implies that  there is an incoming edge to the state $a$ labelled with $v$ inside $\mathcal{A}_\pm$. Thus Lemma~\ref{lemma: directed edge creates self-loop} implies that the state $a$ has a self-loop labelled with $v$ in $\mathcal{A}_\pm$, which implies~\eqref{eq: directed path}. Finally, in the remaining case we must have $p_{i-1} = q_N'$, and the statement follows from the fact that the outgoing edges and their labels from $q_N'$ coincide with those of the state $q$. Thus~\eqref{eq: directed path} holds for all $1 \leq i \leq l$. In particular, $w'$ matches the labels of a closed walk in $\mathcal{A}_\pm$, and hence $w'$ is a contiguous sub-word of a lexicographically minimal word. Thus $w'$ itself must be lexicographically minimal, and hence the restriction of $g$ on $G'$ induces an injection. 

        Recall from Lemma~\ref{thm: Albert et al regular language} that the function $g$ is order-preserving, i.e.\ whenever $u$ is a sub-word of $w$, then $g(u)$ is a sub-pattern of $g(w)$ as a gridded permutation. Thus, it suffices to prove that for every $u \in G'$ there exists $w$ that is accepted by $\mathcal{A}_{\mathcal{F},\#}$ so that $u$ is a sub-word of $w$. 

        Let $u_0$ be a word labelling a closed walk in $C$ based at $q$ and containing the letter $v$. Such a word exists since $C$ is strongly connected and contains an edge labelled with $v$. Given $w'\in G'$, we follow its unique closed walk in $C'$. Every visit to one of the new states $q_i'$ corresponds to visiting the states $q,q_1',\dots,q_N'$ in this order. Furthermore, the final transition from $q_N'$ onward cannot be labelled with the letter $v$ as $q$ has no outgoing directed edges in $C$ with label $v$. In particular, such visit corresponds to a contiguous sub-word of the form $v^N$, and such letters must be the last $N$ instances of the letter $v$ in a contiguous sub-word of the form $v^i$ in $w'$ with $i \geq N$ (and hence we must always have $i \leq 2N-1$). 

        Given $w' \in G'$, replace each contiguous block of $v^N$ which is at the end of the word $w'$ or followed by a letter $v' \neq v$ with the word $u_0^N$, and write $w''$ for the word obtained as the outcome of this process. By the construction, $w''$ labels the edges of a closed walk based at $q$ in $C$, so there exists $u_1$ and $u_2$ so that $u_1w''u_2$ is accepted by $\mathcal{A}_{\mathcal{F},\#}$. By choosing $u_1$ and $u_2$ to be minimal words with such property, we must have $u_1w''u_2 \in \Sigma_{M,\pm}^*$, and thus $g(u_1w''u_2) \in \mathcal{F}^\#$. Since $u_0$ contains the letter $v$, $w'$ is a sub-word of both $w''$ and $u_1w''u_2$, and hence Lemma~\ref{thm: Albert et al regular language} implies that $g(w') \in \mathcal{F}^\#$. This completes the proof of Claim~\ref{claim: important injectivity claim}. 
    \end{proof}
    Given Claim~\ref{claim: important injectivity claim}, the result follows easily. Write $\mu = \rho(A_C)$ and $\mu' = \rho(A_{C'})$. Order the states of $C'$ so that those in $C$ occur at the first $|C|$ rows and columns in $A_{C'}$, and let $B$ denote the matrix obtained by adding $N$ zero rows and columns at the end of $A_C$. Then $A:=A_{C'}$ and $B$ satisfy the conditions of Lemma~\ref{lemma: first spectral drops lemma}, and thus $\mu<\mu'$. Hence Claim~\ref{claim: important injectivity claim} implies that 
    \[
    \mu < \mu' = \limsup_{n \rightarrow \infty} \left((A_{C'}^n)_{q,q}\right)^{1/n} \leq \limsup_{n \rightarrow \infty} |\mathcal{F}_n^\#|^{1/n} = \rho(\mathcal{A}_{\mathcal{F},\#}).
    \]
    In particular, $C$ cannot be dominant, which gives the required contradiction.
    
\end{proof}
\subsection{$C$ is contained in a copy of $\mathcal{A}_d$} \label{sec:recursive-component-aperiodicity}
Let $C$ be a dominant strongly connected component which is contained in $\mathcal{A}_{d,i',j'} \times \mathcal{B}_\mathcal{F}$ for some $(i',j') \in \mathcal{S}_\star$. For convenience, we drop $i'$ and $j'$ from our notation, and simply refer to this copy as $\mathcal{A}_d$. Let $D$ denote the projection of $C$ onto $\mathcal{A}_d$. Then by Lemma~\ref{lemma: iterative sequence lemma} there exists an iterative sequence $\underline{\tau}$ of depth $j$ so that $D \subseteq H_{\underline{\tau}}$ and $D \not\subseteq H_{\underline{\tau},\emptyset}$. We start with the following observation.

\begin{lem} \label{lemma: only LR implies self loop}
    If $C$ is periodic and contains at least one directed cycle, then $j\geq 2$, and $D$ contains at least one state in each of $H_{\underline{\tau},\{L\}}$ and $H_{\underline{\tau},\{R\}}$.
\end{lem}
\begin{proof}
    If $j=1$, then $D$ consists of a single state, which must contain a self-loop with a unique label on it. Similarly, if $j\geq 2$ and $D$ meets neither $H_{\underline{\tau},\{L\}}$ nor $H_{\underline{\tau},\{R\}}$, then as $D\not\subseteq H_{\underline{\tau},\emptyset}$, $D$ must consist of the single state in $H_{\underline{\tau},\{L,R\}}$, and thus it must also contain the self-loop with unique label on it. Indeed, the first observation follows from the fact that there are no directed edges between $H_{\underline{\tau},\emptyset}$ and $H_{\underline{\tau},\{L,R\}}$, and the uniqueness of the label on the self-loop follows from Lemma~\ref{lemma: unique single letter nonstar}. In both cases, $C$ contains arbitrarily long walks whose edges are all labelled with the same letter, and thus Lemma~\ref{lemma: too long directed paths with same label} implies that $C$ is aperiodic. This gives the required contradiction. 

    Thus $D$ must contain a state from at least one of $H_{\underline{\tau},\{L\}}$ or $H_{\underline{\tau},\{R\}}$. Since any closed walk that visits one of $H_{\underline{\tau},\{L\}}$ and $H_{\underline{\tau},\{R\}}$ must also visit the other, we conclude that $D$ contains a state from both $H_{\underline{\tau},\{L\}}$ and $H_{\underline{\tau},\{R\}}$. This completes the proof. 
\end{proof}

For the remainder of this section, we suppose that $C$ is periodic, and we choose a particular state  $q\in C$ whose projection belongs to $H_{\underline{\tau},\{R\}}$. Let $\mathcal{C}$ denote the set of closed walks based at $q$ in $C$. We say that a word $w$ is attached to a walk $c \in \mathcal{C}$ if $w_i$ labels the $i$-th edge of $c$, and we write $L_{\mathcal{C}}$ for the collection of such words. 

Define the function $\pi := \rho_1^{d-j}\circ \rho_5$. Note that for any $w \in L_\mathcal{C}$ we may decompose $w$ as $w = u_1\dots u_s$ so that for each $1 \leq i \leq s$
\begin{enumerate}
    \item The first letter of $\pi(u_i)$ is the unique letter whose $\rho_4$-image contains $L$, and 
    \item The last letter of $\pi(u_i)$ is the unique letter whose $\rho_4$-image contains $R$. 
\end{enumerate}
Since $C$ is a dominant component, for each $w \in L_\mathcal{C}$ there exists $u_1$ and $u_2$ so that $\mathcal{A}_{\mathcal{F},\#}$ accepts $u_1wu_2$, and thus by following the transition rules iteratively, we conclude that $\mathcal{A}_{j-1}$ accepts each $\rho_1(\pi(u_i))$, and there exists $\alpha:=\tau_j\in \{\mathrm{inc},\mathrm{dec}\}$ so that each $\rho_2$-image of the letters in $\pi(u_1),\dots,\pi(u_s)$ equals $\alpha$. Define the function $h: L_\mathcal{C} \rightarrow \widetilde{\mathcal O}_k^{[j]}$ by setting 
\begin{align} \label{eq: important direct sum}
    h(w) = \tau_{\alpha,s}[f_{j-1}(\rho_1(\pi(u_1))),\dots,f_{j-1}(\rho_1(\pi(u_s)))],
\end{align}
where $\tau_{\alpha,s}$ denotes the monotone permutation whose direction is given by  $\alpha\in \{\mathrm{inc},\mathrm{dec}\}$ and whose length is $s$. Analogously to Claim~\ref{claim: important injectivity claim} in previous subsection, we prove that the function $h$ satisfies certain important properties which allows us to prove that a dominant $C$ in fact cannot be periodic. 
\begin{lem} \label{lemma: important properties of h}
    \begin{enumerate}
        \item $|h(w)| = |w|$ for every $w \in L_\mathcal{C}$.
        \item $h$ is an injection.
        \item The image of $h$ is contained in $\mathcal F$.
        
    \end{enumerate}
\end{lem}
\begin{proof}
    Note that the first property follows immediately from the fact that $f_{j-1}$ is a length-preserving function. Since $C$ lies on an accepting path, there exist fixed words $w_1$ and $w_2$ so that $w_1ww_2$ is accepted by $\mathcal A_{\mathcal F,\#}$ for every $w\in L_{\mathcal C}$, and we may choose minimal such $w_1$ and $w_2$. If the common $\rho_6$-image of the letters in $w$ is $(i',j')$, then minimality implies that the same must be true for all the letters in $w_1$ and $w_2$ as well. For each $w$, define $u(w) := \rho_5(w_1ww_2) \in \Sigma_{d}^*$ and note that $u(w) \models \Delta_d$. Set $h_0(w) = f_d(u(w))$, and we note that $h_0$ is an injection as $u(w) \models \Delta_d$ for every $w$, and its image is contained in $\mathcal{F}$.

    While the proof of the remaining two statements is rather technical and tedious, in a vague sense our goal is to show that for every $w \in L_\mathcal{C}$, the expression~\eqref{eq: important direct sum} is the left-greedy $\widetilde{\mathcal O}_k^{[j-1]}$-decomposition of $h(w)$. We prove this by showing that at depth $s > j$, the index of the block of $u(w)$ containing the letters of $w$ is independent of the choice of $w$, and at depth $s = j$, the blocks correspond to the aforementioned left-greedy decomposition.

    When $d > j$, we start by proving that the letters of $w$ in $u(w)$ must lie in the same iterated block whose index is independent of $w$. Indeed, if $u(w)$ decomposes properly as $u_1(w) \dots u_{s(w)}(w)$, then the profile is uniquely determined by $\tau_d$. If $\tau_d \in \mathcal{Q}_k$, then each letter of $w$ must be contained in the sub-word $u_{i_d}(w)$ for every $w\in L_{\mathcal{C}}$. If $\tau_d = \mathrm{dec}^*$, then each letter of $w$ must be contained in the sub-word $u_1(w)$. Finally, if $\tau_d \in \{\mathrm{inc},\mathrm{dec}\}$ and $j < d$, let $n(w_1)$ denote the number of letters in $\rho_5(w_1)$ whose $\rho_4$ image contains $L$. Since $\rho_4(\rho_5(w_i))  = \emptyset$ for each $w \in L_\mathcal{C}$ and for each $i$, we conclude that each letter of $w$ must be contained in the sub-word $u_{n(w_1)}(w)$. 
    
    When $d > j$, let $s$ be chosen so that the letters of $w$ are always contained in $u_s(w)$. As the words $u_i(w)$ for $i \neq s$ are unaffected by the choice of $w$, the injectivity of $h_0$ implies that $h_1(w) := f_{d-1}(\rho_1(u_s(w)))$ induces an injection. 

    Iterate this process until we reach depth $j$. Along the process, we have constructed injections $h_0,\dots,h_{d-j}$ as above. Write $u_j(w)$ for the appropriate projected sub-word of $u(w)$ over $\Sigma_j$. Then $u_j(w)$ decomposes properly as $u_j(w) = u_1'(w)\dots u_s'(w)$, and 
    \[h_{d-j}(w) = \tau_{\tau_j,s}[f_{j-1}(\rho_1(u_1'(w))),\dots,f_{j-1}(\rho_1(u_s'(w)))].\] 
    For convenience, we suppose that the profile $\tau_j$ in the iterative block $\underline{\tau}$ is monotone increasing. 

    Recall that $w$ decomposes as $u_1''\dots u_{t(w)}''$ so that for each $i$, \begin{enumerate}
        \item The first letter of $\pi(u_i'')$ is the unique letter whose $\rho_4$-image contains $L$, and 
        \item The last letter of $\pi(u_i'')$ is the unique letter whose $\rho_4$-image contains $R$. 
    \end{enumerate}
    In particular, we conclude that there exists $t_1$ independent of $w$ for which $\pi(u_i'') = u_{i+t_1}'(w)$ for every $1 \leq i \leq t(w)$. In particular, 
    \[h_{d-j}(w) = \sigma_1(w) \oplus \dots \oplus \sigma_{s}(w),\]
    where $\sigma_i(w) = f_{j-1}(\rho_1(u_i'(w)))$, and by the inductive construction of the functions $h_{d-j}$, the right-hand side also coincides with the left-greedy decomposition of the underlying permutation at depth $j$. 

    By construction, we have 
    \[h(w) = \sigma_{t_1 + 1}(w) \oplus \dots \oplus \sigma_{t_1 + t(w)}(w),\]
    and note that for $w, w' \in L_\mathcal{C}$ we have $\sigma_i(w) = \sigma_i(w')$ for $i \leq t_1$, and $\sigma_i(w) = \sigma_{i'}(w')$ for $i \geq t_1 + t(w) + 1$ and $i' = i - t(w) + t(w')$. Indeed, this follows by noticing that the blocks $u_i'(w)$ outside this range correspond to the letters of the fixed words $w_1$ and $w_2$, whereas the blocks inside this range correspond to the letters of $\pi(w)$. In particular, we conclude that there exist fixed permutations $\sigma_1'$ and $\sigma_2'$ so that for every $w \in L_{\mathcal{C}}$, we have 
    \begin{align}\label{eq: equation giving injectivity}
        h_{d-j}(w) = \sigma_1' \oplus h(w) \oplus \sigma_2'.    
    \end{align}
    
    Since each $h_i(w)$ is a sub-pattern of $h_{i-1}(w)$, and $h_0(w) = f_d(\rho_5(w_1ww_2))$, we conclude that $h(w)$ is a sub-pattern of $h_{d-j}(w)$, and hence is a sub-pattern of  a permutation in $\mathcal{F}$. In particular, the image of $h$ is contained in $\mathcal{F}$. Since $h_{d-j}$ is an injection, and the permutations $\sigma_1'$ and $\sigma_2'$ are independent of $w$, the equation~\eqref{eq: equation giving injectivity} implies that $h$ is also an injection. Finally, the decomposition~\eqref{eq: important direct sum} is left-greedy by Lemma~\ref{lemma: albert left greedy}, as any better decomposition of $h(w)$ as a direct sum would also improve the left-greedy decomposition of $h_{d-j}(w)$. 
\end{proof}
Our next goal is to prove that long monotone sub-intervals in the image of $h$ correspond to long contiguous sub-words consisting of a single letter. Combining this with Lemma~\ref{lemma: too long directed paths with same label} we conclude that there cannot be long monotone sub-intervals in the image of $h$. 

We start by proving an analogous statement for the functions $f_j$. For clarity, we assume that $k \geq 2$ in the following lemma. Note that this does not pose a restriction, as increasing the value of $k$ simply increases the underlying permutation class $\widetilde{\mathcal{O}}_{k}$
\begin{lem}\label{lemma: no long monotone subintervals for fs}
    Let $w \in \Sigma_d^*$ be chosen so that $w \models \Delta_d$. If $f_d(w)$ contains a monotone interval of length $s > 2^{d-1}(d+k)$, then $w$ contains a contiguous sub-word of length at least $\frac{s}{2^{d-1}}-2d$ which is constant. 
\end{lem}
\begin{proof}
    We prove the statement by induction on $d$. When $d = 1$, the fact that $|w| > k$ and $w \models \Delta_d$ implies that $w = a_\mathrm{inc}^n$ or $w = a_\mathrm{dec}^n$. Thus the claim follows immediately in both cases. 

    Now suppose that the statement holds for $d-1$, and suppose that $w$ decomposes properly as $w = u_1\dots u_\ell$. For each $1 \leq i \leq \ell$ write $\theta_i = f_{d-1}(\rho_1(u_i))$. Let $I$ be a monotone sub-interval of length $s$ in $f_d(w)$, and let $J \subseteq [\ell]$ be chosen so that $I$ has non-empty intersection with the block of $f_d(w)$ corresponding to $f_{d-1}(\rho_1(u_i))$ whenever $i \in J$. Note that $J$ is clearly a sub-interval $\{a,\dots,b\}$. For simplicity, we only consider the case when $I$ is monotone increasing. 

    If $b \geq a+2$, the block corresponding to $f_{d-1}(\rho_1(u_{a+1}))$ is fully contained in $I$, and hence $f_{d-1}(\rho_1(u_{a+1}))$ itself must be a monotone increasing permutation. Let $\tau$ be either the profile of $w$ or the increasing\footnote{Note that the profile cannot be $\mathrm{dec}$ as $I$ is monotone increasing and $b > a$.} permutation of length $\ell$. Since $w \models \Delta_d$, the left-greedy decomposition of $f_d(w)$ is given as 
    \[f_d(w) = \tau[\theta_1,\dots,\theta_\ell].\]
    Since $b \geq a+2$ and $\theta_{a+1}$ is monotone increasing, we may move the points from $I\cap \theta_{a+2}$ into $\theta_{a+1}$, and thus obtain a decomposition which is smaller with respect to the left-greedy order. This gives the desired contradiction, and implies that we must have $b \leq a+1$. 

    Now suppose that $b \in \{a,a+1\}$. Thus by the pigeonhole principle we have $|I\cap \theta_a|\geq \frac{s}{2}$ or $|I \cap \theta_{a+1}| \geq \frac{s}{2}$. Applying the inductive hypothesis with $w' = \rho_1(u_t)$  and to the interval $I' = I \cap \theta_t$, where $t$ is chosen so that $|I'| \geq \frac{s}{2} > 2^{d-2}((d-1)+k)$, it follows that $w'$ contains a contiguous constant sub-sequence of length $N':=\frac{s/2}{2^{d-2}} - 2(d-1)$. Since this constant contiguous subsequence is contained in one particular sub-word $u_t$, it can be lifted to a constant contiguous subsequence in $w$ of length at least $N'-2$; indeed, the $\rho_2$ and $\rho_3$-projections (if appropriate) of all the letters in this contiguous subsequence must be equal, but we may have to remove the first and last letter to ensure that their $\rho_4$-projections are equal (if the type is $\mathrm{inc}$ or $\mathrm{dec}$). This completes the proof.
\end{proof}
We now obtain the following corollary.
\begin{cor}\label{cor: no long monotone subintervals for hs}
    If $h(w)$ contains a monotone interval of length $s>2^{j-1}(j+k)$, then $w$ contains a contiguous sub-word of length at least $\frac{s}{2^{j-1}}- 2j$ which uses only a single letter $v$. 
\end{cor}
\begin{proof}
    Recall that every $w \in L_{\mathcal{C}}$ decomposes as $w = u_1 \dots u_m$ so that 
    \begin{align}\label{eq: left greedy for h}
        h(w) = \sigma_1\oplus \dots \oplus \sigma_m,
    \end{align} 
    where $\sigma_i = f_{j-1}(\rho_1(\pi(u_i)))$. In addition, recall from the proof of Lemma~\ref{lemma: important properties of h} that the right-hand side of the equation~\eqref{eq: left greedy for h} also coincides with the left-greedy decomposition of $h(w)$ when viewed as an element of $\widetilde{\mathcal{O}}_{k}^{[j]}$. Thus $\pi(w) \models \Delta_{j}$, and hence the result follows from Lemma~\ref{lemma: no long monotone subintervals for fs}. 
\end{proof}

We are now ready to prove the main result of this subsection.
\begin{lem}\label{lemma: periodic drops in A_d}
    Suppose that $C$ is a dominant strongly connected component in $\mathcal A_{\mathcal F,\#}$ which lies in the part corresponding to $\mathcal A_{d,i',j'}\times\mathcal B_{\mathcal F}$. Then $C$ is aperiodic.
\end{lem}
\begin{proof}
    Suppose for the sake of contradiction that $C$ is periodic, and let $q$, $L_{\mathcal C}$ and $h$ be defined as before. Let 
    \[\mathcal{G} = \{h(w): w\in L_{\mathcal{C}}\}\]
    and choose an integer $s$ so that $s > 2^{j-1}(j+k)$ and $\frac{s}{2^{j-1}} - 2j \geq N$, where $N \geq 2$ is an integer for which the conclusion of Lemma~\ref{lemma: iterate} holds for $\mathcal{B}_\mathcal{F}$. 

    First suppose that there exists $\sigma =h(w) \in \mathcal{G}$ which contains a monotone sub-interval of length at least $s$. Then Corollary~\ref{cor: no long monotone subintervals for hs} implies that $w$ contains a contiguous constant sub-word of length at least $\frac{s}{2^{j-1}} - 2j \geq N$, whence Lemma~\ref{lemma: too long directed paths with same label} implies that $C$ is aperiodic. Thus no $\sigma \in \mathcal{G}$ can contain a monotone sub-interval of length at least $s$. 

    For notational clarity, we assume that $\tau_j = \mathrm{inc}$. Define the particular permutation 
    \[\mu = 21 \oplus 12\dots s\oplus 21.\]
    Let $\Sigma' = \{c_1,\dots,c_{s+4}\}$ denote an alphabet that is disjoint from $\Sigma_{\star,(i',j')}$, and define a finite deterministic automaton $C'$ as follows. The initial and the only accepting state in $C'$ is $q$. The states of $C'$ are the states of $C$ together with $\{q_1',\dots,q_{s+3}'\}$, the alphabet is $\Sigma_{\star,(i',j')}\cup \Sigma'$, and the transition rules are given as follows. 
    \begin{enumerate}
        \item The transitions between states inside $C$ are inherited by $C'$. 
        \item Writing $q_0' = q_{s+4}'= q$, we have $\delta_{C'}(q_{i-1}',c_i) = q_{i}'$ for every $1 \leq i \leq s+4$. 
    \end{enumerate}
    That is, we add $s+3$ new states which form a closed cycle based at $q$ which corresponds to the word $c':= c_1\dots c_{s+4}$. Thus it is easy to see that every word $w$ accepted by $C'$ can be decomposed uniquely as 
    \begin{align}\label{eq: unique decomposition}
        w = u_1c'u_2c'\dots c'u_r,
    \end{align}
    where $u_1,\dots,u_r\in L_\mathcal{C}$, and they are allowed to be empty. Let $L_\mathcal{C'}$ denote the set of words accepted by $C'$, and define $h': L_\mathcal{C'} \rightarrow \mathfrak{S}$ by setting 
    \begin{align}\label{eq: definition of h'}
        h'(w)= h(u_1) \oplus \mu \oplus h(u_2) \oplus \mu \oplus \dots \oplus \mu \oplus h(u_r),    
    \end{align}
    where $h(\varepsilon)$ is defined to be the empty permutation. As in the proof of Lemma~\ref{lemma: aperiodicity or drops for plusminus}, the following essentially completes the proof of Lemma~\ref{lemma: periodic drops in A_d}.  

    \begin{claim}\label{claim: important counting claim for A_d}
        The following statements are true. 
        \begin{enumerate}
            \item $h'$ is a length-preserving map.
            \item $h'$ is an injection.
            \item $h'(w) \in \mathcal{F}$ for every $w \in L_{\mathcal{C'}}$. 
        \end{enumerate}
    \end{claim}
    \begin{proof}
        We note that $h'$ is clearly length-preserving by its construction. To prove that $h'$ is an injection, we note by Corollary~\ref{cor: no long monotone subintervals for hs} and by the choice of $s$, for every $w \in L_\mathcal{C'}$, any monotone sub-interval of length at least $s$ must be contained in a copy of $\mu$ in the expression of $h'(w)$ via~\eqref{eq: definition of h'}, and thus must correspond to the contiguous sub-words of $w$ that are of the form $c'$. In particular, given $\sigma = h'(w)$, we may recover each copy of $\mu$ in the decomposition~\eqref{eq: definition of h'} from monotone sub-intervals of length $s$. Deleting such copies of $\mu$ corresponds to deleting all the sub-words $c'$ from $w$. Thus injectivity of $h'$ follows from the injectivity of $h$. 

        Let $u_0 \in L_\mathcal{C}$ be chosen so that $|u_0| \geq s$. Thus $h(u_0)$ cannot be a monotone permutation, and hence it contains $21$ as a sub-permutation. Given $w \in L_\mathcal{C'}$, suppose $w$ decomposes as 
        \[w = u_1c'u_2c'\dots c'u_r,\]
        where $u_1,\dots,u_r \in L_{\mathcal{C}}$ or $u_i = \varepsilon$. Let $u' = (u_0)^{s+2}$, and define the particular word $w'$ as 
        \[w' = u_1u'u_2u' \dots u'u_r.\]
        Since each $u_0,u_1,\dots,u_r$ corresponds to a closed walk in $C$ based at $q$, we conclude that $w' \in L_\mathcal{C}$. Furthermore, we have 
        \[h(w') = h(u_1) \oplus h(u') \oplus h(u_2) \oplus \dots \oplus h(u') \oplus h(u_r). \]
        Since $u' = (u_0)^{s+2}$, we have 
        \[h(u') = h(u_0) \oplus \dots \oplus h(u_0),\]
        where the direct sum consists of $s+2$ copies of $h(u_0)$. We claim that $\mu$ is a sub-permutation of $h(u')$. Indeed, since $h(u_0)$ is not a monotone increasing permutation, it contains $21$ as a sub-permutation. Delete all but two points forming a sub-permutation $21$ from the first and last copies of $h(u_0)$, and all but one points from the middle $s$ copies of $h(u_0)$. Then the resulting permutation is 
        \[21 \oplus 1 \oplus \dots \oplus 1 \oplus 21 = 21 \oplus 12 \dots s \oplus 21 = \mu,\]
        which verifies the fact that $\mu$ is a sub-permutation of $h(u')$. Hence 
        \[h'(w) = h(u_1) \oplus \mu \oplus h(u_2) \oplus \dots \oplus \mu \oplus h(u_r)\]
        is a sub-permutation of $h(w')$. Since $w' \in L_\mathcal{C}$, Lemma~\ref{lemma: important properties of h} implies that $h(w') \in \mathcal{F}$. Since $\mathcal{F}$ is a closed under sub-patterns, we conclude that $h'(w) \in \mathcal{F}$. This completes the proof. 
    \end{proof}

    Equipped with Claim~\ref{claim: important counting claim for A_d}, the rest of the proof follows easily. Set $\lambda=\rho(A_C)$ and $\lambda'=\rho(A_{C'})$. Since $C'$ can be easily checked to be strongly connected, it follows that the matrix $A_{C'}$ is irreducible. Let $B$ denote the matrix obtained by adding $|C'| - |C|$ rows and columns consisting of $0$'s to the matrix $A_C$. Then $A_{C'}$ and $B$ satisfy the conditions  of Lemma~\ref{lemma: first spectral drops lemma}, and hence we must have $\lambda'>\lambda$. Since $C$ is dominant, we have 
    \[\lambda = \rho(A_C) = \rho(\mathcal{A}_{\mathcal{F},\#}) = \limsup_{n\rightarrow \infty} |\mathcal{F}_n^\#|^{1/n}.\]
    
    Since Claim~\ref{claim: important counting claim for A_d} implies that we have $h'(w) \in \mathcal{F}$ for every $w \in L_{\mathcal{C'}}$, we conclude that 
    \[
        (A_{C'}^n)_{q,q}\leq |\mathcal F_n|\leq |\mathcal F_n^\#|,
    \]
    and thus 
    \[\lambda' = \limsup_{n \rightarrow \infty}((A_{C'}^n)_{q,q})^{1/n} \leq \limsup_{n \rightarrow \infty}|\mathcal F_n^\#|^{1/n} = \lambda.\]
    This contradicts the fact that $\lambda' > \lambda$, and completes the proof.
\end{proof}
\subsection{Passing to $\mathcal{F}$} \label{sec:ungridded-aperiodicity}
Combining Lemma~\ref{lemma: aperiodicity or drops for plusminus} and Lemma~\ref{lemma: periodic drops in A_d}, we conclude that every dominant strongly connected component in $\mathcal{A}_{\mathcal{F},\#}$ is aperiodic. Let $\mathcal{B}_\#$ denote the finite deterministic automaton satisfying the conditions of Lemma~\ref{lemma: iterate} and which corresponds to the first-order sentence $\Delta_M$. Then it is straightforward to check that $\mathcal{A}_\mathcal{F}:= \mathcal{A}_\mathcal{F,\#} \times \mathcal{B}_\#$ accepts precisely the words in $L_{\mathcal{F}}$. Our goal is to prove the following result. 

\begin{thm}\label{thm: automata main theorem}
    Let $C$ be a dominant strongly connected component in $\mathcal{A}_\mathcal{F}$. Then $C$ is aperiodic. 
\end{thm}
We use the following simple pegs and holes -type of counting argument due to Vatter~\cite{vatter2011small}. 
\begin{lem}[Proposition 2.1 in~\cite{vatter2011small}] \label{lemma: gridded of by poly factor}
Let $M$ be a suitable matrix with $\ell$ columns and $m$ rows, and let $\mathcal{F} \subseteq \mathrm{Geom}(M)_{d,k}$. We have $|\mathcal{F}_n| \leq |\mathcal{F}^\#_n| \leq \binom{n+l}{l}\cdot \binom{n+m}{m}\cdot |\mathcal{F}_n|$. 
\end{lem}
Indeed, the statement follows easily by considering the number of ways to choose the least point in each row and column.  
\begin{proof}[Proof of Theorem~\ref{thm: automata main theorem}]
Write $A$ for the transition matrix of $\mathcal{A}_{\mathcal{F},\#}$ once the non-productive states have been removed, and write $A'$ for the transition matrix of $\mathcal{A}_\mathcal{F}$ once the non-productive states have been removed. Write $\mu = \rho(A)$ and $\mu' = \rho(A')$, and recall that we have $\mu = \limsup_{n \rightarrow \infty} |\mathcal{F}_n^\#|^{1/n}$ and $\mu' = \limsup_{n \rightarrow \infty} |\mathcal{F}_n|^{1/n}$. As Lemma~\ref{lemma: gridded of by poly factor} implies that 
\begin{align*}
    \limsup_{n\rightarrow \infty} |\mathcal{F}_n|^{1/n} &\leq \limsup_{n \rightarrow \infty} |\mathcal{F}_n^\#|^{1/n} \leq \limsup_{n\rightarrow \infty} \left(\binom{n+l}{l}\cdot \binom{n+m}{m}\cdot |\mathcal{F}_n| \right)^{1/n} = \limsup_{n\rightarrow \infty} |\mathcal{F}_n|^{1/n},
\end{align*}
we conclude that 
\[\mu' = \limsup_{n\rightarrow \infty} |\mathcal{F}_n|^{1/n} = \limsup_{n\rightarrow \infty} |\mathcal{F}_n^\#|^{1/n} = \mu.\]

Let $C$ denote a dominant component in $\mathcal{A}_\mathcal{F}$, and let $D$ denote the strongly connected component in $\mathcal{A}_{\mathcal{F},\#}$ containing the projection of $C$. Then Lemma~\ref{lemma: loose radius or maintain aperiodicity} implies that $\rho(A'_C) \leq \rho(A_D)$, and hence
\[\mu = \rho(A'_C) \leq \rho(A_D) \leq \rho(A) = \mu.\]
Thus we must have $\rho(A'_C) = \rho(A_D) = \mu$, and hence $D$ is a dominant component in $\mathcal{A}_{\mathcal{F},\#}$. Thus $D$ is aperiodic either by Lemma~\ref{lemma: aperiodicity or drops for plusminus} or by Lemma~\ref{lemma: periodic drops in A_d}. Hence Lemma~\ref{lemma: loose radius or maintain aperiodicity} implies that $C$ must be aperiodic as well, which completes the proof. 
\end{proof}
\section{Proof of main results}\label{section: Final section}
For convenience, we restate Theorem~\ref{theorem: eeroklas}. 
\begin{thm*}
Let $\mathcal{F}$ be a small permutation class. 
    \begin{enumerate}
        \item Then there exists a regular language $L_{\mathcal{F}}$ and a length-preserving bijection
        \[f_{\mathcal{F}}: L_{\mathcal{F}} \rightarrow \mathcal{F}\]
        satisfying the following property: for every first-order logic sentence $\Psi$ on $\mathcal{F}$ there exists a first-order logic  sentence $\Phi$ on $L_{\mathcal{F}}$ so that 
        \[\{w \in L_\mathcal{F}: w \models \Phi\} = \{w \in L_\mathcal{F}: f_\mathcal{F}(w) \models \Psi\}.\]
        The same conclusion also holds for sentences in monadic second-order logic.
        \item{(sketched)} The bijection $f_\mathcal{F}$ is in fact a monadic second-order isomorphism: that is, for every sentence $\Phi$ in monadic second-order logic on $L_\mathcal{F}$ there exists a corresponding sentence in monadic second-order logic on $\mathcal{F}$ for which the conclusion above holds. In addition, it is also possible to construct a regular language $L_\mathcal{F}'$ and a length-preserving bijection $f'_\mathcal{F}: L_\mathcal{F}' \rightarrow \mathcal{F}$ which is a first-order isomorphism. 
        \item For the last three assertions, suppose that $\mathcal{F}$ is not finite. There exist a positive real number $\alpha$, a real number $\beta \geq 1$, and a non-negative integer $c$ so that 
        \[\vert \mathcal{F}_n\vert = (\alpha+o(1))\cdot \beta ^n\cdot n^c.\]
        \item Given a monadic second-order sentence $\Psi$ on $\mathcal{F}$, write \[\mathcal{F}^\Psi_n := \{\sigma \in \mathcal{F}_n: \sigma \models \Psi\}.\]
        Then the sets $\mathcal{F}^\Psi_n$ have a rational generating function.
        \item For every first-order $\mathrm{TOTO}$ sentence $\Phi$ the limit  $\lim_{n\rightarrow \infty}\Pr(\pi \models \Phi)$ exists, where $\pi$  is taken uniformly at random from $\mathcal{F}_n$. 
    \end{enumerate}
\end{thm*}
Note that we omit the proof of the second part, whose proof will be sketched in Section~\ref{sec: conclusion}.
\begin{proof} [Proof of Theorem~\ref{theorem: eeroklas}]
    Let $d$ and $k$ be positive integers and let $M$ be a multiplicative $\star$-isolated forest matrix chosen so that $\mathcal{F} \subseteq \mathrm{Geom}(M)_{d,k}$. Note that such $M$ exists by Proposition~\ref{prop: small implies contained inside Geom (M) dk}. Since $\mathcal{F}$ is a small permutation class, Proposition~\ref{prop: small classes FOS classifiable} implies that there exists a first-order sentence $\hat{\upsilon}_{\mathcal{F}}$ so that $\sigma \models \hat{\upsilon}_\mathcal{F}$ if and only if $\sigma \in \mathcal{F}$. 

    Let $\chi_{\mathcal{F}}$ be the first-order sentence in $\Sigma_M^*$ obtained by applying Theorem~\ref{thm: FO correspondence for M} to the sentence $\hat{\upsilon}_{\mathcal{F}}$, and set $L_{\mathcal{F}} = \{w \in \Sigma_M^*: w \models \Delta_M \land \chi_{\mathcal{F}}\}$. Thus Theorem~\ref{thm: closure for regular languages} implies that $L_{\mathcal{F}}$ is a regular language, and Theorem~\ref{thm: FO correspondence for M} implies that the restriction $f_\mathcal{F}$ of $f_M$ on $L_{\mathcal{F}}$ is both a first-order and a monadic second-order homomorphism. Finally, since the restriction of $f_M$ on $\{w\in \Sigma_M^*: w\models \Delta_M\}$ induces a bijection whose image is $\mathrm{Geom}(M)_{d,k}$, it follows that $f_\mathcal{F}$ induces a bijection from $L_{\mathcal{F}}$ to $\mathcal{F}$. This completes the proof of the first part. 
    
     Let $\mathcal{A'}_\mathcal{F}$ denote the automaton obtained by removing the unproductive states from the automaton $\mathcal{A}_\mathcal{F}$ constructed in Theorem~\ref{thm: automata main theorem}. Thus by Theorem~\ref{thm: automata main theorem} all the dominant strongly connected components of $\mathcal{A}_\mathcal{F}$ are aperiodic, and at least one such component must exist as $\mathcal{F}$ is not finite. Hence Lemma~\ref{lemma: primitive dominant asymptotics} implies that there exists positive constant $c$, non-negative integer $k$ and $\lambda \geq 1$ so that 
    \[a_n:=|\mathcal{F}_n| = (c+o(1))\cdot \lambda^n \cdot n^k\]
    as $n\rightarrow \infty$. This completes the proof of the third part. 
    
    Let $\Psi$ be a monadic second-order sentence on $\mathcal{F}$, and let $\Phi$ denote the first-order sentence on words for which we have 
    \[\{w\in L_{\mathcal{F}}:w\models \Phi\}=\{w \in L_\mathcal{F}:f_{\mathcal{F}}(w) \models \Psi\}.\]
    Define $L_\mathcal{F}^\Phi := \{w \in L_{\mathcal{F}}: w\models \Phi\}$. Then $L_\mathcal{F}^\Phi$ is a monadic second-order restriction of a regular language, and thus itself a regular language by Theorem~\ref{thm: closure for regular languages}.  Hence the sets $(L_\mathcal{F}^\Phi)_n$ has a rational generating function by Theorem~\ref{thm: reg languages have RGF}, and thus $\mathcal{F}^\Psi_n$ also have rational generating function, which completes the proof of the fourth part. 
    
    Let $\Psi$ be a first-order sentence on permutations, and again we write $\Phi$ for the corresponding first-order sentence on words. Let $\mathcal{B}_\Phi$ denote the automaton chosen to satisfy the conclusion of Lemma~\ref{lemma: iterate}, and set $\mathcal{A}_\Psi := \mathcal{A}_\mathcal{F} \times B_\Phi$. Then $\mathcal{A}_\Psi$ accepts precisely those $w \in L_{\mathcal{F}}$ with $w \models \Phi$. 

    Define $b_n = |\mathcal{F}_n^\Psi|$. If $b_n = 0$ for sufficiently large $n$, then the limit is clearly $0$. Otherwise, let $\mathcal{A}_\Psi'$ denote the automaton obtained from $\mathcal{A}_\Psi$ by removing all the unproductive states. Then $\mathcal{A}_\Psi'$ accepts the same words as $\mathcal{A}_\Psi$, and since it accepts infinitely many words, it contains at least one directed cycle. 
    
    Let $\mu$ denote the largest spectral radius of a strongly connected component in $\mathcal{A}_\Psi'$, and let $r+1$ denote the largest number of strongly connected components with spectral radius $\mu$ along any walk from the initial state to an accepting state. Then standard properties of regular languages imply that $|\mathcal{F}_n^\Psi| \leq  C\cdot n^{r}\mu^n$ for some constant $C$. 

    If $\mu < \lambda$, then 
    \[\limsup_{n\rightarrow  \infty} \mathbb{P}\left(\boldsymbol{\sigma_n} \in \mathcal{F}_n: \boldsymbol{\sigma_n} \models \Psi\right) = \limsup_{n \rightarrow \infty} \frac{b_n}{a_n} \leq \lim_{n\rightarrow \infty} \frac{C\cdot n^r \cdot \mu^n}{c\cdot n^k \cdot \lambda^n} = 0,\] 
    and thus in particular the limit exists. If $\mu \geq \lambda$, let $C$ be a strongly connected component in $\mathcal{A}_\Psi'$ with spectral radius $\mu$, and let $D$ be the strongly connected component in $\mathcal{A'}_{\mathcal{F}}$ containing its projection. Write $A$ and $B$ for the transition matrices of $\mathcal{A'}_\mathcal{F}$ and $\mathcal{A}_\Psi'$ respectively. By Lemma~\ref{lemma: loose radius or maintain aperiodicity} we have $\rho(B_C) \leq \rho(A_D)$, and thus 
    \[\lambda \leq \mu = \rho(B_C) \leq \rho(A_D) \leq \lambda.\]
    Thus we must have $\rho(B_C) = \rho(A_D) = \lambda$, and hence $D$ must be a dominant component, and hence $D$ is aperiodic by Theorem~\ref{thm: automata main theorem}. Since $\rho(B_C) = \rho(A_D)$,  Lemma~\ref{lemma: loose radius or maintain aperiodicity} in turn implies that $C$ must be aperiodic. 
    
    Thus every dominant component in $\mathcal{A}_\Psi$ must be aperiodic, and thus Lemma~\ref{lemma: primitive dominant asymptotics} implies that there exits $c'>0$ for which $b_n =(c'+o(1))\cdot n^r \cdot \lambda^n$ as $n \rightarrow \infty$. Thus 
    \[\lim_{n\rightarrow  \infty} \mathbb{P}\left(\boldsymbol{\sigma_n} \in \mathcal{F}_n: \boldsymbol{\sigma_n} \models \Psi\right) = \lim_{n \rightarrow \infty} \frac{b_n}{a_n} = \lim_{n \rightarrow \infty} \frac{c'\cdot n^r}{c \cdot n^k}.\] 
    If $r = k$, the limit exists and equals $\frac{c'}{c}$. If $r < k$, the limit exists and equals $0$. Since $b_n \leq a_n$ for every $n$, we cannot have $r > k$. This completes the proof of the final part. 
\end{proof}
Note that as a corollary of the first statement we may deduce that small permutation classes have rational generating functions, which was first proved by Albert, Ru{\v{s}}kuc, and Vatter~\cite{albert2015inflations}.

\subsection{Monadic second-order limits}\label{msl}
By using the fact that small permutation classes and their monadic second-order restrictions are in bijection with a regular language, and hence can be described by a finite automaton with an integer transition matrix, we can prove the modular convergence law mentioned in the introduction. 
\begin{thm*} [Restatement of Theorem~\ref{thm: Monadic modular convergence}]
    Let $\mathcal F$ be an infinite small permutation class,  and let $\boldsymbol\sigma_n$ denote a permutation chosen uniformly at random from $\mathcal F_n$. For every $\mathrm{MSO}$-$\mathrm{TOTO}$ sentence $\Psi$ there exists an integer $\ell\geq1$ such that for every $0\leq r<\ell$, the limit 
    \[\lim_{n\rightarrow \infty} \mathbb{P}\left( \mathbf{\sigma_{n\ell+r}}\models \Psi\right)\] 
    exists. 
\end{thm*}
Here $\ell$ can be taken as the least common multiple of the periods of the  dominant strongly connected components of the transition graphs for the automata for $L_{\mathcal{F}}$ and $L_{\mathcal{F}}^\Psi$.  We omit the full proof which follows by analyzing the poles of the  respective generating functions for $L_{\mathcal{F}}$ and $L_{\mathcal{F}}^\Psi$ of least absolute value. 

Furthermore, the modular behaviour cannot be excluded for all sentences. For every pair of positive integers $\ell$ and $1 \leq a \leq \ell$, we  have a monadic second-order sentence $\Psi_{\ell,a}$ so that $\sigma \models \Psi_{\ell,a}$ if and only if $|\sigma| \equiv a \text{ (mod }\ell\text{)}$. To construct such sentence, we simply define variable sets $X_1,\dots,X_\ell$, ensure that the first point belongs to $X_1$, the last point belongs to $X_a$, that every point belongs to exactly one of these sets, and that for each $1 \leq j \leq \ell$, if $x \in X_j$ and $x$ is not the last point, then the successor of $x$ belongs to $X_{j+1}$ (with the convention that $X_{\ell+1} = X_1$). We omit the full details of the construction. 

\section{The inverse mapping and a first-order version}\label{sec: conclusion}

By Theorem~\ref{theorem: eeroklas} the bijection $f_{\mathcal{F}}$ translates both first-order and monadic second-order sentences on the permutations into first-order and monadic second-order sentences on regular languages, respectively. To conclude the paper, we give a sketch of a proof that monadic second-order sentences on $L_{\mathcal{F}}$ are translated to monadic second-order sentences on $\mathcal{F}$ under $f_\mathcal{F}$.

\begin{thm}\label{thm: full monadic isomorphism}
    The function $f_{\mathcal{F}}$ constructed in the proof of Theorem~\ref{theorem: eeroklas} is a monadic second-order isomorphism.
\end{thm}
\begin{proof}[Proof sketch]
    For clarity we write $\alpha_w : = \alpha_{w,M}$. The atomic formulas are of the form $w_i = a$, $i < j$, $i = j$ and $i \in A$, where $a$ is a letter, $i$ and $j$ are first-order variables, and $A$ is a second-order variable. Again, the last two formulas are simple to translate into formulas on permutations.

    For the atomic formulas of the form $w_i =a$, we may use Lemma~\ref{lemma: left-greedy drawing, which has its extremal points in A} to write down a condition for the corresponding $p:=\alpha_w(i)$ which ensures that $p$ belongs to the cell in the left-greedy representation $\sigma_\#$ of $\sigma:=f_M(w)$ described by $(i',j'):=\rho_6(a)$. If $a \in \Sigma_{M,\pm}$, this translation on its own suffices. If $a \in \Sigma_{\star,(i',j')}$, we set $b:=\rho_5(a)$ and write down the appropriate condition inductively.

    The base case $d = 1$ is easy to handle, we omit the details. If $\rho_2(b) = \tau \in \mathcal{Q}_k$, we may use the first part of Lemma~\ref{lemma: handling Q_ks for bijection} to write down a condition which ensures that $p$ belongs to the $\rho_3(b)$-th block of $(\sigma_\#)_{(i',j')}$, and proceed by induction within this block. 
    
    If $\rho_2(b) \in \{\mathrm{inc},\mathrm{dec}\}$, we write $u_{i',j'}$ for the sub-word of $w$ consisting of letters whose $\rho_6$-projection equals $(i',j')$, and suppose that $\rho_5(u_{i',j'})$ decomposes properly as $u_1\dots u_s$. By using Lemmas ~\ref{lemma: albert left greedy} and~\ref{lemmma: oplus} it is easy to construct a second-order variable $B$ which consists of the left-most points in the left-greedy decomposition of the corresponding permutation. We may then apply the inductive hypothesis inside this block. Note that at this step we have to be slightly careful with points whose $\rho_4$ image is $\{L\}$, $\{R\}$ or $\{L,R\}$, but we may again construct a first-order sentence by induction to handle this technicality.  We remark that this is the step which does not carry over to the first-order logic: the construction of the second-order variable $B$ is required even if we apply the construction to a first-order sentence. 
    
    For atomic formulas $i < j$, we write $p = \alpha_w(i)$ and $q = \alpha_w(j)$. Note that other parts of the monadic second-order sentence must impose some conditions on the $\rho_6$-images of $w_i$ and $w_j$, which correspond to the cells of the left-greedy decomposition of $f_M(w)$ containing $p$ and $q$ respectively. Since the extremal points of the left-greedy decomposition are first-order definable, the translation is trivial if the other conditions force $p$ and $q$ to be contained in different cells. If $p$ and $q$ belong to same cell indexed by $\pm 1$-entry, we may use $\delta_1,\dots,\delta_l$ to translate the condition $i<j$ to a condition comparing $p$ and $q$ with respect to the order $<_1$. For $\star$-entries, we may simply construct such a formula by induction on $d$, we omit the details as the technique is similar to the one used above. 
\end{proof}

We remark that for the encoding $f_\mathcal{F}$, the analogous statement does not hold for first-order sentences. Indeed, set $\mathcal{F} = \widetilde{\mathcal O}_6^{[2]}$ and consider the permutations $\sigma_m = 12\dots m [21,21,\dots ,21]$. We claim that there exists a first-order sentence $\Psi$ on $L_\mathcal{F}$ for which $w \models \Psi$ if and only if $f_d(w) = \sigma_{2m}$ for some positive integer $m$. This certainly cannot correspond to a first-order sentence $\Phi$ on $\mathcal{F}$, as for any $\Phi$, a standard Ehrenfeucht--Fra\"iss\'e -game argument implies that if $\sigma_{2m}\models \Phi$ for all $m$, then for all sufficiently large $m'$ we have $\sigma_{m'}\models \Phi$. 

The fact that such a sentence $\Psi$ exists can be proved as follows. Note that the left-greedy decomposition of $\sigma_{2m}$ in $\widetilde{\mathcal O}_6^{[2]}$ is $12\dots m[2143, \dots,2143]$, and thus the corresponding word $w_{2m}$ decomposes as $u_1\dots u_m$ with $\rho_1(u_i) = a_{\tau,1}a_{\tau,2}a_{\tau,3}a_{\tau,4}$, where $\tau = 2143$. In contrast, the corresponding word $w_{2m+1}$ for $\sigma_{2m+1}$ decomposes as $u_1\dots u_s$ with $u_1,\dots,u_{s-1}$ given as above, and $u_s = b_1b_2$ with $\rho_1(b_1b_2) = a_\mathrm{dec}a_\mathrm{dec}$. It is not difficult to construct a first-order sentence $\Psi$ on $L_{\mathcal{F}}$ for which $w \models \Psi$ if and only if $w = w_{2m}$ for some positive integer $m$, we omit the details.
\subsection{A first-order isomorphism}
We complete the discussion by sketching a different construction of an encoding $\Sigma_d'^* \rightarrow \widetilde{\mathcal O}_k^{[d]}\cup\{\star\}$ which allows us to construct a regular language $K_\mathcal{F}$ and a first-order isomorphism $f'_\mathcal{F}:K_\mathcal{F} \rightarrow \mathcal{F}$ for any small permutation class $\mathcal{F}$. The idea is to use the canonical substitution decomposition instead of the left-greedy decomposition.

Call an interval $I$ of a permutation \emph{strong} if every interval meeting $I$ either contains $I$, or is contained in $I$. Starting with a permutation $\sigma$, we repeatedly decompose each non-singleton strong interval of $\sigma$ into its maximal proper strong subintervals, until only intervals of length $1$ remain. This gives a unique decomposition tree, in which every profile is either monotone, or has no proper non-singleton interval. In particular, if the profile is monotone, the blocks correspond to the sum-indecomposable or skew-indecomposable components. We call this the \emph{canonical substitution decomposition} of $\sigma$. 

If $\sigma \in \widetilde{\mathcal O}_k^{[d]}$, it is not too hard to check that the height of this tree is bounded in terms of $d$ and $k$. Furthermore, the non-monotone profiles are precisely the simple permutations in $\mathcal{O}_k$, which is a finite set. Thus, we may use an analogous encoding for substitution trees: if the profile at a given depth is a simple permutation $\tau$, we use letters labelled by $\tau$ and the index of the block in which the corresponding point is contained, as is the case in our current construction. For monotone profile, we label the letters with its direction, and the first and last letters corresponding to each sum-indecomposable or skew-indecomposable block  with $L$ and $R$ respectively, as in our current construction. The critical difference from our construction is that the canonical blocks are in fact first-order definable: informally, when zooming into the blocks in the middle of the inflation,  we can detect the blocks in the canonical decomposition purely from the local before. This is not the case for the left-greedy decomposition, as for the left-greedy decomposition, all the previous points may affect the decomposition into blocks. This allows us to avoid the construction of the second-order set $B$ in the construction above. 

In particular, we conclude the following result. 
\begin{thm}
    Let $\mathcal{F}$ be a small permutation class. Then there exist a regular language $L_{\mathcal{F}}$ and a bijection $f_\mathcal{F}:L_\mathcal{F}\rightarrow \mathcal{F}$ for which $f_{\mathcal{F}}$ is both a first-order and a monadic second-order isomorphism. 
\end{thm}
\section*{Statement of AI use} 
The results of this paper were proved during 2024 and 2025, and most of the manuscript was written during the early part of 2026, and the main results were published before as an extended abstract~\cite{eeroklasconference}. In particular, all mathematical ideas and proofs were solely obtained by the authors. During the last phases of the write-up process, GPT-6 Astra was used for proofreading purposes, and to assist generating all the figures in the manuscript. 

\bibliographystyle{plain}
\bibliography{permutationsbiblio}

@article{ourotherpaper,
    author = {Falgas--Ravry, Victor and Markstr\"om, Klas and Skerman, Fiona and R\"aty, Eero},
    title = {On first-order limiting probabilities for small permutation classes},
    journal = {manuscript},
    year ={2026},
}

@book{hopcroft-motwani-ullman,
    author = {John Hopcroft and Rajeev Motwani and Jeffrey D. Ullman},
    title = {Introduction to Automata Theory, Languages, and Computation, 3rd edition},
    publisher = {Pearson},
    year ={2006},
}

@article{eeroklasconference,
    author = {Markstr\"om, Klas and R\"aty, Eero},
    title = {First-order convergence law for small
permutation classes},
    journal = {Proceedings of the 13th European Conference on Combinatorics, Graph Theory and Applications (Eurocomb’25)},
    pages = {882-887},
    year = {2025},
}

@article{braunfeld2025decidability,
    author = {Braunfeld, Samuel},
    title = {Decidability in geometric grid classes of permutations},
    journal = {Proceedings of the American Mathematical Society},
    volume = {153},
    number = {3},
    pages = {987--1000},
    year = {2025},
}

@article{bona2008combinatorics,
    title={Combinatorics of permutations},
    author={B{\'o}na, Mikl{\'o}s},
    journal={ACM SIGACT News},
    volume={39},
    number={4},
    pages={21--25},
    year={2008},
    publisher={ACM New York, NY, USA}
}

@article{albertrationality,
author = {Michael Albert and Robert Brignall and Nik Ru{\v{s}}kuc and Vincent Vatter},
title = {Rationality for subclasses of 321-avoiding permutations},
journal = {European Journal of Combinatorics},
volume = {78},
pages = {44-72},
year = {2019},
}

@incollection {Wolfgang,
    author = {Thomas, Wolfgang},
     title = {Languages, automata, and logic},
 booktitle = {Handbook of formal languages, {V}ol.\ 3},
     pages = {389--455},
 publisher = {Springer, Berlin},
      year = {1997},
}

@article{MinerPak,
author = {Sam Miner and Igor Pak},
title = {The shape of random pattern-avoiding permutations},
journal = {Advances in Applied Mathematics},
volume = {55},
pages = {86-130},
year = {2014},
}

@article {Bloom,
    author = {Bloom, Jonathan and Saracino, Dan},
     title = {On bijections for pattern-avoiding permutations},
   journal = {J. Combin. Theory Ser. A},
    volume = {116},
      year = {2009},
    number = {8},
     pages = {1271--1284},
}

@article {Claesson,
    author = {Claesson, Anders and Kitaev, Sergey},
     title = {Classification of bijections between 321- and 132-avoiding
              permutations},
   journal = {S\'em. Lothar. Combin.},
    volume = {60},
      year = {2008/09},
     pages = {Art. B60d, 30},
}

@article{Krat,
    author = {Krattenthaler, C.},
     title = {Permutations with restricted patterns and {D}yck paths},
   journal = {Adv. in Appl. Math.},
    volume = {27},
      year = {2001},
    number = {2-3},
     pages = {510--530},
}

@book {Rautenberg,
    author = {Rautenberg, Wolfgang},
     title = {A concise introduction to mathematical logic},
 publisher = {Springer, New York},
      year = {2010},
     pages = {xxii+319},
}

@article{albert2013geometric,
	title={Geometric grid classes of permutations},
	author={Albert, Michael and Atkinson, Michael and Bouvel, Mathilde and Ru{\v{s}}kuc, Nik and Vatter, Vincent},
	journal={Transactions of the American Mathematical Society},
	volume={365},
	number={11},
	pages={5859--5881},
	year={2013}
}

@article{albert2020two,
	title={Two first-order logics of permutations},
	author={Albert, Michael and Bouvel, Mathilde and F{\'e}ray, Valentin},
	journal={Journal of Combinatorial Theory, Series A},
	volume={171},
	pages={105158},
	year={2020},
	publisher={Elsevier}
}

@article{albert2024logical,
	title={A logical limit law for $231$-avoiding permutations},
	author={Albert, Michael and Bouvel, Mathilde and F{\'e}ray, Valentin and Noy, Marc},
	journal={Discrete Mathematics \& Theoretical Computer Science},
	volume={26},
	number={Special issues},
	year={2024},
	publisher={Episciences. org}
}

@article{albert2015inflations,
	title={Inflations of geometric grid classes of permutations},
	author={Albert, Michael H and Ru{\v{s}}kuc, Nik and Vatter, Vincent},
	journal={Israel Journal of Mathematics},
	volume={205},
	number={1},
	pages={73--108},
	year={2015},
	publisher={Springer}
}

@article{alon2000number,
	title={On the number of permutations avoiding a given pattern},
	author={Alon, Noga and Friedgut, Ehud},
	journal={Journal of Combinatorial Theory, Series A},
	volume={89},
	number={1},
	pages={133--140},
	year={2000}
}

@article{arratia1999stanley,
	title={On the {S}tanley-{W}ilf conjecture for the number of permutations avoiding a given pattern},
	author={Arratia, Richard},
	journal={The Electronic Journal of Combinatorics},
	pages={N1--N1},
	year={1999}
}

@article{braunfeld2021logical,
	title={Logical limit laws for layered permutations and related structures},
	author={Braunfeld, Samuel and Kukla, Matthew},
	journal={Enum.
	Combin. Appl.},
	volume={2},
	number={4},
	pages={Paper Number S4PP2},
	year={2022}
}

@inproceedings{foy1990probabilities,
	title={Probabilities of sentences about two linear orderings},
	author={Foy, John and Woods, Alan R.},
	booktitle={Feasible Mathematics: A Mathematical Sciences Institute Workshop, Ithaca, New York, June 1989},
	pages={181--193},
	year={1990},
	organization={Springer}
}

@article{pantonevatter,
    author = {Pantone, Jay and Vatter, Vincent},
    title = {Growth rates of permutation classes: Categorization up to the uncountability threshold},
    journal = {Israel Journal of Mathematics},
    year = {2020},
    volume = {236},
    pages = {1--43}, 
    publisher = {Springer}
}

@article{kaiser2002growth,
	title={On growth rates of closed permutation classes},
	author={Kaiser, Tom{\'a}{\v{s}} and Klazar, Martin},
	journal={The Electronic Journal of Combinatorics},
	pages={R10--R10},
	year={2002}
}

@article{polygrowth,
    author = {Albert, Michael and Atkinson, Michael and Brignall, Robert},
    title = {Permutation Classes of Polynomial Growth},
    journal = {Annals of Combinatorics},
    volume = {11}, 
    pages = {249 -- 264},
    year = {2007}
}

@article{kitaev2010survey,
	title={A survey on partially ordered patterns},
	author={Kitaev, Sergey},
	journal={Permutation Patterns},
	volume={376},
	pages={115--135},
	year={2010},
	publisher={Cambridge University Press}
}

@article{huczynska2006grid,
	title={Grid classes and the Fibonacci dichotomy for restricted permutations.},
	author={Huczynska, Sophie and Vatter, Vincent},
	journal={The Electronic Journal of Combinatorics},
	volume={13},
	pages={Paper R54},
	year={2006}
}

@book{linton2010permutation,
	title={Permutation patterns},
	author={Linton, Steve and Ru{\v{s}}kuc, Nik and Vatter, Vincent},
	volume={376},
	year={2010},
	publisher={Cambridge University Press}
}

@inproceedings{klazar2000furedi,
	title={The {F\"u}redi--{H}ajnal conjecture implies the {S}tanley--{W}ilf conjecture},
	author={Klazar, Martin},
	booktitle={Formal Power Series and Algebraic Combinatorics: 12th International Conference, FPSAC’00, Moscow, Russia, June 2000, Proceedings},
	pages={250--255},
	year={2000},
	organization={Springer}
}

@book{knuth1968art,
	title={The Art of Computer Programming, Volume 1},
	author={Knuth, Donald E},
	year={1968},
	publisher={Addison-Wesley Professional}
}

@article{marcus2004excluded,
	title={Excluded permutation matrices and the {S}tanley--{W}ilf conjecture},
	author={Marcus, Adam and Tardos, G{\'a}bor},
	journal={Journal of Combinatorial Theory, Series A},
	volume={107},
	number={1},
	pages={153--160},
	year={2004},
}

@article{steingrimsson2013some,
	title={Some open problems on permutation patterns.},
	author={Steingr{\'\i}msson, Einar},
	journal={Surveys in combinatorics},
	volume={409},
	pages={239--263},
	year={2013}
}

@article{vatter2011small,
	title={Small permutation classes},
	author={Vatter, Vincent},
	journal={Proceedings of the London Mathematical Society},
	volume={103},
	number={5},
	pages={879--921},
	year={2011},
	publisher={Wiley Online Library}
}

@article{vatter2010permutation,
	title={Permutation classes of every growth rate above 2.48188},
	author={Vatter, Vincent},
	journal={Mathematika},
	volume={56},
	number={1},
	pages={182--192},
	year={2010},
	publisher={London Mathematical Society}
}

@article{vatter2015permutation,
	title={Permutation classes},
	author={Vatter, Vincent},
	journal={Handbook of enumerative combinatorics},
	pages={754--833},
	year={2015},
	publisher={CRC Press Boca Raton}
}

@article{ozdemir2023first,
	title={First-order convergence for $321 $-avoiding permutations},
	author={{\"O}zdemir, Alperen},
	journal={arXiv preprint arXiv:2312.01749},
	year={2023}
}

@book{Straubing,
	title={Finite Automata, Formal Logic, and Circuit Complexity},
	author={Howard Straubing},
	volume={},
	year={1994},
	publisher={Birkh\"auser}
}

@book{FS,
	title={Analytic Combinatorics},
	author={Philippe Flajolet and Robert Sedgewick},
	volume={},
	year={2009},
	publisher={Cambridge University Press}
}
\section*{Appendix}
\label{appendix: first blocks in left greedy decompositions}

The purpose of this appendix is to provide the proof of the following technical tool which is required in~\cite{ourotherpaper}. For positive integers $k$ and $B$, write
\[
    \widetilde{\mathcal O}_{k,B}
    :=\{\tau\in\widetilde{\mathcal O}_k:|\tau|\leq B\}.
\]
When $B>k$, this is the union of $\mathcal O_k$ and the monotone permutations of length at most $B$. 

\begin{proposition}\label{prop: appendix first blocks definable}
    Let $B,d$ and $k$ be positive integers with $d \geq 2$. Then the following statements hold.
    \begin{enumerate}
        \item For every non-empty $\tau\in\widetilde{\mathcal O}_{k,B}$, there exists a first-order formula
        \[
            \psi_\tau(x_1,\dots,x_{|\tau|},y_1,\dots,y_{|\tau|})
        \]
        for which $\sigma\models\psi_\tau(a_1,\dots,a_{|\tau|},b_1,\dots,b_{|\tau|})$ if and only if the left-greedy decomposition of $\sigma\in\widetilde{\mathcal O}_k^{[d]}$ is
        \[
            \sigma=\tau[\theta_1,\dots,\theta_{|\tau|}],
        \]
        and each $\theta_i$ consists precisely of the points in $[a_i,b_i]_1$.
        \item For each $s\in\{\mathrm{inc},\mathrm{dec}\}$, there exists a first-order formula
        \[
            \psi_s(x_1,\dots,x_B,y_1,\dots,y_B)
        \]
        for which $\sigma\models\psi_s(a_1,\dots,a_B,b_1,\dots,b_B)$ if and only if the left-greedy decomposition of $\sigma\in\widetilde{\mathcal O}_k^{[d]}$ is
        \[
            \sigma=\tau'[\theta_1,\dots,\theta_{|\tau'|}],
        \]
        where $\tau'$ is monotone in direction $s$, $|\tau'|\geq B+1$, and $\theta_i$ consists precisely of the points in $[a_i,b_i]_1$ for every $1\leq i\leq B$.
    \end{enumerate}
\end{proposition}
We remark that the first statement follows easily from Lemma~\ref{lemma: handling Q_ks for bijection}. To prove the second statement, we start with the following preliminary lemma. 
\begin{lem}\label{lemma: key inflation lemma}
    Fix $\alpha \in \{\mathrm{inc},\mathrm{dec}\}$ and let $s$ be a positive integer. There exists a first-order formula $\psi_{\alpha,s}(x_{i,u,v})$ with $i \in [s],u,v\in \{1,2\}$ for which $\sigma \models \psi_{\alpha,s}(a_{i,u,v})$ if and only if $\sigma = \tau_{\alpha,s+1}[\sigma_1,\dots,\sigma_{s+1}]$, where $\tau_{\alpha,s+1}$ is the monotone permutation of direction $\alpha$ and length $s+1$, $\sigma_i$ consists of points in $[a_{i,1,1},a_{i,1,2}]_1\times [a_{i,2,1},a_{i,2,2}]_2$ for each $1 \leq i \leq s$, each $a_{i,u,v}$ belongs to $\sigma_i$, and $\sigma_i \in \widetilde{\mathcal O}_k^{[d-1]}$ for every $1 \leq i \leq $s. 
\end{lem}
We remark that we do not impose any restriction on $\sigma_{s+1}$: it does not need to belong to $\widetilde{\mathcal O}_k^{[d-1]}$, nor needs to be non-empty. The construction and the proof are very similar to the proof of Lemma~\ref{lemma: handling Q_ks for bijection}. 
\begin{proof}
    Let $\upsilon_{d-1,k}(x_1,x_2,y_1,y_2)$ denote the first-order formula constructed in the statement of Proposition~\ref{prop: small classes FOS classifiable} with $\mathcal{F} = \widetilde{\mathcal{O}}_k^{[d-1]}$. 
    When $\alpha = \mathrm{inc}$, we claim that one can take $\psi_{\mathrm{inc},s}(x_{i,u,v})$ to be the formula 
    \begin{align*}
         &\left(\forall z  \left(z \leq_1 x_{s,1,2} \implies \bigvee_{i=1}^{s} \left(z \in [x_{i,1,1},x_{i,1,2}]_1 \times [x_{i,2,1},x_{i,2,2}]_2\right)\right)\right) \\
         & \land \bigwedge_{i,i'=1}^{s}  \bigwedge_{u,u'=1}^{2} \bigwedge_{v,v'=1}^{2} \left(i < i' \implies \left(x_{i,u,v} <_1 x_{i',u',v'} \land x_{i,u,v} <_2 x_{i',u',v'})\right)\right) \\
         & \land \bigwedge_{i=1}^{s} \upsilon_{d-1,k}(x_{i,1,1},x_{i,2,1},x_{i,1,2},x_{i,2,2})
        \\ &\land\bigwedge_{i=1}^{s}\bigwedge_{u,v=1}^{2}
        \left(x_{i,u,v}\in
        [x_{i,1,1},x_{i,1,2}]_1\times
        [x_{i,2,1},x_{i,2,2}]_2\right)
        \\ & \land \left(\forall z_1 \forall z_2: \left(z_1 \leq_1 x_{s,1,2} \land x_{s,1,2} <_1 z_2\right) \implies z_1 <_2 z_2\right)
    \end{align*}
    Let $a_{i,u,v}$ be chosen so that $\sigma \models \psi_{\mathrm{inc},s}(a_{i,u,v})$, and for $1 \leq i \leq s$ , write $\sigma_i$ for the sub-permutation induced by the points in $[a_{i,1,1},a_{i,1,2}]_1 \times [a_{i,2,1},a_{i,2,2}]_2$. We also write $\sigma_{s+1}$ for the remaining points. Then the first and last lines imply that $\sigma = \varphi \oplus \sigma_{s+1}$, where $\varphi$ consists of the points in $\sigma_{1},\dots,\sigma_{s}$. Analogously as in the proof of Lemma~\ref{lemma: handling Q_ks for bijection}, we conclude that the remaining lines imply that $\varphi = \sigma_1 \oplus \dots \oplus \sigma_s$ and that for all $1 \leq i \leq s$ we have $\sigma_i \in \widetilde{\mathcal{O}}_k^{[d-1]}$. 

    Analogously, we set $\psi_{\mathrm{dec},s}(x_{i,u,v})$ to be the formula 
    \begin{align*}
         &\left(\forall z  \left(z \leq_1 x_{s,1,2} \implies \bigvee_{i=1}^{s} \left(z \in [x_{i,1,1},x_{i,1,2}]_1 \times [x_{i,2,1},x_{i,2,2}]_2\right)\right)\right) \\
         & \land \bigwedge_{i,i'=1}^{s}  \bigwedge_{u,u'=1}^{2} \bigwedge_{v,v'=1}^{2} \left(i < i' \implies \left(x_{i,u,v} <_1 x_{i',u',v'} \land x_{i',u',v'} <_2 x_{i,u,v})\right)\right) \\
         & \land \bigwedge_{i=1}^{s} \upsilon_{d-1,k}(x_{i,1,1},x_{i,2,1},x_{i,1,2},x_{i,2,2})
        \\ &\land\bigwedge_{i=1}^{s}\bigwedge_{u,v=1}^{2}
        \left(x_{i,u,v}\in
        [x_{i,1,1},x_{i,1,2}]_1\times
        [x_{i,2,1},x_{i,2,2}]_2\right)
        \\ & \land \left(\forall z_1 \forall z_2: \left(z_1 \leq_1 x_{s,1,2} \land x_{s,1,2} <_1 z_2\right) \implies z_2 <_2 z_1\right),
    \end{align*}
    and we note that checking the details follows similarly. 
\end{proof}
As a consequence, we obtain the following lemma
\begin{lem}\label{lemma: final step}
    Let $\alpha \in \{\mathrm{inc},\mathrm{dec}\}$ and let $s$ be a positive integer. Then there exists a first-order formulas $\psi'_{\alpha,s}(x_{i,u,v})$ with $i \in [s], u,v\in \{1,2\}$ so that for $\sigma \in \widetilde{\mathcal{O}}_{k}^{[d]}$ we have  $\sigma \models \psi'_{\alpha,s}(a_{i,u,v})$ if and only if the left-greedy decomposition of $\sigma$ is given as $\tau_r[\theta_1,\dots,\theta_r]$ for some $r \geq s+1$, where $\tau_r$ is the monotone permutation of length $r$ and direction $\alpha$, and for each $1\leq i\leq s$ the permutation $\theta_i$ consists of the points in $[a_{i,1,1},a_{i,1,2}]_1 \times [a_{i,2,1},a_{i,2,2}]_2$, and each $a_{i,u,v}$ belongs to $\theta_i$. 
\end{lem}
\begin{proof}
    Consider the first-order formula $\psi'_{\mathrm{inc},s}(x_{i,u,v})$ with $i \in [s], u,v\in \{1,2\}$ given by 
    \begin{align*}
        &\bigwedge_{\tau \in \mathcal{Q}_k\cup(\widetilde{\mathcal O}_{k,s}} \left(\forall_{i \in [|\tau|]} x_i,y_i: \neg \psi_{\tau}(x_1,\dots,x_{|\tau|},y_1,\dots,y_{|\tau|})\right)
        \\& \land \neg \left(\exists z\exists u\forall x\forall y:(z <_1 u)\land \left((x \leq_1 z\land z<_1y)\implies y<_2x\right)\right)
        \\& \land \psi_{\mathrm{inc},s}(x_{i,u,v}) \land \left(\forall_{i\in [s],u,v\in \{1,2\}} y_{i,u,v}: \psi_{\mathrm{inc},s}(y_{i,u,v})\implies \bigwedge_{i=1}^s \left(y_{i,2,2} \leq_2x_{i,2,2} \right)\right).
    \end{align*}
    We check that $\psi'_{\mathrm{inc},s}$ satisfies the required properties. Suppose that $\sigma \models \psi'_{\mathrm{inc},s}(a_{i,u,v})$, and let $\sigma = \tau[\theta_1,\dots,\theta_r]$ denote the left-greedy decomposition of $\sigma$. The first line excludes every non-monotone permutation and every monotone permutation of length at most $s$ as the profile $\tau$ of $\sigma$. Thus $\tau$ must be a monotone permutation of length at least $s+1$. 

    Recall that Proposition~2.3 in~\cite{albert2015inflations} implies that for any $\sigma \in \widetilde{\mathcal{O}}_{k}^{[d]}$, the set $T(\sigma)$ of possible profiles $\tau'$  of decompositions of $\sigma$ as $\sigma = \tau'[\sigma_1,\dots,\sigma_\ell]$ with $\sigma_i \in \widetilde{\mathcal{O}}_{k}^{[d-1]}$ has a unique minimum element under sub-pattern containment. 

    The second line implies that it is impossible to write down $\sigma$ as a skew-sum of two non-empty permutations. Thus, by the observation above, this condition is equivalent to the profile $\tau$ of $\sigma$ being a monotone increasing permutation, and hence a monotone increasing permutation of length at least $s+1$. 

    For each $1 \leq i\leq s$ and $u,v\in \{1,2\}$, let $b_{i,u,v}\in \theta_i$ be chosen so that $\theta_i$ consists of the points in $[b_{i,1,1},b_{i,1,2}]_1 \times [b_{i,2,1},b_{i,2,2}]_2$. Let $\sigma_1,\dots,\sigma_s$ be defined as in the proof of Lemma~\ref{lemma: key inflation lemma}.
    
    Since $\sigma \models \psi'_{\mathrm{inc},s}$, the last line implies that for all $1 \leq j \leq s$ we have $\sum_{i=1}^j|\theta_i| \leq \sum_{i=1}^j |\sigma_i|$. If the inequality is strict for some $j$, pick the least $j$ for which it is strict. Writing $\theta_i' = \theta_i$ for $i < j$, $\theta_j' = \sigma_j$ and  $\theta_i' = \theta_i \setminus \sigma_j$ for $i \geq j+1$, we have $\sigma = \theta_1' \oplus \dots \oplus \theta_r'$, omitting the empty blocks. This gives a decomposition of $\sigma$ as a direct sum of fewer blocks, or with the same number of blocks and larger $j$-th block, thus contradicting the fact that $\theta_1 \oplus \dots \oplus \theta_r$ is the left-greedy decomposition of $\sigma$.

    Thus the equality $\sum_{i=1}^j|\theta_i| =\sum_{i=1}^j |\sigma_i|$ must hold for every $1 \leq j \leq s$. Hence we conclude that $|\theta_i| = |\sigma_i|$ and hence $\sigma_i = \theta_i$ holds for each $1 \leq i\leq s$. Hence the $i$-th block $\theta_i$ in the left-greedy decomposition consists of points $[a_{i,1,1},a_{i,1,2}]_1 \times [a_{i,2,1},a_{i,2,2}]_2$. The proof of the converse direction follows similarly, we omit the details. 

    The construction of $\psi'_{\mathrm{dec},s}(x_{i,u,v})$ follows analogously, we omit the details. 
\end{proof}
We are now ready to deduce the proof of Proposition~\ref{prop: appendix first blocks definable}. 
\begin{proof}[Proof of Proposition~\ref{prop: appendix first blocks definable}]
    The first assertion follows by considering the first-order formula $\exists_{i \in [|\tau|],j\in \{1,2\}} x_{i,2,j}: \psi_{\tau,<_L}(x_{i,u,v})$ constructed in Lemma~\ref{lemma: handling Q_ks for bijection}, and renaming the variables as $x_i :=x_{i,1,1}$ and $y_i := x_{i,1,2}$. The second assertion follows analogously for each $\alpha\in\{\mathrm{inc},\mathrm{dec}\}$ by considering the first-order formula $\exists_{i \in [B], j\in \{1,2\}} x_{i,2,j} \psi'_{\alpha,B}(x_{i,u,v})$ constructed in Lemma~\ref{lemma: final step}, and renaming the variables as $x_i = x_{i,1,1}$ and $y_i = x_{i,1,2}$. This completes the proof of Proposition~\ref{prop: appendix first blocks definable}.
\end{proof}


\end{document}